\pdfoutput=1
\documentclass[11pt,reqno]{amsart}

\usepackage[utf8]{inputenc}
\usepackage{lmodern}
\usepackage[T1]{fontenc}
\usepackage{amsmath,amssymb,amsthm,mathrsfs}
\usepackage[margin=1.1in]{geometry}
\usepackage{enumerate}
\usepackage{graphicx}
\usepackage[hidelinks]{hyperref}

\theoremstyle{plain}
\newtheorem{theorem}{Theorem}[section]
\newtheorem{proposition}[theorem]{Proposition}
\newtheorem{lemma}[theorem]{Lemma}
\newtheorem{corollary}[theorem]{Corollary}
\newtheorem{conjecture}[theorem]{Conjecture}

\theoremstyle{definition}
\newtheorem{definition}[theorem]{Definition}
\newtheorem{hypothesis}{Hypothesis}

\theoremstyle{remark}
\newtheorem{remark}[theorem]{Remark}

\numberwithin{equation}{section}

\begin{document}

\title[Persistence of the Wiener sausage]{Persistence of the Wiener Sausage:
Sampling Stability and Limit Theorems for Betti-Curve Functionals
of Drifted Planar Brownian Motion}

\author{Tristan Guillaume}
\address{CY Cergy Paris Universit\'e, Laboratoire Thema,
33 boulevard du port, F-95011 Cergy-Pontoise Cedex, France}
\email{tristan.guillaume@cyu.fr}


\begin{abstract}
We study the persistent homology of the offset filtration generated by the
range of a planar Brownian motion with constant nonzero drift --- the family
of Wiener sausages of increasing radius --- and prove what are, to our
knowledge, the first law of large numbers and central limit theorem for
smoothed Betti-curve functionals of it. The main result is the law of large
numbers: for every bounded Borel
weight \(\psi\) supported in a compact radius window
\(\lbrack r_{0},r_{1}\rbrack \subset (0,\infty)\), the smoothed Betti-curve
functional \(\Phi_{\psi}(T) = \int_{r_{0}}^{r_{1}}\beta_{1}(T,r)\,\psi(r)\,dr\),
where \(\beta_{1}(T,r)\) counts the holes of the radius-\(r\) sausage at time
\(T\), satisfies \(\Phi_{\psi}(T)/T \rightarrow \rho_{\psi}\) almost surely and
in every \(L^{p}\) for a deterministic constant \(\rho_{\psi}\). This yields a finite
positive Betti-curve intensity measure on the radius axis, which we identify
as absolutely continuous through a fixed-radius law of large numbers along
regeneration times, valid at almost every radius. We also prove Gaussian
fluctuations at scale \(\sqrt{T}\), for one weight and jointly for finitely
many, with a Green--Kubo variance over the regeneration cycles. The proof
combines a renewal structure along the drift direction, obtained by the method
of Sznitman and Zerner without invoking the strong Markov property at the cut
times, with a Mayer--Vietoris control of the non-additivity of topology under
concatenation and a deterministic bound on the integrated hole count of a
planar offset proved in a companion paper. A sampling theorem, bounding the
bottleneck distance between the diagram of the continuous filtration and that
of a time-sampled point cloud by the pathwise modulus of continuity, places
these objects in a data-analytic framework.
\end{abstract}

\maketitle
\section{Introduction}\label{sec:1}

Since the seminal paper of Spitzer \cite{ref23} and the foundational works
of Donsker and Varadhan \cite{ref10} and Le Gall \cite{ref17}, the Wiener
sausage has been a classical object of probability theory. Given a path
\(X = \left( X_{t} \right)_{t \geq 0}\) in \(\mathbb{R}^{d}\), its range
up to time \(T\),

\[K_{T}:= X\left( \lbrack 0,T\rbrack \right),\]

generates the offset filtration

\[\left( K_{T}^{(r)} \right)_{r \geq 0},\quad\quad K_{T}^{(r)}:= \{ x \in \mathbb{R}^{d}:\operatorname{dist}\left( x,K_{T} \right) \leq r\},\]

Wiener sausages have largely been studied through geometric functionals
such as volume, capacity, intersection structure, surface area,
curvature measures, and Euler characteristic. In contrast, persistent
homology treats the entire one-parameter family
\(r \mapsto K_{T}^{(r)}\) as a single multiscale object and records the
birth and death of topological features as the radius varies. For a
sausage filtration this viewpoint is especially natural: in the planar
case, degree-one persistence records the creation and filling of holes
in the thickened trace, thereby capturing geometric information that is
invisible to single-scale observables.

The aim of this paper is to initiate a limit theory for the persistent
homology of Wiener sausage filtrations. We focus on the planar drifted
Brownian case

\[X_{t} = x_{0} + \mu t + B_{t},\quad\quad\mu \in \mathbb{R}^{2} \smallsetminus \{ 0\},\]

and study the degree-one persistence diagram of the offset filtration of
\(K_{T}\). The restriction to dimension \(2\) is geometrically natural:
because the range of a continuous path is connected, the first
nontrivial persistent information in the sausage filtration is carried
by one-dimensional holes. The restriction to nonzero drift is
probabilistically structural: the drift creates a preferred direction
along which the path advances through fresh space, and this yields a
regeneration mechanism that makes a law of large numbers possible.
Without drift, the corresponding one-dimensional projection is recurrent
and the renewal structure used here collapses.

Our first contribution is a sampling theorem that places the persistence
of Wiener sausages in a data-analytic framework. For a continuous path
observed on a time grid, we prove that the bottleneck distance between
the persistence diagram of the continuous sausage filtration and that of
the sampled point cloud is bounded by the pathwise modulus of
continuity. The proof combines Hausdorff control of the sampled range,
interleaving of offset filtrations, and the stability theory of
persistence diagrams. Specializing the deterministic bound to Brownian motion
yields the natural almost-sure rate
\(O\left( \sqrt{\left| \pi_{n} \right|\log\left( 1/\left| \pi_{n} \right| \right)} \right)\)
in the mesh \(\left| \pi_{n} \right|\) of the time grid, while globally
Lipschitz diffusions give H\"older-type rates via Kolmogorov
continuity. We also record two extensions of direct practical relevance:
additive observation noise and polygonal interpolation.

Our second and main contribution is a large-time limit theorem for a
persistence-derived functional of the Wiener sausage in the drifted planar
case. For bounded Borel \(\psi\), of arbitrary sign, supported in a compact
radius window
\(\left\lbrack r_{0},r_{1} \right\rbrack \subset (0,\infty)\), we consider the
smoothed Betti-curve functional

\[\Phi_{\psi}(T):= \int_{r_{0}}^{r_{1}}\beta_{1}(T,r)\,\psi(r)\, dr,\]

where \(\beta_{1}(T,r)\) is the number of holes in the radius-\(r\) sausage at
time \(T\). Equivalently, whenever Fubini's theorem applies,
\(\Phi_{\psi}(T) = \int_{\Delta}\varphi_{\psi}(b,d)\,\mu_{T}^{(1)}(db\, dd)\)
with \(\varphi_{\psi}(b,d) = \int_{b}^{d}\psi(r)dr\) and \(\mu_{T}^{(1)}\) the
degree-one persistence counting measure of the sausage filtration. We prove that
there exists a deterministic constant \(\rho_{\psi}\) with

\[\frac{\Phi_{\psi}(T)}{T} \rightarrow \rho_{\psi}\quad\quad\text{almost surely and in }L^{1}.\]

Thus the smoothed alive-count mass carried by the radius window grows linearly
in time with a deterministic slope, and this yields a finite positive
Betti-curve intensity measure \(\lambda\) on the radius axis. The observable
is supported on a window in the radius parameter, not on a compact
birth--death rectangle: it sees every persistence interval alive somewhere in
\(\left\lbrack r_{0},r_{1} \right\rbrack\), including intervals born below
\(r_{0}\) or dying above \(r_{1}\). This test class is exactly the one
compatible with the Betti-curve representation and with the regeneration
argument, and it yields a genuine intensity theory at the level of
alive-count observables; but it is not measure-determining on the diagram
(Remark~\ref{rem:6-17}), seeing only the Betti-curve transform of the diagram
measure, so the results below are asymptotics for that transform and not a
complete asymptotic theory of persistence diagrams.

The proof mechanism differs substantially from existing limit-theorem
frameworks in topological data analysis. The first ingredient is the
Betti-curve representation just displayed, which allows us to work with
fixed-radius hole counts rather than individual persistence points, and which is
exactly what makes the regeneration argument possible. The second ingredient is a
regeneration structure built from the drift direction: projecting onto
\(\mu/ \lVert \mu \rVert\) produces a one-dimensional Brownian
motion with positive drift, whose ladder hits and bounded-backtracking
events generate regeneration times. The cut times are not stopping
times, so the independence of the resulting blocks is not a consequence
of the strong Markov property; it is obtained by the renewal argument of
Sznitman and Zerner \cite{SZ99}, which yields blocks that are independent, and
identically distributed from the second block onwards, the first being
delayed. The third ingredient is a deterministic Mayer--Vietoris
estimate which shows that the non-additivity of \(\beta_{1}\) under
concatenation is confined to a local interface term, together with a
deterministic bound on the integrated hole count of a planar offset, proved in
the companion paper \cite{G26g} and recalled in Section~\ref{sec:4}, which
controls that term through the diameters of two consecutive blocks. The resulting Boundary
Lemma shows that the expected interface correction per regeneration cut
is finite. With these pieces in place, the law of large numbers follows
from ergodic and renewal arguments at the cycle level.

Our third contribution is a fluctuation theorem. The same regeneration
structure, together with the moment bounds that the interface analysis
already provides, yields a central limit theorem for the smoothed
functional at scale \(\sqrt{T}\), and its finite-dimensional extension;
the limiting variance is the long-run variance of the centred cycle
rewards, in which only the lag-zero and lag-one covariances appear
because the \(1\)-dependence produced by the interface structure is
exact (Section~\ref{sec:7}).

Our fourth contribution is to delimit what the method does not reach. The
law of large numbers above concerns smoothed observables; at a single fixed
radius the corresponding statement requires the expected hole count of the
sausage to be finite at that radius, which we obtain for almost every radius
but not for every radius, and along the regeneration times rather than at
deterministic times. Section~\ref{sec:8} isolates what is missing in each
case --- for the every-radius question, a finite expected hole count of the
block sausages at the prescribed radius, the block-level analogue under the
drift of the fixed-time finiteness proved for the driftless sausage by Rataj,
Spodarev and Meschenmoser \cite{RSM09} --- and reduces a deterministic-time
version of that estimate to the fragmentation of near-tangent pairs of
sausage pieces, shows where the natural route to it fails, and states the
conjecture that the expected hole count is comparable to the expected
boundary length divided by the radius.

This paper sits at the intersection of three literatures. The first is
the general theory of persistent homology and its stability. Persistence
diagrams for tame or \(q\)-tame filtrations are now standard, beginning
with \cite{ref11} and developed further in \cite{ref9}, \cite{ref7}, \cite{ref12},
\cite{ref8}. The second is the probabilistic literature on topological
summaries of random geometric objects. A major line of work, surveyed by
Bobrowski and Kahle \cite{ref5}, studies homology and persistent homology
for random geometric complexes generated by i.i.d. or Poisson point
clouds. For persistent-homological observables and their extremes see Bobrowski,
Kahle, and Skraba \cite{ref4}; the structural precedent for the present
article is the law of large numbers for persistence diagrams of stationary
point processes of Hiraoka, Shirai, and Trinh \cite{ref13}; and strong laws,
central limit theorems and process-level limit theorems for Betti functionals
of random complexes are obtained in Yogeshwaran, Subag, and Adler
\cite{ref25}, Owada and Thomas \cite{ref20}, and Krebs and Hirsch
\cite{ref15}.

Closest to the present work in its treatment of dependence is the work of
Krebs \cite{Kre21} on limit theorems for persistent Betti numbers computed
from dependent data --- stationary time series and random fields satisfying
mixing conditions \cite{ref6}, sampled in growing windows --- and its sequel
\cite{ref15} on cylindrical networks. The difference is what forces a
different method. There, the input is a stationary field observed through an
increasing spatial window, dependence is quantified by mixing coefficients,
and the limit theory is driven by spatial ergodicity and the decoupling of
far-apart regions of the window. Here, the input is the trace of a single
continuous path: there is no stationary field to sample, no mixing hypothesis
is available, and distant portions of the trace may revisit one another. The
substitute for spatial ergodicity is temporal regeneration, supplied by the
drift rather than by a mixing assumption: the path advances through fresh
space along the drift direction, and the resulting cut times decompose the
trace into blocks that are independent and geometrically separated. The
present results are thus complementary to \cite{Kre21} rather than a special
case of them: the dependence is stronger, but it is generated by a mechanism
that can be exploited exactly.

The idea that finite sampled data can recover underlying topology with
probabilistic guarantees goes back at least to Niyogi, Smale, and Weinberger
\cite{ref19}. Our sampling theorem is of a different nature --- it concerns the
discretization error induced by time sampling of a continuous stochastic path,
not homology recovery for an underlying manifold --- but belongs to the same
program of topological inference from incomplete observations. The
algebraic-topological background used throughout (Nerve Theorem, \v{C}ech and
Vietoris--Rips constructions, Mayer--Vietoris arguments) may be found in
Edelsbrunner and Harer \cite{ref12}.

The third literature is the classical Wiener-sausage literature.
Foundational contributions include Spitzer \cite{ref23}, which initiated the
subject through its connections with capacity and Brownian motion,
Donsker and Varadhan \cite{ref10} on asymptotics for the Wiener sausage, and
Le Gall \cite{ref17,ref18} on planar Wiener sausages, multiple points, and
their connection with self-intersection local times. For a broad
synthesis of Brownian motion, obstacles, and Wiener-sausage-type
questions, see Sznitman \cite{ref24}. Closer in spirit to the geometric side
of the present paper are works on fixed-radius geometric characteristics
of the sausage: Rataj, Schmidt, and Spodarev \cite{RSS09} study its expected
surface area; Last \cite{ref16} studies mean curvature functions of Brownian
paths; and Rataj, Spodarev, and Meschenmoser \cite{RSM09} analyze
approximations and curvature measures. In the planar setting, Honzl
\cite{Ho14} studies connected components of the complement of a Wiener
sausage and derives an upper bound on its Euler characteristic ---
equivalently a lower bound on the number of holes --- in a regime where
the radius tends to zero coupled to the area of the components. These
works are directly relevant to the geometry underlying holes in planar
sausages, but they concern single-radius geometric observables rather
than the persistent homology of the entire offset filtration.

Baryshnikov \cite{ref2} connects Brownian motion and persistence in a
different sense, through the sublevel-set filtration of the path viewed as a
one-dimensional time series, which detects extrema of a real-valued function;
our object is the offset filtration of a planar trace, whose features are
spatial holes in the thickened range, so the filtration, the ambient space
and the resulting persistence are different. We are not aware of prior work on the persistence diagrams
of Wiener-sausage filtrations themselves, nor of a limit theorem of the type
proved here for a continuous stochastic path. The driftless planar sausage is
treated in \cite{G26c}, where the regeneration mechanism
used below is unavailable and the growth is governed by Brownian scaling rather
than by renewal; the contrast is discussed in Section~\ref{sec:9}.

The paper is organized as follows. Section~\ref{sec:2} introduces the offset
filtration of a compact set, its persistence diagram, the associated counting
measures and Betti curves, and the smoothed Betti-curve functionals that appear
in the main theorems. Section~\ref{sec:3} proves the sampling-stability theorem
for continuous paths and derives explicit rates for Brownian motion and
Lipschitz diffusions. Section~\ref{sec:4} collects the deterministic planar
geometry on which the proofs rest --- the Mayer--Vietoris toolkit for compact
planar sets, the integrated complexity bound that supplies all the
integrability used later, and the fact that no pointwise-in-radius analogue of
that bound exists --- all of which is proved in the companion paper
\cite{G26g}. Section~\ref{sec:5} constructs the regeneration structure for
planar Brownian motion with nonzero drift, establishes the independence of the
resulting blocks without appealing to the strong Markov property at the cut
times, and derives the exponential moment bounds used throughout.
Section~\ref{sec:6} combines these ingredients to prove the law of large numbers.
Section~\ref{sec:7} proves the central limit theorem and its finite-dimensional
extension. Section~\ref{sec:8} treats fixed-radius asymptotics: a law of large
numbers at almost every radius, and a discussion of the single geometric
estimate that would extend it to every radius.
Section~\ref{sec:9} collects concluding remarks and open problems.
\section{Setup and preliminaries}\label{sec:2}

We work in the birth--death plane, not in birth--lifetime coordinates, and we
parametrize all offset filtrations by the radius \(r \geq 0\), not by the
diameter \(2r\). All homology is singular homology with coefficients in a fixed field
\(\mathbb{F}\). For \(q = 0\) we use reduced homology, in order to avoid
carrying the single essential connected-component class; for \(q \geq 1\),
ordinary homology is used, and for \(q \geq 1\) the two agree.

From Section~\ref{sec:4} onwards a different convention is in force: the
deterministic planar geometry, imported there from the companion paper
\cite{G26g}, is carried out in \v{C}ech cohomology, and there \(\beta_{0}\)
denotes the unreduced number of connected components. The two conventions are
reconciled in Remark~\ref{rem:4-1} and Lemma~\ref{lem:4-1b}, which shows that
at every noncritical radius --- hence, by Fu's theorem \cite{Fu85}, at
Lebesgue-almost every radius --- the singular and \v{C}ech Betti numbers of an
offset coincide in every degree. Since every
identity in this paper connecting the persistence diagram to the hole count is
an absolutely convergent radius integral, agreement at almost every radius is
all that is used. We write \(\beta_{q}\) throughout for the common value at
such radii, and flag the reduced convention explicitly on the two occasions
where degree zero is at issue.

\subsection{Offsets of compact sets and Wiener sausage filtrations}\label{sec:2-1}

Let \(A \subset \mathbb{R}^{d}\) be nonempty and compact. For
\(r \geq 0\), its closed \(r\)-offset is

\[A^{(r)}:= \{ x \in \mathbb{R}^{d}:\operatorname{dist}(x,A) \leq r\},\]

where

\[\operatorname{dist}(x,A):= \inf_{a \in A} \lVert x - a \rVert .\]

The family

\[\mathcal{F}(A):= \left( A^{(r)} \right)_{r \geq 0}\]
is the offset filtration generated by \(A\).

If \(X = \left( X_{t} \right)_{t \geq 0}\) is a continuous
\(\mathbb{R}^{d}\)-valued path and \(T > 0\), we write

\[K_{T}:= X\left( \lbrack 0,T\rbrack \right)\]
for its range up to time \(T\). The corresponding Wiener sausage
filtration is

\[\mathcal{F}_{T}:= \mathcal{F}\left( K_{T} \right) = \left( K_{T}^{(r)} \right)_{r \geq 0}.\]

Equivalently,

\[K_{T}^{(r)} = \bigcup_{0 \leq t \leq T}\overline{B}\left( X_{t},r \right).\]

where \(\overline{B}(x,r)\) denotes a closed Euclidean ball of radius
\(r\) centered at \(x\)

When comparing two compact sets \(A,B \subset \mathbb{R}^{d}\), we use
the Hausdorff distance

\[d_{H}(A,B):= \max\left\{ \sup_{a \in A}\operatorname{dist}(a,B),\sup_{b \in B}\operatorname{dist}(b,A) \right\}.\]

A basic identity that will be used repeatedly is

\[\lVert d_{A} - d_{B} \rVert_{\infty} = d_{H}(A,B),\]

where \(d_{A}(x):= \operatorname{dist}(x,A)\) and \(d_{B}(x):= \operatorname{dist}(x,B)\).

For later use, we record one elementary geometric fact. If \(A\) is
connected, then each offset \(A^{(r)}\) is connected. In particular, for
the range of a continuous path in the plane, the interesting persistent
homology begins in degree \(1\), where holes in the thickened trace
appear and disappear as the radius varies.

\subsection{Persistence diagrams and their metrics}\label{sec:2-2}

Let \(\mathcal{X} = \left( X_{r} \right)_{r \geq 0}\) be an
increasing filtration of topological spaces. For \(0 \leq r \leq s\),
the inclusion \(X_{r} \hookrightarrow X_{s}\) induces a linear map

\[H_{q}\left( X_{r} \right) \rightarrow H_{q}\left( X_{s} \right).\]

The \(q\)-dimensional persistent homology module of \(\mathcal{X}\) is
the resulting functor in the parameter \(r\).

When this module is \(q\)-tame, its multiset of finite intervals is
encoded by a persistence diagram, denoted

\[\operatorname{Dgm}_{q}\left( \mathcal{X} \right).\]

In the present paper we apply this to offset filtrations and write, by
abuse of notation,

\[\operatorname{Dgm}_{q}(A):= \operatorname{Dgm}_{q}\left( \mathcal{F}(A) \right),\quad\quad\operatorname{Dgm}_{q}(T):= \operatorname{Dgm}_{q}\left( \mathcal{F}_{T} \right).\]

We regard persistence diagrams as locally finite multisets in

\[\Delta:= \{(b,d) \in \lbrack 0,\infty)^{2}:\ b < d\},\]

augmented by the diagonal

\[\partial\Delta:= \{(r,r):r \geq 0\}\]
for the purpose of matchings.

The metric used most often below is the bottleneck distance. If \(D\)
and \(E\) are persistence diagrams, their bottleneck distance is

\[d_{B}(D,E):= \inf_{\gamma}\sup_{x} \lVert x - \gamma(x) \rVert_{\infty},\]

where the infimum ranges over all partial matchings between \(D\) and
\(E\), with unmatched points allowed to be paired with the diagonal.

\subsection{q-tameness of offset filtrations}\label{sec:2-3}

The use of persistence diagrams for Wiener sausages rests on a standard
tameness fact for sublevel-set filtrations of proper functions.

\begin{proposition}\label{prop:2-1}
 Let \(A \subset \mathbb{R}^{d}\)
be compact. Then the offset filtration
\(\mathcal{F}(A) = \left( A^{(r)} \right)_{r \geq 0}\) is
\(q\)-tame for every \(q \geq 0\). Consequently, the
persistence diagram \(\operatorname{Dgm}_{q}(A)\) is well defined.
\end{proposition}

\begin{proof}
The distance function
\(d_{A}:\mathbb{R}^{d} \rightarrow \lbrack 0,\infty)\) is continuous,
\(1\)-Lipschitz, and proper because \(A\) is compact. Since

\[A^{(r)} = d_{A}^{- 1}(\left( - \infty,r\rbrack \right),\]

the filtration \(\mathcal{F}(A)\) is the sublevel-set filtration of a
proper continuous function on the triangulable space \(\mathbb{R}^{d}\).
Standard results in persistent homology imply that such filtrations are
\(q\)-tame (\cite{ref8}).

What \(q\)-tameness gives is local finiteness of the diagram away from
the diagonal: for every \(\eta > 0\) and every finite interval
\(\left\lbrack r_{0},r_{1} \right\rbrack \subset (0,\infty)\), only finitely
many diagram points, counted with multiplicity, lie in

\[\left\{ (b,d) \in \Delta:\ r_{0} \leq b < d \leq r_{1},\ d - b \geq \eta \right\}.\]

It does not give finiteness on the full rectangle
\[\Delta\left\lbrack r_{0},r_{1} \right\rbrack:= \left\{ (b,d) \in \Delta:\ r_{0} \leq b < d \leq r_{1} \right\},\]
which meets the diagonal and may contain infinitely many points accumulating
towards it. This is one reason we do not build the observable of
Section~\ref{sec:2-5} on a compact birth--death rectangle: its finiteness is
supplied instead by the integrated complexity bound of Theorem~\ref{thm:4-11},
which controls a radius integral and is insensitive to accumulation at the
diagonal.

A second elementary observation will also be useful. If
\(A \subset \mathbb{R}^{d}\) is compact, then \(A^{(r)}\) is
contractible for all sufficiently large \(r\). Indeed, if
\(r > \operatorname{diam}(A)\) and \(a_{0} \in A\), then
\(a_{0} \in \overline{B}(a,r)\) for every \(a \in A\), so \(A^{(r)}\) is
a union of convex sets all containing \(a_{0}\), hence is star-shaped.
Therefore all positive-dimensional persistence classes of
\(\mathcal{F}(A)\) have finite death times.
\end{proof}

\subsection{Stability with respect to Hausdorff perturbations}\label{sec:2-4}

The passage from a continuous path to a sampled point cloud relies on
the stability of persistence for sublevel-set filtrations. In the
present geometric setting, the statement takes a particularly simple
form.

\begin{theorem}[stability for offsets]\label{thm:2-2}
For compact sets \(A,B \subset \mathbb{R}^{d}\) and each
\(q \geq 0\),

\[d_{B}\left( \operatorname{Dgm}_{q}(A),\operatorname{Dgm}_{q}(B) \right) \leq d_{H}(A,B).\]
\end{theorem}

\begin{proof}
This is an immediate consequence of the general stability
theorem for persistence diagrams (\cite{ref7}) applied to the distance
functions \(d_{A}\) and \(d_{B}\), together with the identity

\(\lVert d_{A} - d_{B} \rVert_{\infty} = d_{H}(A,B).\)
\end{proof}

We shall use Theorem~\ref{thm:2-2} in Section~\ref{sec:3} with \(A\) equal to the continuous
range \(K_{T}\) and \(B\) equal to a discrete approximation of that
range. The theorem is the last step in the chain

\[\text{path regularity}\, \Rightarrow \,\text{Hausdorff control}\, \Rightarrow \,\text{barcode control}.\]

\subsection{Persistence counting measures, Betti curves, and smoothed functionals}\label{sec:2-5}

Let \(A \subset \mathbb{R}^{d}\) be compact, and fix \(q \geq 0\). The
\(q\)-th persistence counting measure of the offset filtration
\(\mathcal{F}(A)\) is the locally finite measure on \(\Delta\) defined
by

\[\mu_{A}^{(q)}:= \sum_{x \in \operatorname{Dgm}_{q}(A)}^{}m(x)\,\delta_{x},\]

where \(m(x)\) denotes the multiplicity of the diagram point \(x\). For
the Wiener sausage filtration of a path up to time \(T\), we write

\[\mu_{T}^{(q)}:= \mu_{K_{T}}^{(q)}.\]

The corresponding Betti curve is the function

\[\beta_{q}^{A}(r):= \operatorname{rank}H_{q}\left( A^{(r)};\mathbb{F} \right)\quad\quad(r \geq 0),\]
with the reduced convention in degree zero. By Lemma~\ref{lem:4-1b} this
coincides, at every noncritical radius and hence at Lebesgue-almost every
\(r > 0\), with the \v{C}ech Betti number used from Section~\ref{sec:4} onwards;
in degree one it is the number of bounded connected components of
\(\mathbb{R}^{2} \smallsetminus A^{(r)}\) when \(A \subset \mathbb{R}^{2}\).

In terms of the persistence counting measure, at every noncritical radius
\(r\) --- hence, in the planar case, by Fu's theorem, at Lebesgue-almost every
\(r > 0\) ---
\begin{equation}\label{eq:2-alive}
\beta_{q}^{A}(r) = \mu_{A}^{(q)}\left( \{(b,d) \in \Delta:\ b \leq r < d\} \right),
\end{equation}
so that at such radii \(\beta_{q}^{A}(r)\) counts the \(q\)-dimensional
persistence intervals alive at scale \(r\). The restriction is needed:
\(q\)-tameness alone does not make the undecorated diagram record classes that
are born and die at one and the same critical radius, so \eqref{eq:2-alive}
can fail at a critical value. At a noncritical radius it holds because the
offsets \(A^{(s)}\) for \(s\) in a neighbourhood of \(r\) free of critical
values are isotopic, by the isotopy lemma of critical-point theory for
distance functions \cite{Gr93,CCL09}, so the structure maps
\(H_{q}(A^{(r - \varepsilon)}) \rightarrow H_{q}(A^{(r)}) \rightarrow H_{q}(A^{(r + \varepsilon)})\)
are isomorphisms for small \(\varepsilon\): no class is born or dies at \(r\),
the intervals alive at \(r\) are exactly those alive throughout a
neighbourhood of \(r\), and their number is \(\operatorname{rank}H_{q}(A^{(r)})\).
Everything below uses \eqref{eq:2-alive} only under an integral over \(r\),
where a null set of radii is invisible.

In the planar case \(d = 2\), and for connected sets \(A\), the degree
\(q = 1\) Betti curve has a direct geometric interpretation:
\(\beta_{1}^{A}(r)\) is the number of bounded connected components of
\(\mathbb{R}^{2} \smallsetminus A^{(r)}\), that is, the number of holes
in the sausage at radius \(r\).

We now define the functionals that will enter our limit theorems. Throughout,
\(\psi:\lbrack 0,\infty) \rightarrow \mathbb{R}\) is bounded, Borel, and
supported in a compact radius window
\(\left\lbrack r_{0},r_{1} \right\rbrack \subset (0,\infty)\); no sign
restriction is imposed. The smoothed Betti-curve functional of a compact
\(A \subset \mathbb{R}^{2}\) is
\begin{equation}\label{eq:2-phi}
\Phi_{\psi}(A):= \int_{r_{0}}^{r_{1}}\beta_{1}^{A}(r)\,\psi(r)\, dr .
\end{equation}
We take \eqref{eq:2-phi} as the definition. It is well defined and
finite for every compact \(A\): by Theorem~\ref{thm:4-11},
\(\int_{r_{0}}^{r_{1}}\beta_{1}^{A}(r)dr \leq 4050\,(\operatorname{diam}A)^{4}r_{0}^{- 3} < \infty\),
so the integrand is finite for almost every \(r\) and the integral converges
absolutely, with
\(\left| \Phi_{\psi}(A) \right| \leq \lVert \psi \rVert_{\infty}4050\,(\operatorname{diam}A)^{4}r_{0}^{- 3}\).

It is useful to know when \eqref{eq:2-phi} may equivalently be read as an
integral against the persistence diagram. Let
\(\varphi:\Delta \rightarrow \mathbb{R}\) be bounded, Borel and compactly
supported, and set
\[\Phi_{\varphi}^{\mathrm{dgm}}(A):= \int_{\Delta}\varphi(b,d)\,\mu_{A}^{(1)}(db\, dd),\]
which is well defined because \(\mu_{A}^{(1)}\) is locally finite and
\(\varphi\) has compact support. The weight \(\psi\) induces
\begin{equation}\label{eq:2-phipsi}
\varphi_{\psi}(b,d):= \int_{b}^{d}\psi(r)\, dr,
\end{equation}
and Tonelli's theorem, applied to \(|\psi|\) and legitimate by the finiteness of
\(\int_{r_{0}}^{r_{1}}\beta_{1}^{A}(r)dr\) just recorded, together with
\eqref{eq:2-alive} at almost every \(r\), gives
\begin{equation}\label{eq:2-betarep}
\int_{\Delta}\varphi_{\psi}(b,d)\,\mu_{A}^{(1)}(db\, dd)\  = \ \int_{r_{0}}^{r_{1}}\beta_{1}^{A}(r)\psi(r)\, dr\  = \ \Phi_{\psi}(A).
\end{equation}
This Betti-curve representation replaces a problem about individual
persistence points by one about the family \(r \mapsto \beta_{1}^{A}(r)\),
and is what makes the regeneration method workable.

\begin{remark}[why the radius integral is the primary definition]\label{rem:2-3}
The function \(\varphi_{\psi}\) of \eqref{eq:2-phipsi} is not compactly
supported in the birth--death plane, even though \(\psi\) is compactly
supported on the radius axis: an interval born below \(r_{0}\) and dying above
\(r_{1}\) receives the full mass \(\int_{r_{0}}^{r_{1}}\psi(r)dr\). The
well-definedness of the left-hand side of \eqref{eq:2-betarep} therefore does
not follow from local finiteness of \(\mu_{A}^{(1)}\); it follows from the
integrability supplied by Theorem~\ref{thm:4-11}. For the same reason we
describe \(\Phi_{\psi}\) as a functional associated with a radius window,
never with a ``fixed birth--death window'': it sees every interval alive
somewhere in \(\left\lbrack r_{0},r_{1} \right\rbrack\), including intervals
outside the compact rectangle
\(\Delta\left\lbrack r_{0},r_{1} \right\rbrack = \left\{ (b,d) \in \Delta:\ r_{0} \leq b < d \leq r_{1} \right\}\).
The condition \(r_{0} > 0\) excludes the small-radius regime, where Brownian
roughness may produce infinitely many tiny features. As Remark~\ref{rem:6-17}
shows, the class \(\left\{ \varphi_{\psi} \right\}\) is not
measure-determining on \(\Delta\left\lbrack r_{0},r_{1} \right\rbrack\): it
sees only the Betti-curve transform of the diagram measure, and the limit
theorems of Sections~\ref{sec:6}--\ref{sec:8} are statements about that
transform.
\end{remark}
\section{Sampling stability for continuous paths}\label{sec:3}

In this section, we show that the persistence diagram of the Wiener
sausage filtration of a continuous path is stable under time
discretization, with an explicit error bound in terms of the pathwise
modulus of continuity. This gives a rigorous bridge from the continuous
object introduced in Section~\ref{sec:2} to the sampled point clouds used in
computation. The argument has three steps: Hausdorff control of the
sampled range, interleaving of offset filtrations, and stability of
persistence diagrams. We then specialize the resulting bound to Brownian
motion and to diffusions with globally Lipschitz coefficients, and
conclude with a few practical variants.

\subsection{Continuous and sampled offset filtrations}\label{sec:3-1}

Let \(X = \left( X_{t} \right)_{0 \leq t \leq T}\) be a continuous path
in \(\mathbb{R}^{d}\), with its range
\(K_{T}:= X\left( \lbrack 0,T\rbrack \right).\)

Fix a partition
\(\pi_{n} = \{ 0 = t_{0} < t_{1} < \cdots < t_{n} = T\}\) of
\(\lbrack 0,T\rbrack\), and write
\[\left| \pi_{n} \right|:= \max_{0 \leq i \leq n - 1}\left( t_{i + 1} - t_{i} \right)\]
for its mesh. The associated sampled point cloud is

\[P_{\pi_{n}}:= \{ X_{t_{0}},X_{t_{1}},\ldots,X_{t_{n}}\} \subset \mathbb{R}^{d}.\]

We compare the offset filtrations

\[\mathcal{F}\left( K_{T} \right) = \left( K_{T}^{(r)} \right)_{r \geq 0}\quad\quad\text{and}\quad\quad\mathcal{F}\left( P_{\pi_{n}} \right) = \left( P_{\pi_{n}}^{(r)} \right)_{r \geq 0}.\]

The relevant path-regularity quantity is the modulus of continuity

\[\omega_{X}(\delta;T):= \sup\{ \lVert X_{t} - X_{s} \rVert :0 \leq s,t \leq T,\ |t - s| \leq \delta\},\quad\quad\delta \geq 0.\]

where \(\delta\) is the time scale on which oscillations are measured.

Since \(X\) is continuous on a compact interval,
\(\omega_{X}(\delta;T) \downarrow 0\) as \(\delta \downarrow 0\).

The key observation is that the Hausdorff error of the discretized range
is exactly controlled by this modulus.

\begin{proposition}[Hausdorff control by the modulus of
continuity]\label{prop:3-1}
For every partition \(\pi_{n}\) of
\(\ \lbrack 0,T\rbrack\),

\[d_{H}\left( K_{T},P_{\pi_{n}} \right) \leq \omega_{X}\left( \left| \pi_{n} \right|;T \right).\]
\end{proposition}

\begin{proof}
Since \(P_{\pi_{n}} \subset K_{T}\), the reverse Hausdorff
term vanishes, so it is enough to bound

\[\sup_{x \in K_{T}}\operatorname{dist}\left( x,P_{\pi_{n}} \right).\]

Let \(x \in K_{T}\). Then \(x = X_{t}\) for some
\(t \in \lbrack 0,T\rbrack\). Choose \(i\) such that
\(t \in \left\lbrack t_{i},t_{i + 1} \right\rbrack\). Hence
\(\left| t - t_{i} \right| \leq \left| \pi_{n} \right|\), and therefore

\[\operatorname{dist}\left( x,P_{\pi_{n}} \right) \leq \lVert X_{t} - X_{t_{i}} \rVert \leq \omega_{X}\left( \left| \pi_{n} \right|;T \right).\]

Taking the supremum over \(x \in K_{T}\) proves the claim.
\end{proof}

This proposition is the geometric core of the discretization step:
between two sampling times, the continuous path can only move by at most
its modulus of continuity, hence the sampled cloud is Hausdorff-close to
the full range.

\subsection{Interleaving of offset filtrations}\label{sec:3-2}

The next step is deterministic and purely geometric.

\begin{lemma}[offset interleaving]\label{lem:3-2}
Let \(A,B \subset \mathbb{R}^{d}\) be nonempty compact sets
and suppose

\[d_{H}(A,B) \leq \varepsilon.\]

Then for every \(r \geq 0\),

\[A^{(r)} \subset B^{(r + \varepsilon)}\quad\quad\text{and}\quad\quad B^{(r)} \subset A^{(r + \varepsilon)}.\]
\end{lemma}

\begin{proof}
We prove the first inclusion; the second is symmetric. Let
\(x \in A^{(r)}\). Then there exists \(a \in A\) such that
\(\lVert x - a \rVert \leq r\). Since
\(d_{H}(A,B) \leq \varepsilon\), there exists \(b \in B\) with
\(\lVert a - b \rVert \leq \varepsilon\). Hence

\[\lVert x - b \rVert \leq \lVert x - a \rVert + \lVert a - b \rVert \leq r + \varepsilon,\]

so \(x \in B^{(r + \varepsilon)}\).
\end{proof}

\begin{corollary}\label{cor:3-3}
The filtrations \(\mathcal{F}\left( K_{T} \right)\) and
\(\mathcal{F}\left( P_{\pi_{n}} \right)\) are
\(\omega_{X}\left( \left| \pi_{n} \right|;T \right)\)-interleaved.
\end{corollary}

\begin{proof}
Apply Lemma~\ref{lem:3-2} with \(A = K_{T}\), \(B = P_{\pi_{n}}\),
and \(\varepsilon = \omega_{X}\left( \left| \pi_{n} \right|;T \right)\)
\end{proof}

\subsection{The sampling theorem}\label{sec:3-3}

Recall from Section~\ref{sec:2} that persistence diagrams of offset filtrations
are well defined for compact sets, and that bottleneck stability holds
with respect to Hausdorff perturbations. Combining Corollary~\ref{cor:3-3} with
that stability result yields the main theorem of this section.

\begin{theorem}[sampling stability for continuous paths]\label{thm:3-4}
Let \(X = \left( X_{t} \right)_{0 \leq t \leq T}\) be a
continuous path in \(\mathbb{R}^{d}\), let \(\pi_{n}\) be
a partition of \(\lbrack 0,T\rbrack\), and let
\(q \geq 0\). Then

\[d_{B}\left( \operatorname{Dgm}_{q}\left( K_{T} \right),\operatorname{Dgm}_{q}\left( P_{\pi_{n}} \right) \right) \leq d_{H}\left( K_{T},P_{\pi_{n}} \right) \leq \omega_{X}\left( \left| \pi_{n} \right|;T \right).\]
\end{theorem}

\begin{proof}
The second inequality is Proposition~\ref{prop:3-1}. The first
follows from the stability theorem for persistence diagrams of offset
filtrations applied to the compact sets \(K_{T}\) and \(P_{\pi_{n}}\).
\end{proof}

Thus the barcode error is bounded by a single pathwise quantity, which tends
to \(0\) with the mesh; the theorem is deterministic once the path is fixed,
and for random \(X\) all rates reduce to estimates on
\(\omega_{X}(\delta;T)\). For uniform sampling, \(\left| \pi_{n} \right| = T/n\)
and the bound reads \(\omega_{X}(T/n;T)\); the explicit rates below substitute
a modulus estimate for the process at hand.

\subsection{Brownian motion: the L\'evy-modulus rate}\label{sec:3-4}

We first specialize to Brownian motion. Let
\(B = \left( B_{t} \right)_{0 \leq t \leq T}\) be standard Brownian
motion in \(\mathbb{R}^{d}\). For each fixed \(T > 0\), L\'evy's modulus
of continuity implies that almost surely

\[\limsup_{\delta \downarrow 0}\frac{\omega_{B}(\delta;T)}{\sqrt{2\delta\log(1/\delta)}} = 1.\]

Consequently, along any sequence \(\delta_{n} \downarrow 0\),

\[\omega_{B}\left( \delta_{n};T \right) = O\left( \sqrt{\delta_{n}\log\left( 1/\delta_{n} \right)} \right)\quad\quad\text{a.s.}.\]

Substituting \(\delta_{n} = T/n\) into Theorem~\ref{thm:3-4} yields the following.

\begin{corollary}[Brownian sampling rate]\label{cor:3-5}
Let \(B\) be standard Brownian motion in
\(\mathbb{R}^{d}\), and let \(P_{\pi_{n}}\) be the uniform
sample cloud on the mesh \(T/n\). Then for every
\(q \geq 0\),

\[d_{B}\left( \operatorname{Dgm}_{q}\left( B\left( \lbrack 0,T\rbrack \right) \right),\operatorname{Dgm}_{q}\left( P_{\pi_{n}} \right) \right) = O\left( \sqrt{\frac{T}{n}\log n} \right) = O\left( \sqrt{\left| \pi_{n} \right|\log\left( 1/\left| \pi_{n} \right| \right)} \right)\quad\quad\text{a.s.}\]
\end{corollary}

\begin{proof}
Substitute \(\delta = T/n\) into Theorem~\ref{thm:3-4} and use the
L\'evy-modulus estimate above.
\end{proof}

This is the natural rate for Brownian sampling: the persistence
approximation error is of the same order as the largest Brownian
fluctuation missed between two successive observation times.

\subsection{Diffusions with globally Lipschitz coefficients}\label{sec:3-5}

We next consider a \(d\)-dimensional It\^o diffusion

\[dX_{t} = b\left( X_{t} \right)\, dt + \sigma\left( X_{t} \right)\, dW_{t},\quad\quad 0 \leq t \leq T,\]

where \(b:\mathbb{R}^{d} \rightarrow \mathbb{R}^{d}\) and
\(\sigma:\mathbb{R}^{d} \rightarrow \mathbb{R}^{d \times m}\) are
globally Lipschitz. Standard moment estimates imply that for each
\(p \geq 2\) there exists \(C_{p,T} < \infty\) such that

\[\mathbb{E}\lVert X_{t} - X_{s} \rVert^{p} \leq C_{p,T}|t - s|^{p/2},\quad\quad 0 \leq s,t \leq T.\]

For \(p > 2\) this is the hypothesis of the Kolmogorov--Chentsov theorem with
exponent \(1 + \varepsilon = p/2\), which yields a modification of \(X\) whose
paths are H\"older continuous of every order \(\alpha < 1/2 - 1/p\) on
\(\lbrack 0,T\rbrack\); since a strong solution under globally Lipschitz
coefficients already has continuous paths, and two continuous modifications
are indistinguishable, the conclusion holds for \(X\) itself. Letting
\(p \rightarrow \infty\), which is legitimate because the moment bound holds
for every \(p \geq 2\), covers every \(\alpha \in (0,1/2)\). The H\"older
constant so obtained is the random seminorm
\begin{equation}\label{eq:3-holder}
\Lambda_{\alpha,T}:= \sup_{0 \leq s \neq t \leq T}\frac{\lVert X_{t} - X_{s} \rVert}{|t - s|^{\alpha}},
\end{equation}
measurable because the supremum may be taken over rational pairs, and finite
almost surely; it is genuinely random, and no deterministic constant would do,
since already for Brownian motion the law of \eqref{eq:3-holder} has unbounded
support. The chaining proof gives more, namely
\(\mathbb{E}\left\lbrack \Lambda_{\alpha,T}^{\, p} \right\rbrack < \infty\) for the
exponent \(p\) used, which we shall not need. What matters below is that
\(\Lambda_{\alpha,T}\) does not depend on \(\delta\), so that
\[\omega_{X}(\delta;T) \leq \Lambda_{\alpha,T}\, \delta^{\alpha}\quad\quad\text{for every }\delta > 0,\quad\text{almost surely},\]
the null set being chosen once and for all. Plugging this into
Theorem~\ref{thm:3-4} gives the following.

\begin{corollary}[diffusion H\"older rates]\label{cor:3-6}
Let \(X\) be a diffusion with globally Lipschitz
coefficients. Then for every \(\alpha \in (0,1/2)\), every
\(q \geq 0\), and every sequence of partitions \(\pi_{n}\)
with \(\left| \pi_{n} \right| \rightarrow 0\),

\[d_{B}\left( \operatorname{Dgm}_{q}\left( K_{T} \right),\operatorname{Dgm}_{q}\left( P_{\pi_{n}} \right) \right) \leq \Lambda_{\alpha,T}\left| \pi_{n} \right|^{\alpha}\quad\quad\text{for every }n,\ \text{almost surely},\]
with \(\Lambda_{\alpha,T}\) the almost surely finite random H\"older seminorm
\eqref{eq:3-holder}, which depends on \(\alpha\) and \(T\) but not on \(n\).
\end{corollary}

\begin{proof}
Apply Theorem~\ref{thm:3-4} with the almost-sure bound
\(\omega_{X}\left( \left| \pi_{n} \right|;T \right) \leq \Lambda_{\alpha,T}\left| \pi_{n} \right|^{\alpha}\);
since the bound holds simultaneously for every \(\delta\) off a single null
set, so does the conclusion for the whole sequence \(\left( \pi_{n} \right)_{n \geq 1}\).
\end{proof}

This is the qualitative rate statement we shall need. Sharper
logarithmic refinements are available under more detailed increment
estimates, but the H\"older form is already enough to justify convergence
of the sampled diagrams to the continuous ones for the class of
diffusions considered here.

\subsection{Practical variants and computational remarks}\label{sec:3-6}

Two variants of Theorem~\ref{thm:3-4} are immediate. Suppose that instead of
the exact samples \(X_{t_{i}}\) we observe \(Y_{i} = X_{t_{i}} + \xi_{i}\),
\(i = 0,\ldots,n\), with
\(\eta_{\pi}:= \max_{0 \leq i \leq n} \lVert \xi_{i} \rVert < \infty\),
and let \(Q_{\pi_{n}}:= \{ Y_{0},\ldots,Y_{n}\}\) be the noisy sample cloud.
Alternatively, since for small radii a point cloud may be a poor visual
approximation of the sausage even when the Hausdorff distance is controlled,
let \(L_{\pi_{n}}\) be the image of the polygonal interpolation of the samples.

\begin{proposition}[observation noise; polygonal interpolation]\label{prop:3-7}
For every \(q \geq 0\),
\begin{align*}
d_{B}\left( \operatorname{Dgm}_{q}\left( K_{T} \right),\operatorname{Dgm}_{q}\left( Q_{\pi_{n}} \right) \right) &\leq \omega_{X}\left( \left| \pi_{n} \right|;T \right) + \eta_{\pi},\\
d_{B}\left( \operatorname{Dgm}_{q}\left( K_{T} \right),\operatorname{Dgm}_{q}\left( L_{\pi_{n}} \right) \right) &\leq \omega_{X}\left( \left| \pi_{n} \right|;T \right).
\end{align*}
\end{proposition}

\begin{proof}
By Theorem~\ref{thm:2-2} it suffices to bound the Hausdorff distances. Since
\(\lVert X_{t_{i}} - Y_{i} \rVert \leq \eta_{\pi}\) for each \(i\),
\(d_{H}\left( P_{\pi_{n}},Q_{\pi_{n}} \right) \leq \eta_{\pi}\), and the
triangle inequality with Proposition~\ref{prop:3-1} gives the first bound. For
the second, if \(x = X_{t} \in K_{T}\) with
\(t \in \left\lbrack t_{i},t_{i + 1} \right\rbrack\), the segment joining
\(X_{t_{i}}\) to \(X_{t_{i + 1}}\) lies in \(L_{\pi_{n}}\) and \(X_{t}\) is
within \(\omega_{X}\left( \left| \pi_{n} \right|;T \right)\) of it; conversely
a point \(y\) of that segment satisfies
\(\operatorname{dist}\left( y,K_{T} \right) \leq \lVert X_{t_{i + 1}} - X_{t_{i}} \rVert \leq \omega_{X}\left( \left| \pi_{n} \right|;T \right)\).
\end{proof}

Thus discretization and observation errors add at the level of bottleneck
distance, and polygonal interpolation has the same worst-case bound as the
point cloud; its advantage is geometric, not asymptotic.

\begin{remark}[\v{C}ech versus Vietoris--Rips persistence]\label{rem:3-9}
For a finite cloud \(P_{\pi_{n}}\) the offset filtration is the union-of-balls
filtration, whose homology agrees by the Nerve Theorem \cite{ref12} with that
of the \v{C}ech filtration on \(P_{\pi_{n}}\); \v{C}ech persistence is thus
the exact combinatorial model of the sampled sausage. When Vietoris--Rips
complexes are used instead, as is common in computation, there are two
approximation steps: the sampling error, controlled additively in the
filtration parameter by Theorem~\ref{thm:3-4}, and the \v{C}ech--Rips model
error, controlled multiplicatively by the standard interleaving, whose
constant depends on the scale convention.
\end{remark}
\section{Deterministic input from the companion paper}\label{sec:4}

The probabilistic arguments of Sections~\ref{sec:5}--\ref{sec:8} rest on a
small number of statements about arbitrary compact subsets of the plane, with
no probability in them. They are proved in the companion paper \cite{G26g},
and we collect them here in the form in which they are used, indicating the
numbering of \cite{G26g} in each case. Throughout this section \(A,B\) denote
nonempty compact subsets of \(\mathbb{R}^{2}\),
\(d_{A}(x) = \min\{\lVert x - a \rVert :a \in A\}\) the distance
function, \(A^{(r)} = \{ d_{A} \leq r\}\) the closed \(r\)-offset, and
\(D = \operatorname{diam}A\).

\subsection{Conventions}\label{sec:4-1}

\begin{remark}[cohomological conventions]\label{rem:4-1}
From here on, Betti numbers \(\beta_{i}\) of compact planar sets are ranks of
\v{C}ech (equivalently, Alexander--Spanier) cohomology with coefficients in a
fixed field, as in \cite[Section 2]{G26g}. Two classical facts are used
repeatedly and record the reason for this choice. First, for compact
\(X \subseteq \mathbb{R}^{2} \subset S^{2}\), Alexander duality gives
\({\check{H}}^{q}(X) \cong {\widetilde{H}}_{1 - q}\left( S^{2} \smallsetminus X \right)\);
taking \(q = 1\) identifies \(\beta_{1}(X)\) with the number of bounded
connected components of \(\mathbb{R}^{2} \smallsetminus X\), and taking
\(q = 2\) gives \({\check{H}}^{2}(X) = 0\), since
\(S^{2} \smallsetminus X \neq \varnothing\). Second, the Mayer--Vietoris
sequence is exact for pairs of closed subsets of a paracompact space in
\v{C}ech cohomology (see \cite[Chapter 6]{Sp81}); no niceness of the sets is
assumed anywhere. When \(X\) is compact,
\(\beta_{0}(X) = \dim{\check{H}}^{0}(X)\) is the unreduced number of connected
components, possibly infinite, and all inequalities in this framework are read
in \(\lbrack 0,\infty\rbrack\). This differs from the degree-zero convention of
Section~\ref{sec:2}, which is reduced; the two differ by one on a nonempty set,
and we never use them interchangeably. In degree one, where all the limit
theory takes place, no such discrepancy arises, by Lemma~\ref{lem:4-1b} below.
\end{remark}

\subsection{Critical values, and the bridge to singular homology}\label{sec:4-2}

Following \cite{Fu85}, a point \(x \in \mathbb{R}^{2} \smallsetminus A\) is
critical for \(d_{A}\) if \(x\) lies in the convex hull of its nearest points
in \(A\), and \(C(A) \subseteq (0,\infty)\) is the set of critical values, the
values \(d_{A}(x)\) at critical points \(x\). Fu's theorem
\cite[Theorem 4.1]{Fu85} states that \(C(A) \cup \{ 0\}\) is compact and has
vanishing one-half-dimensional Hausdorff measure; in particular \(C(A)\) is
Lebesgue-null.

\begin{lemma}[singular and \v{C}ech Betti numbers agree at noncritical radii; {\cite[Lemma 3.3]{G26g}}]\label{lem:4-1b}
Let \(r \in (0,\infty) \smallsetminus C(A)\). Then \(A^{(r)}\) is a compact
absolute neighbourhood retract, and for every \(q \geq 0\) and every field
\(\mathbb{F}\),
\[\operatorname{rank}H_{q}\left( A^{(r)};\mathbb{F} \right)\  = \ \dim{\check{H}}^{q}\left( A^{(r)};\mathbb{F} \right).\]
In particular the two conventions assign the same value to
\(\beta_{1}\left( A^{(r)} \right)\), that common value is the number of bounded
connected components of \(\mathbb{R}^{2} \smallsetminus A^{(r)}\), and by
Fu's theorem this holds for Lebesgue-almost every \(r > 0\).
\end{lemma}

Lemma~\ref{lem:4-1b} is not used inside any proof below, all of which are
written in \v{C}ech terms; its role is to make the persistence module, the
diagram and the Betti curve of Section~\ref{sec:2} the same objects as those
appearing in the estimates, so that \(\Phi_{\psi}\) may be read either way.
Because it is invoked only inside radius integrals, agreement at almost every
radius suffices.

\subsection{The integrated complexity bound}\label{sec:4-3}

\begin{theorem}[integrated complexity bound; {\cite[Theorem 1.1]{G26g}}]\label{thm:4-11}
For every compact \(A \subseteq \mathbb{R}^{2}\) and every \(r_{0} > 0\),
\[\int_{r_{0}}^{\infty}\beta_{1}\left( A^{(r)} \right)dr\, \leq \, 4050\,\frac{D^{4}}{r_{0}^{3}},\]
the integrand vanishing for \(r \geq D/\sqrt{3}\).
\end{theorem}

Theorem~\ref{thm:4-11} is the integrability backbone of everything that
follows: it supplies the well-definedness of \(\Phi_{\psi}\)
(Section~\ref{sec:2-5}), the moments of the block rewards and of the interface
terms (Sections~\ref{sec:6}--\ref{sec:7}), and the finiteness of
\(\beta_{1}(A^{(r)})\) at almost every radius under which the Mayer--Vietoris
arithmetic of Lemma~\ref{lem:6-2} below is carried out. The constant is not
optimized, and for the applications any polynomial dependence on \(D\)
suffices. One feature of the statement matters for the fixed-radius theory of
Section~\ref{sec:8}.

\begin{remark}[no pointwise bound]\label{rem:4-7}
Theorem~\ref{thm:4-11} bounds an integral of the hole count over radii, not
the hole count at a prescribed radius, and this is not an artefact of its
proof. By \cite[Theorem 1.2]{G26g}, for every \(\varepsilon \in (0,1/10\rbrack\)
there is a connected rectifiable curve, lying in a fixed disc of radius \(4\),
of length at most \(12\) and admitting at most \(30\) increments at scale
\(1/2\), whose unit offset has at least \(1/(2\varepsilon)\) holes while the
area and the boundary length of that offset stay bounded. No bound on
\(\beta_{1}(A^{(r)})\) at a fixed radius can therefore be a locally bounded
function of the diameter, length and increment count of \(A\) and of the area
and boundary length of \(A^{(r)}\). The integrated estimate thus cannot be
replaced by a fixed-radius bound in terms of these coarse quantities, and the
fixed-radius question is genuinely probabilistic, which is how
Section~\ref{sec:8} treats it.
\end{remark}

\subsection{Mayer--Vietoris inequalities}\label{sec:4-4}

\begin{lemma}[Mayer--Vietoris bound in degree one; {\cite[Lemma 2.3]{G26g}}]\label{lem:6-2}
Let \(A,B \subset \mathbb{R}^{2}\) be compact and assume
\begin{equation}\label{eq:6-2-fin}
\beta_{1}(A) < \infty\quad\text{and}\quad \beta_{1}(B) < \infty .
\end{equation}
Then
\[- \beta_{1}(A \cap B)\  \leq \ \beta_{1}(A \cup B) - \beta_{1}(A) - \beta_{1}(B)\  \leq \ \beta_{0}(A \cap B),\]
and in particular
\[\left| \beta_{1}(A \cup B) - \beta_{1}(A) - \beta_{1}(B) \right|\  \leq \ \beta_{1}(A \cap B) + \beta_{0}(A \cap B).\]
\end{lemma}

The finiteness hypothesis \eqref{eq:6-2-fin} is not cosmetic: the proof
computes a difference of ranks, and that arithmetic has no meaning when
infinite dimensions occur. It is also no restriction in practice. Every
application below sits inside an absolutely convergent radius integral, and
Theorem~\ref{thm:4-11} supplies \eqref{eq:6-2-fin} at Lebesgue-almost every
radius, which is all such an integral sees. The one-sided upper bound
\(\beta_{1}(A \cup B) \leq \beta_{1}(A) + \beta_{1}(B) + \beta_{0}(A \cap B)\),
and the two statements that follow, need no finiteness hypothesis and hold at
every radius in \(\lbrack 0,\infty\rbrack\).

\begin{corollary}[interface bounds; {\cite[Corollary 2.4]{G26g}}]\label{cor:4-3}
For all compact \(U,V \subseteq \mathbb{R}^{2}\),
\[\beta_{1}(U \cap V)\, \leq \,\beta_{1}(U) + \beta_{1}(V),\quad\quad\beta_{0}(U \cap V)\, \leq \,\beta_{0}(U) + \beta_{0}(V) + \beta_{1}(U \cup V).\]
\end{corollary}

\begin{lemma}[planar intersection complexity for connected sets; {\cite[Lemma 2.5]{G26g}}]\label{lem:6-7}
Let \(A,B \subset \mathbb{R}^{2}\) be compact connected sets. Then
\(\beta_{0}(A \cap B) \leq \beta_{1}(A \cup B) + 1\) if
\(A \cap B \neq \varnothing\), and
\(\beta_{1}(A \cap B) \leq \beta_{1}(A) + \beta_{1}(B)\). Consequently
\[\beta_{1}(A \cap B) + \beta_{0}(A \cap B)\  \leq \ \beta_{1}(A) + \beta_{1}(B) + \beta_{1}(A \cup B) + 1.\]
\end{lemma}

Corollary~\ref{cor:4-3} and Lemma~\ref{lem:6-7} are what make the interface
terms of Section~\ref{sec:6} tractable: the intersection of two sausages is
not itself the offset of a compact set, but its Betti numbers are dominated by
hole counts of honest sausages, to which Theorem~\ref{thm:4-11} applies.

Several arguments below apply Fubini's theorem to Betti numbers of random
offsets, jointly in the sample point and the radius, to produce a single
deterministic null set of radii; this requires the following.

\begin{lemma}[measurability]\label{lem:4-meas}
Let \(\mathcal{K}\) be the space of nonempty compact subsets of
\(\mathbb{R}^{2}\) with the Hausdorff metric. The maps
\[(K,r) \mapsto \beta_{1}\left( K^{(r)} \right),\qquad
(K,K',r) \mapsto \beta_{1}\left( K^{(r)} \cap K'^{(r)} \right),\qquad
(K,K',r) \mapsto \beta_{0}\left( K^{(r)} \cap K'^{(r)} \right)\]
are Borel on \(\mathcal{K} \times (0,\infty)\), respectively
\(\mathcal{K} \times \mathcal{K} \times (0,\infty)\), with values in
\(\{ 0,1,\ldots\} \cup \{\infty\}\). Since
\(K^{(r)} \cup K'^{(r)} = (K \cup K')^{(r)}\) and \((K,K') \mapsto K \cup K'\)
is continuous, the same holds for Betti numbers of unions. Consequently the
random variables built from these maps in Sections~\ref{sec:6}--\ref{sec:8}
--- \(\Phi_{\psi}\) of a random compact set, \(Z_{k}\), \(Z_{k}(r)\),
\(W_{k}(r)\) --- are measurable, jointly in \((\omega,r)\) where the radius
enters as a variable, because \(\omega \mapsto K_{k}(\omega)\) is measurable
into \(\mathcal{K}\): the range map \((X,s,t) \mapsto X(\lbrack s,t\rbrack)\)
is continuous from \(C(\lbrack 0,\infty),\mathbb{R}^{2}) \times \{ s \leq t\}\)
into \(\mathcal{K}\), and the regeneration endpoints \(\tau_{k - 1},\tau_{k}\)
are measurable random variables.
\end{lemma}

\begin{proof}
Write \(f = d_{K}\) or \(f = \max(d_{K},d_{K'})\); in either case
\((K,K',x) \mapsto f(x)\) is continuous, since
\(|d_{K}(x) - d_{L}(x)| \leq d_{H}(K,L)\), and
\(\{ f \leq r\}\) is the compact set whose Betti numbers are in question. By
Alexander duality (Remark~\ref{rem:4-1}), \(\beta_{1}(\{ f \leq r\})\) is the
number of bounded components of the open set \(U:=\{ f > r\}\). Let \(Q\) be
the set of points of \(\mathbb{R}^{2}\) with rational coordinates. Two points
\(q,q' \in Q \cap U\) lie in the same component of \(U\) if and only if some
polygonal path \(P\) with vertices in \(Q\) joins them inside \(U\), because
\(U\) is open; and \(\{(K,K',r):P \subset U\} = \{\min_{P}f > r\}\) is open,
\(P\) being compact and \(f\) continuous. So, for each pair \((q,q')\), the
relation ``\(q \sim q'\)'' is an open condition in \((K,K',r)\). The
component of \(q\) is unbounded if and only if, for every \(n\), \(q \sim q'\)
for some \(q' \in Q\) with \(\lVert q' \rVert > n\) (if \(K \cup K' \subset B(0,R)\),
every point with \(\lVert x \rVert > R + r\) lies in the unbounded component),
a countable intersection of countable unions of open sets. Enumerating
\(Q = \{ q_{1},q_{2},\ldots\}\), the number of bounded components of \(U\) is
\[\sum_{i \geq 1}\mathbf{1}\left\{ q_{i} \in U,\ \text{the component of }q_{i}\text{ is bounded},\ q_{i} \not\sim q_{j}\text{ for all }j < i \right\},\]
every component of \(U\) containing a rational point; this is a countable sum
of Borel indicators, so the first two maps are Borel.

For the third, let \(X:=\{ f \leq r\}\). Components of a compact set coincide
with its quasicomponents, so if \(X\) has at least \(n\) components it is the
disjoint union of \(n\) nonempty closed subsets, at positive mutual distance;
hence \(\beta_{0}(X) = \lim_{\varepsilon \downarrow 0}c_{\varepsilon}\), where
\(c_{\varepsilon}\) is the number of components of the open set
\(\{\operatorname{dist}(\cdot,X) < \varepsilon\}\), and the limit is
nondecreasing as \(\varepsilon \downarrow 0\) (components of the
\(\varepsilon\)-neighbourhood can only split as \(\varepsilon\) decreases, and
each contains a point of \(X\)). The counting argument above applies to
\(c_{\varepsilon}\) once \((K,K',r,x) \mapsto \operatorname{dist}(x,X)\) is
known to be Borel: the containment of a polygonal path \(P\) in the open
\(\varepsilon\)-neighbourhood reads \(\max_{P}\operatorname{dist}(\cdot,X) < \varepsilon\),
and since \(\operatorname{dist}(\cdot,X)\) is \(1\)-Lipschitz the maximum over
\(P\) is the supremum over a countable dense subset of \(P\), a Borel condition.
And \(\operatorname{dist}(x,X)\) is Borel, since
\(\operatorname{dist}(x,X) = \lim_{s \downarrow r}\operatorname{dist}\left( x,\{ f < s\} \right) = \lim_{s \downarrow r}\inf\left\{ \lVert x - y \rVert :y \in Q,\ f(y) < s \right\}\),
the sets \(\{ f < s\}\) being open, with compact closures decreasing to \(X\)
as \(s \downarrow r\).
\end{proof}

\subsection{Two further deterministic facts}\label{sec:4-5}

The following two statements are used only in Section~\ref{sec:8-2}.

\begin{lemma}[small sausages are star-shaped; {\cite[Lemma 4.3]{G26g}}]\label{lem:8-4}
Let \(K \subset \mathbb{R}^{2}\) be compact nonempty with
\(d = \operatorname{diam}K\), let \(r \geq d/\sqrt{3}\) and let \(c\) be the
centre of the smallest disc containing \(K\). Then \(K^{(r)}\) is star-shaped
about every \(p\) with \(\lVert p - c \rVert \leq r - d/\sqrt{3}\); in
particular it is contractible and \(\beta_{1}K^{(r)} = 0\).
\end{lemma}

\begin{proposition}[perimeter of an offset; {\cite[Proposition 4.4]{G26g}}]\label{prop:4-perim}
For every compact \(A \subset \mathbb{R}^{2}\) and every \(r \notin C(A)\) at
which \(s \mapsto \left| A^{(s)} \right|\) is differentiable,
\(\mathcal{H}^{1}\left( \partial A^{(r)} \right) \leq 2\left| A^{(r)} \right|/r\).
\end{proposition}

\subsection{Moments at deterministic times}\label{sec:4-6}

Finally we record the one probabilistic consequence of Theorem~\ref{thm:4-11}
that is needed at deterministic times; at the regeneration times the
corresponding statements are Corollary~\ref{cor:5-15} and
Proposition~\ref{prop:6-5} below. Let \(X_{t} = x_{0} + \mu t + B_{t}\) be
planar Brownian motion with drift, \(\nu = \lVert \mu \rVert\),
\(K_{t} = X\left( \lbrack 0,t\rbrack \right)\), and
\(D_{t} = \operatorname{diam}K_{t} \leq 2\nu t + 2\sup_{s \leq t} \lVert B_{s} \rVert\).

\begin{lemma}[diameter moments]\label{lem:4-diam}
For every \(m \geq 1\) there is a constant \(C_{m}\), depending only on
\(m\), such that
\(\mathbb{E}\left\lbrack D_{t}^{\, m} \right\rbrack \leq C_{m}\left( 1 + (\nu t)^{m} + t^{m/2} \right)\)
for every \(t \geq 0\). Consequently, for the window \(0 < r_{0} < r_{1}\)
and every \(q \geq 1\),
\[\mathbb{E}\left\lbrack \left( \int_{r_{0}}^{r_{1}}\beta_{1}\left( K_{t}^{(r)} \right)dr \right)^{q} \right\rbrack\  \leq \ 4050^{\, q}\, r_{0}^{- 3q}\,\mathbb{E}\left\lbrack D_{t}^{4q} \right\rbrack\  \leq \ C'_{q}\, r_{0}^{- 3q}\left( 1 + (\nu t)^{4q} + t^{2q} \right).\]
\end{lemma}

\begin{proof}
\((a + b)^{m} \leq 2^{m - 1}\left( a^{m} + b^{m} \right)\) and Doob's \(L^{p}\)
inequality applied to each coordinate of \(B\) give
\(\mathbb{E}\left\lbrack \sup_{s \leq t} \lVert B_{s} \rVert^{m} \right\rbrack \leq C\, t^{m/2}\),
with \(C\) explicit from Gaussian moments; the second display is
Theorem~\ref{thm:4-11} applied pathwise. \end{proof}

\section{A regeneration structure for the drifted path}\label{sec:5}

\subsection{Setting and one-dimensional preliminaries}\label{sec:5-0}

Throughout,
\(\left( \Omega,\mathcal{F},\left( \mathcal{F}_{t} \right)_{t \geq 0},\mathbb{P} \right)\)
carries the coordinate process \(X_{t} = \mu t + B_{t}\) on
\(C\left( \lbrack 0,\infty \right),\mathbb{R}^{2})\), with \(B\) a
standard planar Brownian motion, \(\mu \neq 0\),
\(\nu = \lVert \mu \rVert > 0\), \(e = \mu/\nu\); without loss of
generality \(X_{0} = 0\). Write

\[U_{t} = \langle X_{t},\, e\rangle = \nu t + W_{t},\quad\quad M_{t} = \sup_{s \leq t}U_{s},\]

so \(W\) is a standard one-dimensional Brownian motion and, almost
surely, \(U_{t}/t \rightarrow \nu\), hence \(U_{t} \rightarrow \infty\).
For \(t \geq 0\) let \(\Theta_{t}\) denote the increment shift,
\(\left( \Theta_{t}X \right)_{s} = X_{t + s} - X_{t}\); recall that
\(X\) has stationary independent increments, so for every
\(\left( \mathcal{F}_{t} \right)\)-stopping time \(S\) and every
\(A \in \mathcal{F}_{S}\) with \(A \subseteq \{ S < \infty\}\) and
bounded measurable \(h\) on path space,

\begin{equation}\label{eq:5-incr}
\mathbb{E}\left\lbrack \mathbf{1}_{A}\, h\left( \Theta_{S}X \right) \right\rbrack = \mathbb{P}(A)\,\mathbb{E}\left\lbrack h(X) \right\rbrack.
\end{equation}

Fix the backtrack depth \(R > 0\) and the spacing \(L > 0\); at
assembly, Section~\ref{sec:6} imposes \(L > 2a > 2\left( R + 2r_{1} \right)\),
but only \(L > 0\) is used in the present section. Define, for any path,

\[D = \inf\{ t \geq 0:\ U_{t} - U_{0} = - R\} \in (0,\infty\rbrack.\]

We first record the classical one-dimensional facts about \(U\) that the
construction uses; the explicit success probability \(p_{R}\) and the
exponential moments of the hitting times both enter the quantitative bounds
below. Write
\(\sigma_{a}^{+}:= \inf\{ t \geq 0:U_{t} - U_{0} = a\}\) and
\(\sigma_{a}^{-}:= \inf\{ t \geq 0:U_{t} - U_{0} = - a\}\) for \(a > 0\).

\begin{lemma}[hitting-time formulas]\label{lem:5-1}
For every \(a > 0\) and every \(\lambda \geq 0\),
\[\mathbb{E}\left\lbrack e^{- \lambda\sigma_{a}^{+}} \right\rbrack = \exp\left( - a\left( \sqrt{\nu^{2} + 2\lambda} - \nu \right) \right),\qquad
\mathbb{E}\left\lbrack e^{- \lambda\sigma_{a}^{-}};\sigma_{a}^{-} < \infty \right\rbrack = \exp\left( - a\left( \sqrt{\nu^{2} + 2\lambda} + \nu \right) \right).\]
In particular \(\mathbb{P}\left( \sigma_{a}^{-} < \infty \right) = e^{- 2\nu a}\),
\(\mathbb{E}\left\lbrack \sigma_{a}^{+} \right\rbrack = a/\nu\), and
\(\mathbb{E}\left\lbrack \sigma_{a}^{-} \mid \sigma_{a}^{-} < \infty \right\rbrack = a/\nu\).
Moreover, for every \(0 \leq \theta < \nu^{2}/2\),
\[\mathbb{E}\left\lbrack e^{\theta\sigma_{a}^{+}} \right\rbrack = \exp\left( a\left( \nu - \sqrt{\nu^{2} - 2\theta} \right) \right) < \infty .\]
\end{lemma}

\begin{proof}
\(\left( e^{\kappa\left( U_{t} - U_{0} \right) - \left( \kappa\nu + \kappa^{2}/2 \right)t} \right)_{t \geq 0}\)
is a martingale for every \(\kappa \in \mathbb{R}\). Optional stopping at
\(\sigma_{a}^{+} \wedge n\), with \(\kappa\) the nonnegative root of
\(\kappa\nu + \kappa^{2}/2 = \lambda\), namely
\(\kappa = \sqrt{\nu^{2} + 2\lambda} - \nu\), gives the first formula on letting
\(n \rightarrow \infty\) (the martingale is bounded by \(e^{\kappa a}\) before
\(\sigma_{a}^{+}\), and \(\sigma_{a}^{+} < \infty\) a.s.\ since \(\nu > 0\)).
Taking \(\kappa\) the nonpositive root, \(\kappa = - \sqrt{\nu^{2} + 2\lambda} - \nu\),
and stopping at \(\sigma_{a}^{-} \wedge n\) gives the second, the contribution of
\(\{\sigma_{a}^{-} = \infty\}\) vanishing in the limit. Setting \(\lambda = 0\)
in the second formula gives
\(\mathbb{P}\left( \sigma_{a}^{-} < \infty \right) = e^{- 2\nu a}\); the two mean
formulas follow by differentiating at \(\lambda = 0\). The last display is the
first with \(\lambda = - \theta\), legitimate for \(\theta < \nu^{2}/2\) since
then \(\sqrt{\nu^{2} - 2\theta}\) is real and the corresponding exponential
martingale is still uniformly integrable up to \(\sigma_{a}^{+}\).
\end{proof}

\begin{corollary}[probability of no backtracking beyond \(R\)]\label{cor:5-2}
For every \(R > 0\),
\[p_{R}:= \mathbb{P}(D = \infty)\  = \ \mathbb{P}\left( \inf_{t \geq 0}\left( U_{t} - U_{0} \right) > - R \right)\  = \ 1 - e^{- 2\nu R}\  \in \ (0,1),\]
and this value does not depend on the starting point.
\end{corollary}

\begin{proof}
\(\{ D < \infty\} = \left\{ \sigma_{R}^{-} < \infty \right\}\), so
Lemma~\ref{lem:5-1} gives
\(\mathbb{P}(D < \infty) = e^{- 2\nu R} \in (0,1)\). Spatial homogeneity of
\(U - U_{0}\) gives the last assertion.
\end{proof}

\subsection{The cut construction}\label{sec:5-1}

\begin{definition}[cut times]\label{def:5-4}
~ Set, with
\(\inf\varnothing = \infty\),

\[S_{0} = \inf\{ t \geq 0:\ U_{t} = U_{0} + L\},\]

and recursively, for \(k \geq 0\): on
\(\{ S_{k} < \infty,\ D\left( \Theta_{S_{k}}X \right) < \infty\}\) put

\[R_{k} = S_{k} + D\left( \Theta_{S_{k}}X \right),\quad\quad S_{k + 1} = \inf\{ t \geq R_{k}:\ U_{t} = M_{t}\ \text{and}\ U_{t} \geq U_{S_{k}} + L\},\]

while on
\(\{ S_{k} < \infty,\ D\left( \Theta_{S_{k}}X \right) = \infty\}\) the
recursion stops (then \(R_{k} = S_{k + 1} = \ldots = \infty\)). Let

\[K = \inf\{ k \geq 0:\ S_{k} < \infty,\ D\left( \Theta_{S_{k}}X \right) = \infty\},\quad\quad\tau_{1} = S_{K}\ \ \text{on }\{ K < \infty\}.\]

Higher cuts are defined recursively through the increment shift:

\[\tau_{m + 1} = \tau_{m} + \tau_{1}\left( Y^{(m)} \right),\quad\quad Y^{(0)} = X,\quad Y^{(m)} = \Theta_{\tau_{1}\left( Y^{(m - 1)} \right)}Y^{(m - 1)}\ \ (m \geq 1),\]

so that \(Y^{(m)} = \Theta_{\tau_{m}}X\) is the post-\(\tau_{m}\)
increment path.

Thus a candidate \(S_{k}\) is a fresh running maximum lying at
least \(L\) above the previous candidate's level; it succeeds if
the path never drops more than \(R\) below it afterwards, and upon
failure (at the backtrack time \(R_{k}\)) the search resumes at the next
admissible fresh maximum. The first cut \(\tau_{1}\) is the first
successful candidate.
\end{definition}

\begin{lemma}[elementary properties]\label{lem:5-5}
~

(i)~ \(S_{0}\) and, for every \(k \geq 0\), \(R_{k}\) and \(S_{k + 1}\)
are \(\left( \mathcal{F}_{t} \right)\)-stopping times, with
\(R_{k} \leq S_{k + 1}\) everywhere. On \(\{ R_{k} < \infty\}\): almost
surely \(S_{k + 1} < \infty\) and

\[U_{S_{k + 1}} = M_{R_{k}} \vee \left( U_{S_{k}} + L \right).\]

Every finite candidate is a fresh-maximum time,
\(U_{S_{k}} = M_{S_{k}}\), and candidate levels grow by at least \(L\):
\(U_{S_{k + 1}} \geq U_{S_{k}} + L\); in particular
\(U_{S_{k}} \geq U_{0} + L\) for all \(k\).

(ii)~ For every \(k\), on \(\{ S_{k} < \infty\}\),

\[\mathbb{P}\left( D\left( \Theta_{S_{k}}X \right) = \infty\ |\ \mathcal{F}_{S_{k}} \right) = p_{R}.\]

Consequently \(\mathbb{P}(K \geq k) = \left( 1 - p_{R} \right)^{k}\) for
all \(k \geq 0\) --- the number of candidates \(K + 1\) is geometric
with success probability \(p_{R}\) --- and \(\tau_{1} = S_{K} < \infty\)
almost surely, with \(U_{\tau_{1}} \geq U_{0} + L\).
\end{lemma}

\begin{proof}
(i) \(S_{0}\) is the d\'ebut of a closed set by the
continuous adapted process \(U\), hence a stopping time. For \(R_{k}\):
on \(\{ S_{k} \leq t\}\) the event \(\{ R_{k} \leq t\}\) reads
\(\{\min_{s \in \left\lbrack S_{k},t \right\rbrack}\left( U_{s} - U_{S_{k}} \right) \leq - R\} \in \mathcal{F}_{t}\),
so \(R_{k}\) is a stopping time. For \(S_{k + 1}\): on
\(\{ R_{k} < \infty\}\) the process
\(\varphi_{t} = \left( M_{t} - U_{t} \right) \vee \left( U_{S_{k}} + L - U_{t} \right)\),
\(t \geq R_{k}\), is continuous and adapted (as \(U_{S_{k}}\) is
\(\mathcal{F}_{S_{k}}\)-, hence \(\mathcal{F}_{t}\)-measurable for
\(t \geq S_{k}\)), and \(S_{k + 1}\) is its first zero after the
stopping time \(R_{k}\); the d\'ebut of a closed set after a stopping time
by a continuous adapted process is a stopping time.
\(R_{k} \leq S_{k + 1}\) holds by definition.

Finiteness and the value: on \(\{ R_{k} < \infty\}\) set
\(m^{*} = M_{R_{k}} \vee \left( U_{S_{k}} + L \right)\) and
\(T^{*} = \inf\{ t \geq R_{k}:\ U_{t} = m^{*}\}\), finite almost surely
since \(U_{t} \rightarrow \infty\). Since
\(U_{R_{k}} = U_{S_{k}} - R < m^{*}\) and \(T^{*}\) is the first hitting
of \(m^{*}\) after \(R_{k}\), continuity gives \(U_{s} < m^{*}\) on
\(\lbrack R_{k},T^{*})\). If some \(s \in \lbrack R_{k},T^{*})\) had
\(\varphi_{s} = 0\), then \(U_{s} = M_{s} \geq M_{R_{k}}\) and
\(U_{s} \geq U_{S_{k}} + L\), whence \(U_{s} \geq m^{*}\) --- a
contradiction; so \(S_{k + 1} \geq T^{*}\). At \(T^{*}\):
\(U_{T^{*}} = m^{*} \geq U_{S_{k}} + L\) and
\[M_{T^{*}} = M_{R_{k}} \vee \sup_{\left\lbrack R_{k},T^{*} \right\rbrack}U = M_{R_{k}} \vee m^{*} = m^{*} = U_{T^{*}}.\]
Hence \(S_{k + 1} = T^{*}\) and \(U_{S_{k + 1}} = m^{*}\); both
displayed claims follow, and \(U_{S_{k + 1}} = M_{S_{k + 1}}\). For
\(k = 0\), \(U_{S_{0}} = U_{0} + L = M_{S_{0}}\) because \(S_{0}\) is
the first hitting of that level.

(ii) The event \(\{ D\left( \Theta_{S_{k}}X \right) = \infty\}\) is
measurable in the increment path \(\Theta_{S_{k}}X\), so \eqref{eq:5-incr}
gives the conditional probability \(\mathbb{P}(D = \infty) = p_{R}\) on
\(\{ S_{k} < \infty\}\). Next,
\(\{ K \geq k\} = \{ R_{k - 1} < \infty\}\) (candidates
\(0,\ldots,k - 1\) were all launched and failed; by (i) the recursion
never halts for lack of a candidate: \(S_{j} < \infty\) a.s.~on
\(\{ R_{j - 1} < \infty\}\)). Since
\(\{ R_{k - 1} < \infty\} \in \mathcal{F}_{R_{k - 1}} \subseteq \mathcal{F}_{S_{k}}\),

\[\mathbb{P}\left( R_{k} < \infty \right) = \mathbb{E}\left\lbrack \mathbf{1}_{\{ R_{k - 1} < \infty\}}\,\mathbb{P}\left( D\left( \Theta_{S_{k}}X \right) < \infty\,|\,\mathcal{F}_{S_{k}} \right) \right\rbrack = \left( 1 - p_{R} \right)\,\mathbb{P}\left( R_{k - 1} < \infty \right),\]

and \(\mathbb{P}\left( R_{0} < \infty \right) = 1 - p_{R}\); induction
gives \(\mathbb{P}(K \geq k) = \left( 1 - p_{R} \right)^{k}\). On
\(\{ K = k\}\) one has \(S_{k} < \infty\) by the definition of \(K\), so
\(\tau_{1} < \infty\) a.s.; \(U_{\tau_{1}} \geq U_{0} + L\) by (i).
\end{proof}

\begin{remark}[what is, and is not, claimed]\label{rem:5-6}
The cut times \(\tau_{m}\) are not stopping times: \(\{\tau_{1} \leq t\}\)
depends on the entire future through the tests
\(\{ D\left( \Theta_{S_{k}}X \right) = \infty\}\). Accordingly no strong Markov
property is applied at any \(\tau_{m}\) anywhere in this paper; every appeal
to the Markov property occurs at the genuine stopping times \(S_{k}\),
\(R_{k}\), via \eqref{eq:5-incr}, and the renewal structure is obtained
through Theorem~\ref{thm:5-7} below. Likewise the success events of distinct
candidates are not claimed to be jointly independent --- they overlap through
the common future --- and only the conditional identity of
Lemma~\ref{lem:5-5}(ii) is used.
\end{remark}

\subsection{The renewal theorem}\label{sec:5-2}

Let

\[\mathcal{G}_{1} = \sigma\left( \tau_{1},\ \left( X_{t \land \tau_{1}} \right)_{t \geq 0} \right)\]
be the \(\sigma\)-field of the first cut time and the stopped path, and
let

\[\overline{\mathbb{P}} = \mathbb{P}(\, \cdot \mid D = \infty)\]
denote the law of the path conditioned never to drop more than \(R\)
below its starting level. Since
\(\overline{\mathbb{P}} \ll \mathbb{P}\), every \(\mathbb{P}\)-almost
sure statement above holds \(\overline{\mathbb{P}}\)-almost surely, and
the construction of Definition~\ref{def:5-4} is well defined under
\(\overline{\mathbb{P}}\).

\begin{theorem}[regeneration at \(\tau_1\)]\label{thm:5-7}
~

(a)~ Under \(\mathbb{P}\), the increment path \(\Theta_{\tau_{1}}X\) is
independent of \(\mathcal{G}_{1}\) and has law
\(\overline{\mathbb{P}}\).

(b)~ The same holds under \(\overline{\mathbb{P}}\): the
\(\overline{\mathbb{P}}\)-law of \(\Theta_{\tau_{1}}X\) is again
\(\overline{\mathbb{P}}\), and \(\Theta_{\tau_{1}}X\) is
\(\overline{\mathbb{P}}\)-independent of \(\mathcal{G}_{1}\).
\end{theorem}

\begin{proof}
(a) Let \(\varphi\) be bounded measurable on
\(\lbrack 0,\infty) \times C\left( \lbrack 0,\infty \right),\mathbb{R}^{2})\),
\(F\) bounded measurable on
\(C\left( \lbrack 0,\infty \right),\mathbb{R}^{2})\), and put
\(Z = \varphi\left( \tau_{1},X_{\cdot \land \tau_{1}} \right)\); such
\(Z\) generate \(\mathcal{G}_{1}\). Decompose over the value of \(K\)
(a.s.~finite, Lemma~\ref{lem:5-5}(ii)):

\[\mathbb{E}\left\lbrack Z\, F\left( \Theta_{\tau_{1}}X \right) \right\rbrack = \sum_{k \geq 0}^{}\mathbb{E}\left\lbrack Z\, F\left( \Theta_{S_{k}}X \right);\, K = k \right\rbrack.\]

On \(\{ K = k\}\): \(\tau_{1} = S_{k}\), so
\(Z = \varphi\left( S_{k},X_{\cdot \land S_{k}} \right) =: Z_{k}\),
which is \(\mathcal{F}_{S_{k}}\)-measurable (a path stopped at a
stopping time is measurable for the stopped \(\sigma\)-field); and

\[\mathbf{1}_{\{ K = k\}} = \mathbf{1}_{\{ G_{k}\}} \cdot \mathbf{1}_{\{ D = \infty\}} \circ \Theta_{S_{k}},\quad\quad G_{k}:= \{ S_{k} < \infty\} \cap \{ R_{j} < \infty\ \forall j < k\} \in \mathcal{F}_{S_{k}},\]

using \(R_{j} \leq S_{k}\) for \(j < k\). Hence, by \eqref{eq:5-incr}
applied at \(S_{k}\) with \(h = \mathbf{1}_{\{ D = \infty\}}F\),

\[\mathbb{E}\left\lbrack Z\, F\left( \Theta_{S_{k}}X \right);\, K = k \right\rbrack = \mathbb{E}\left\lbrack Z_{k}\mathbf{1}_{\{ G_{k}\}} \right\rbrack \cdot \mathbb{E}\left\lbrack \mathbf{1}_{\{ D = \infty\}}F \right\rbrack = \mathbb{E}\left\lbrack Z_{k}\mathbf{1}_{\{ G_{k}\}} \right\rbrack \cdot p_{R}\,\overline{\mathbb{E}}\lbrack F\rbrack.\]

The terms are summable
(\(\left| \mathbb{E}\left\lbrack Z_{k}\mathbf{1}_{\{ G_{k}\}} \right\rbrack \right| \leq \lVert \varphi \rVert_{\infty}\left( 1 - p_{R} \right)^{k}\)),
so

\[\mathbb{E}\left\lbrack Z\, F\left( \Theta_{\tau_{1}}X \right) \right\rbrack = p_{R}\,\overline{\mathbb{E}}\lbrack F\rbrack\sum_{k \geq 0}^{}\mathbb{E}\left\lbrack Z_{k}\mathbf{1}_{\{ G_{k}\}} \right\rbrack.\]

Taking \(F \equiv 1\) identifies the sum:
\[\sum_{k}^{}\mathbb{E}\left\lbrack Z_{k}\mathbf{1}_{\{ G_{k}\}} \right\rbrack = p_{R}^{- 1}\mathbb{E}\lbrack Z\rbrack.\]
Substituting back,

\[\mathbb{E}\left\lbrack Z\, F\left( \Theta_{\tau_{1}}X \right) \right\rbrack = \mathbb{E}\lbrack Z\rbrack\,\overline{\mathbb{E}}\lbrack F\rbrack,\]

which is simultaneously the independence and (with \(\varphi \equiv 1\))
the identification of the law.

(b) First observe that, \(\mathbb{P}\)-almost surely,

\begin{equation}\label{eq:5-sep}
\inf_{t \geq \tau_{1}}U_{t}\, > \, U_{\tau_{1}} - R\, \geq \, U_{0} + L - R\, > \, U_{0} - R.
\end{equation}

Indeed the success of the \(K\)-th candidate means
\(D\left( \Theta_{\tau_{1}}X \right) = \infty\); since
\(\langle\left( \Theta_{\tau_{1}}X \right)_{t},\, e\rangle \rightarrow \infty\),
its infimum over \(\lbrack 0,\infty)\) is attained on a compact
interval, so ``never hits \(- R\)'' forces the strict inequality
\(\inf_{t}\langle\left( \Theta_{\tau_{1}}X \right)_{t},\, e\rangle > - R\),
and \(U_{\tau_{1}} \geq U_{0} + L\) with \(L > 0\) gives the
rest. By \eqref{eq:5-sep} the whole-path event
\(\{ D = \infty\} = \{\inf_{t \geq 0}U_{t} > U_{0} - R\}\) coincides, up
to a \(\mathbb{P}\)-null set, with the pre-cut event

\[\{ D = \infty\} = \{\inf_{t \leq \tau_{1}}U_{t} > U_{0} - R\} \in \mathcal{G}_{1}.\]

Hence for \(Z\), \(F\) as in (a),

\[\overline{\mathbb{E}}\left\lbrack Z\, F\left( \Theta_{\tau_{1}}X \right) \right\rbrack = p_{R}^{- 1}\,\mathbb{E}\left\lbrack Z\,\mathbf{1}_{\{ D = \infty\}}F\left( \Theta_{\tau_{1}}X \right) \right\rbrack = p_{R}^{- 1}\,\mathbb{E}\left\lbrack Z\mathbf{1}_{\{ D = \infty\}} \right\rbrack\,\overline{\mathbb{E}}\lbrack F\rbrack = \overline{\mathbb{E}}\lbrack Z\rbrack\,\overline{\mathbb{E}}\lbrack F\rbrack,\]

where the middle equality is (a) applied to the
\(\mathcal{G}_{1}\)-measurable variable
\(Z\mathbf{1}_{\{ D = \infty\}}\). Taking \(Z \equiv 1\) and then
general \(Z\) gives both claims. \end{proof}

\begin{corollary}[independent blocks; delayed renewal]\label{cor:5-8}
~ Define the
blocks

\[\mathcal{B}_{0} = \left( \tau_{1},\ X_{\cdot \land \tau_{1}} \right),\qquad
\mathcal{B}_{m} = \left( \tau_{1}\left( Y^{(m)} \right),\ Y_{\cdot \land \tau_{1}\left( Y^{(m)} \right)}^{(m)} \right)\quad (m \geq 1),\]
so that equivalently
\[\mathcal{B}_{m} = \left( \tau_{m + 1} - \tau_{m},\ \left( \Theta_{\tau_{m}}X \right)_{\cdot \land \left( \tau_{m + 1} - \tau_{m} \right)} \right),\qquad m \geq 1.\]

Then, under \(\mathbb{P}\): all \(\tau_{m}\) are finite almost surely
and \(\tau_{m} \uparrow \infty\); the blocks
\(\left( \mathcal{B}_{m} \right)_{m \geq 0}\) are independent;
\(\left( \mathcal{B}_{m} \right)_{m \geq 1}\) are i.i.d.~with the law of \(\mathcal{B}_{0}\)
under \(\overline{\mathbb{P}}\); and for every \(m \geq 1\) the post-cut
path \(Y^{(m)} = \Theta_{\tau_{m}}X\) is independent of

\[\mathcal{G}_{m}:= \sigma\left( \mathcal{B}_{0},\ldots,\mathcal{B}_{m - 1} \right) = \sigma\left( \tau_{1},\ldots,\tau_{m},\ \left( X_{t \land \tau_{m}} \right)_{t \geq 0} \right)\]
with law \(\overline{\mathbb{P}}\).
\end{corollary}

\begin{proof}
We prove by induction on \(m \geq 1\) the statement \(\left( I_{m} \right)\):
\begin{quote}
\(Y^{(m)}\) is independent of
\(\left( \mathcal{B}_{0},\ldots,\mathcal{B}_{m - 1} \right)\) and has law
\(\overline{\mathbb{P}}\); and \(\mathcal{B}_{1},\ldots,\mathcal{B}_{m - 1}\)
are i.i.d.\ with the law of \(\mathcal{B}_{0}\) under
\(\overline{\mathbb{P}}\), independent of \(\mathcal{B}_{0}\).
\end{quote} Statement \(\left( I_{1} \right)\) is
Theorem~\ref{thm:5-7}(a). Assume \(\left( I_{m} \right)\). Both \(\mathcal{B}_{m}\) and \(Y^{(m + 1)}\) are
fixed measurable functionals of \(Y^{(m)}\) alone:
\(\mathcal{B}_{m} = \mathcal{B}_{0}\left( Y^{(m)} \right)\) and
\(Y^{(m + 1)} = \Theta_{\tau_{1}\left( Y^{(m)} \right)}Y^{(m)}\). Since
\(Y^{(m)}\) has law \(\overline{\mathbb{P}}\), Theorem~\ref{thm:5-7}(b) applied to
it yields: \(\mathcal{B}_{0}\left( Y^{(m)} \right)\) and
\(\Theta_{\tau_{1}\left( Y^{(m)} \right)}Y^{(m)}\) are independent, with
respective laws (law of \(\mathcal{B}_{0}\) under \(\overline{\mathbb{P}}\)) and
\(\overline{\mathbb{P}}\). As the pair
\(\left( \mathcal{B}_{m},Y^{(m + 1)} \right)\) is a functional of \(Y^{(m)}\), it
is moreover independent of \(\left( \mathcal{B}_{0},\ldots,\mathcal{B}_{m - 1} \right)\) by
\(\left( I_{m} \right)\); this is exactly \(\left( I_{m + 1} \right)\).
The independence of the full sequence
\(\left( \mathcal{B}_{m} \right)_{m \geq 0}\) and the law identification follow.

Finiteness: \(\tau_{1} < \infty\) a.s.~under \(\mathbb{P}\)
(Lemma~\ref{lem:5-5}(ii)) and under \(\overline{\mathbb{P}}\) (absolute
continuity), so by induction every \(\tau_{m + 1} - \tau_{m} < \infty\)
a.s.; the increments
\(\left( \tau_{m + 1} - \tau_{m} \right)_{m \geq 1}\) are i.i.d.~and
strictly positive (\(\tau_{1} \geq S_{0} > 0\) pathwise), so
\(\tau_{m} \uparrow \infty\) a.s.

Finally
\(\sigma\left( \mathcal{B}_{0},\ldots,\mathcal{B}_{m - 1} \right) = \sigma\left( \tau_{1},\ldots,\tau_{m},X_{\cdot \land \tau_{m}} \right)\),
since concatenating the stopped block paths reconstructs
\(X_{\cdot \land \tau_{m}}\) and the \(\tau_{j}\), and conversely.
\end{proof}

\begin{lemma}[good-cut property]\label{lem:5-9}
~ Almost surely, for every
\(m \geq 1\):

(i)~ \(U_{\tau_{m}} = M_{\tau_{m}}\); hence
\(X\left( \left\lbrack 0,\tau_{m} \right\rbrack \right) \subseteq \{ x:\ \langle x,\, e\rangle \leq U_{\tau_{m}}\}\);

(ii)~ \(\inf_{t \geq \tau_{m}}U_{t} > U_{\tau_{m}} - R\); hence
\(X\left( \lbrack\tau_{m},\infty \right)) \subseteq \{ x:\ \langle x,\, e\rangle > U_{\tau_{m}} - R\}\);

(iii)~ \(U_{\tau_{m + 1}} \geq U_{\tau_{m}} + L\).

In particular the cut levels \(\left( U_{\tau_{m}} \right)_{m}\) are
strictly increasing with gaps at least \(L\), and for every
\(0 < r \leq r_{1}\) the block sausages
\(X\left( \left\lbrack \tau_{m},\tau_{m + 1} \right\rbrack \right)^{(r)}\)
satisfy the half-plane confinements used in Lemma~\ref{lem:6-1}.
\end{lemma}

\begin{proof}
Write \(Y = Y^{(m)}\), \(\sigma = \tau_{1}(Y)\), and for
a path \(y\) let \(U_{t}^{y} = \langle y_{t},\, e\rangle\),
\(M_{t}^{y} = \sup_{s \leq t}U_{s}^{y}\). All conclusions of Lemma~\ref{lem:5-5}
and display \eqref{eq:5-sep} hold \(\mathbb{P}\)-a.s., hence
\(\overline{\mathbb{P}}\)-a.s.~(\(\overline{\mathbb{P}} \ll \mathbb{P}\)),
hence a.s.~for each \(Y^{(m)}\) (law \(\overline{\mathbb{P}}\) by
Corollary~\ref{cor:5-8}).

\(m = 1\): (i) is Lemma~\ref{lem:5-5}(i) at the candidate \(S_{K}\); (ii) is
\eqref{eq:5-sep}; (iii) is
\(U_{\tau_{2}} - U_{\tau_{1}} = U_{\sigma}^{Y} \geq L\) by Lemma~\ref{lem:5-5}
applied to \(Y\).

\(m \rightarrow m + 1\): by (i) at \(m\) and
\(U_{\sigma}^{Y} = M_{\sigma}^{Y} \geq L > 0\) (Lemma~\ref{lem:5-5}(i) for \(Y\)),

\[M_{\tau_{m + 1}} = M_{\tau_{m}} \vee \left( U_{\tau_{m}} + M_{\sigma}^{Y} \right) = U_{\tau_{m}} + M_{\sigma}^{Y} = U_{\tau_{m}} + U_{\sigma}^{Y} = U_{\tau_{m + 1}},\]

which is (i) at \(m + 1\). (ii) at \(m + 1\) is \eqref{eq:5-sep} for
\(Y\), translated by \(U_{\tau_{m}}\):
\[\inf_{t \geq \tau_{m + 1}}U_{t} = U_{\tau_{m}} + \inf_{s \geq \sigma}U_{s}^{Y} > U_{\tau_{m}} + U_{\sigma}^{Y} - R = U_{\tau_{m + 1}} - R.\]
(iii) at \(m + 1\) is Lemma~\ref{lem:5-5} for \(Y^{(m + 1)}\), translated by
\(U_{\tau_{m + 1}}\). \end{proof}

\begin{remark}[provenance]\label{rem:5-10}
The \(\tau_{m}\) are the successive successful cuts of a single candidate
search on the original path, since after \(\tau_{m}\) a time is a fresh
maximum of \(U\) if and only if \(Y^{(m)}\) is at a fresh maximum
(Lemma~\ref{lem:5-9}(i)) and the backtrack tests agree on the two paths. The
argument is the renewal structure of Sznitman and Zerner \cite{SZ99},
specialized to a driving process with stationary independent increments: the
absence of a random environment collapses their inductive coupling to the
single computation in the proof of Theorem~\ref{thm:5-7}, and the independence
obtained is exact rather than conditional on an environment \(\sigma\)-field.
\end{remark}

\subsection{The conditioning cost, and exponential moments}\label{sec:5-3}

This subsection supplies the quantitative layer: the regeneration epochs
have exponential moments, with constants explicit in \((\nu,R,L)\).
Throughout, for \(0 < \theta \leq \nu^{2}/2\) put

\[c_{\theta} = \nu - \sqrt{\nu^{2} - 2\theta} \in (0,\nu\rbrack,\quad\quad c_{\theta} \downarrow 0\ \text{as}\ \theta \downarrow 0.\]

The elementary inequality of Lemma~\ref{lem:5-13} below replaces, here and in
the next section, every appeal to the Markov property at a cut time: a bound
valid for the unconditioned path transfers to the block law, and to
conditional expectations given the pre-cut past, at the multiplicative cost
\(p_{R}^{- 1} = \left( 1 - e^{- 2\nu R} \right)^{- 1}\).

\begin{lemma}[one-dimensional estimates]\label{lem:5-12}
~ Let
\(T_{h} = \inf\{ t \geq 0:\ U_{t} - U_{0} = h\}\) for \(h > 0\), and
\({\overline{M}}_{D} = \sup_{t \leq D}\left( U_{t} - U_{0} \right)\).
Then: (i)
\(\mathbb{E}\left\lbrack e^{\theta T_{h}} \right\rbrack = e^{c_{\theta}h}\)
for \(0 \leq \theta \leq \nu^{2}/2\); in particular
\(\mathbb{E}\left\lbrack T_{h} \right\rbrack = h/\nu\). (ii)
\(\mathbb{E}\left\lbrack e^{\theta D} \mid D < \infty \right\rbrack = e^{c_{\theta}R}\)
for \(0 \leq \theta \leq \nu^{2}/2\); in particular
\(\mathbb{E}\lbrack D \mid D < \infty\rbrack = R/\nu\). (iii)
\(\mathbb{P}\left( {\overline{M}}_{D} \geq x \mid D < \infty \right) \leq e^{- 2\nu x}\)
for \(x \geq 0\); hence
\(\mathbb{E}\left\lbrack e^{\lambda{\overline{M}}_{D}} \mid D < \infty \right\rbrack \leq 2\nu/(2\nu - \lambda)\)
for \(0 \leq \lambda < 2\nu\), and
\(\mathbb{E}\left\lbrack {\overline{M}}_{D} \mid D < \infty \right\rbrack \leq (2\nu)^{- 1}\).
\end{lemma}

\begin{proof}
(i) and the unconditioned form of (ii),
\[\mathbb{E}\left\lbrack e^{\theta D};D < \infty \right\rbrack = e^{- R\left( \nu + \sqrt{\nu^{2} - 2\theta} \right)},\]
are the classical first-passage formulas for Brownian motion with drift:
apply optional stopping to the martingales
\(\exp\{\gamma_{\pm}\left( U_{t} - U_{0} \right) + \theta t\}\), where
\(\gamma_{\pm} = - \nu \pm \sqrt{\nu^{2} - 2\theta}\) are the roots of
\(\gamma^{2}/2 + \nu\gamma + \theta = 0\), at the stopping time
\(T_{h} \land T_{- n}\), and let \(n \rightarrow \infty\)
(cf.~Lemma~\ref{lem:5-1} and \cite[Part~II.2]{BS02}). Dividing by
\(\mathbb{P}(D < \infty) = e^{- 2\nu R}\) gives (ii), since
\(- R\left( \nu + \sqrt{\nu^{2} - 2\theta} \right) + 2\nu R = Rc_{\theta}\);
differentiating at \(\theta = 0\) gives the means.

(iii) On \(\{{\overline{M}}_{D} \geq x\} \cap \{ D < \infty\}\) one has
\(T_{x} \leq D < \infty\), and after \(T_{x}\) the longitudinal
coordinate must decrease by \(x + R\); by the strong Markov property at
\(T_{x}\) and Corollary~\ref{cor:5-2} this has probability
\(e^{- 2\nu(x + R)}\), so
\(\mathbb{P}\left( {\overline{M}}_{D} \geq x,\, D < \infty \right) \leq e^{- 2\nu(x + R)}\);
divide by \(e^{- 2\nu R}\). The moment bounds follow from
\[\mathbb{E}\left\lbrack e^{\lambda{\overline{M}}_{D}} \mid \cdot \right\rbrack = 1 + \lambda\int_{0}^{\infty}e^{\lambda x}\mathbb{P}\left( {\overline{M}}_{D} \geq x \mid \cdot \right)\, dx\]
and
\[\mathbb{E}\left\lbrack {\overline{M}}_{D} \mid \cdot \right\rbrack = \int_{0}^{\infty}\mathbb{P}\left( {\overline{M}}_{D} \geq x \mid \cdot \right)\, dx.\]
\end{proof}

\begin{lemma}[conditioning cost]\label{lem:5-13}
~ For every measurable
\(F \geq 0\) on path space,

\[\overline{\mathbb{E}}\lbrack F\rbrack\  = \ p_{R}^{- 1}\,\mathbb{E}\lbrack F;\, D = \infty\rbrack\  \leq \ p_{R}^{- 1}\,\mathbb{E}\lbrack F\rbrack,\]

and for every \(m \geq 1\), almost surely,

\[\mathbb{E}\left\lbrack F\left( \Theta_{\tau_{m}}X \right)\,|\,\mathcal{G}_{m} \right\rbrack\  = \ \overline{\mathbb{E}}\lbrack F\rbrack\  \leq \ p_{R}^{- 1}\,\mathbb{E}\lbrack F\rbrack.\]
\end{lemma}

\begin{proof}
The first display is the definition of
\(\overline{\mathbb{P}}\); the second is Corollary~\ref{cor:5-8}. \end{proof}

\begin{proposition}[exponential moments of the epochs]\label{prop:5-14}
~ For
\(\theta \in \left( 0,\nu^{2}/4 \right)\) set

\[m(\theta)\  = \ e^{c_{\theta}(L + R)}\, e^{Rc_{2\theta}/2}\,\left( \frac{\nu}{\nu - c_{\theta}} \right)^{1/2}.\]

Then \(m\) is continuous with \(m(0 + ) = 1\), so
\(\Theta^{*} = \{\theta \in \left( 0,\nu^{2}/4 \right):\ \left( 1 - p_{R} \right)\, m(\theta) < 1\}\)
is nonempty. For every \(\theta \in \Theta^{*}\),

\[\mathbb{E}\left\lbrack e^{\theta\tau_{1}} \right\rbrack\  \leq \ \frac{p_{R}\, e^{c_{\theta}L}}{1 - \left( 1 - p_{R} \right)\, m(\theta)}\  < \ \infty,\quad\quad\overline{\mathbb{E}}\left\lbrack e^{\theta\tau_{1}} \right\rbrack\  \leq \ p_{R}^{- 1}\,\mathbb{E}\left\lbrack e^{\theta\tau_{1}} \right\rbrack\  < \ \infty.\]

Consequently the i.i.d.~block durations \(\tau_{m + 1} - \tau_{m}\)
(\(m \geq 1\)) and the delayed one \(\tau_{1}\) have exponential
moments; in particular
\(\bar{\mu}:= \overline{\mathbb{E}}\left\lbrack \tau_{1} \right\rbrack \in (0,\infty)\)
and
\(\overline{\mathbb{E}}\left\lbrack \tau_{1}^{\, q} \right\rbrack < \infty\)
for every \(q \geq 1\).
\end{proposition}

\begin{proof}
\emph{Step 1 (the trial variables).} For \(k \geq 0\)
define, on \(\{ S_{k} < \infty\}\), the failure indicator
\(\phi_{k} = \mathbf{1}_{\{ D\left( \Theta_{S_{k}}X \right) < \infty\}}\)
and, on \(\{\phi_{k} = 1\}\), the segment duration
\(\xi_{k} = S_{k + 1} - S_{k}\). From the proof of Lemma~\ref{lem:5-5}(i),
\(\xi_{k} = D' + \widehat{T}'\), where
\(D' = D\left( \Theta_{S_{k}}X \right)\) and \(\widehat{T}'\)
is the first time after \(D'\) at which the longitudinal coordinate
of the increment path reaches \({\overline{M}}_{D} \vee L\) --- a
climb of \(\left( {\overline{M}}_{D} \vee L \right) + R\) from
level \(- R\). Both \(\phi_{k}\) and \(\xi_{k}\) are measurable
functionals of \(\Theta_{S_{k}}X\) alone, so \eqref{eq:5-incr} gives, on
\(\{ S_{k} < \infty\}\),

\[\mathbb{E}\left\lbrack e^{\theta\xi_{k}}\phi_{k}\,|\,\mathcal{F}_{S_{k}} \right\rbrack = \mathbb{E}\left\lbrack e^{\theta\xi};\, D < \infty \right\rbrack =: \left( 1 - p_{R} \right)\,\widetilde{m}(\theta),\quad\quad\widetilde{m}(\theta) = \mathbb{E}\left\lbrack e^{\theta\xi}\,|\, D < \infty \right\rbrack,\]

where \(\xi = D + \widehat{T}\) denotes the same functional of an
unshifted path.

Step 2 (\(\widetilde{m} \leq m\)). On \(\{ D < \infty\}\),
the strong Markov property at the stopping time \(D\) and Lemma~\ref{lem:5-12}(i)
give
\[\mathbb{E}\left\lbrack e^{\theta\widehat{T}} \mid \mathcal{F}_{D} \right\rbrack = e^{c_{\theta}\left( \left( {\overline{M}}_{D} \vee L \right) + R \right)} \leq e^{c_{\theta}(L + R)}e^{c_{\theta}{\overline{M}}_{D}}.\]
Hence, by Cauchy--Schwarz and Lemma~\ref{lem:5-12}(ii)--(iii) (applied at
\(2\theta \leq \nu^{2}/2\) and at \(\lambda = 2c_{\theta} < 2\nu\)),

\[\widetilde{m}(\theta) \leq e^{c_{\theta}(L + R)}\,\mathbb{E}\left\lbrack e^{\theta D}e^{c_{\theta}{\overline{M}}_{D}}\,|\, D < \infty \right\rbrack \leq e^{c_{\theta}(L + R)}\left( e^{Rc_{2\theta}} \right)^{1/2}\left( \frac{2\nu}{2\nu - 2c_{\theta}} \right)^{1/2} = m(\theta).\]

\emph{Step 3 (peeling).} Let
\(a_{k} = \mathbb{E}\left\lbrack e^{\theta S_{k}}\mathbf{1}_{\{ G_{k}\}} \right\rbrack\)
with \(G_{k} \in \mathcal{F}_{S_{k}}\) as in the proof of Theorem~\ref{thm:5-7};
\(a_{0} = \mathbb{E}\left\lbrack e^{\theta S_{0}} \right\rbrack = e^{c_{\theta}L}\)
by Lemma~\ref{lem:5-12}(i). Since \(e^{\theta S_{k}}\mathbf{1}_{\{ G_{k}\}}\) is
\(\mathcal{F}_{S_{k}}\)-measurable and
\(G_{k + 1} = G_{k} \cap \{\phi_{k} = 1\}\) up to a null set, Step~1
gives

\[a_{k + 1} = \mathbb{E}\left\lbrack e^{\theta S_{k}}\mathbf{1}_{\{ G_{k}\}}\, e^{\theta\xi_{k}}\phi_{k} \right\rbrack = \left( 1 - p_{R} \right)\,\widetilde{m}(\theta)\, a_{k},\quad\quad\mathbb{E}\left\lbrack e^{\theta S_{k}}\mathbf{1}_{\{ G_{k}\}}\left( 1 - \phi_{k} \right) \right\rbrack = p_{R}\, a_{k}.\]

By monotone convergence, using
\(\{ K = k\} = G_{k} \cap \{\phi_{k} = 0\}\) and
\[\left( 1 - p_{R} \right)\widetilde{m}(\theta) \leq \left( 1 - p_{R} \right)m(\theta) < 1,\]

\[\mathbb{E}\left\lbrack e^{\theta\tau_{1}} \right\rbrack = \sum_{k \geq 0}^{}\mathbb{E}\left\lbrack e^{\theta S_{k}}\mathbf{1}_{\{ G_{k}\}}\left( 1 - \phi_{k} \right) \right\rbrack = p_{R}\sum_{k \geq 0}^{}a_{k} = \frac{p_{R}\, e^{c_{\theta}L}}{1 - \left( 1 - p_{R} \right)\widetilde{m}(\theta)} \leq \frac{p_{R}\, e^{c_{\theta}L}}{1 - \left( 1 - p_{R} \right)m(\theta)}.\]

\emph{Step 4.}
\[\overline{\mathbb{E}}\left\lbrack e^{\theta\tau_{1}} \right\rbrack = p_{R}^{- 1}\mathbb{E}\left\lbrack e^{\theta\tau_{1}};D = \infty \right\rbrack \leq p_{R}^{- 1}\mathbb{E}\left\lbrack e^{\theta\tau_{1}} \right\rbrack\]
by Lemma~\ref{lem:5-13}. Finally
\(\bar{\mu} \geq \overline{\mathbb{E}}\left\lbrack S_{0} \right\rbrack > 0\)
and
\(\bar{\mu} \leq \theta^{- 1}\overline{\mathbb{E}}\left\lbrack e^{\theta\tau_{1}} \right\rbrack < \infty\).
\end{proof}

\begin{corollary}[block functionals; one-dependence]\label{cor:5-15}
~ There is
\(\theta_{1} = \theta_{1}(\nu,R,L) > 0\) such that, with
\(\Delta_{m} = \tau_{m + 1} - \tau_{m}\) and
\(\Xi_{m} = \sup_{\tau_{m} \leq s \leq \tau_{m + 1}} \lVert X_{s} - X_{\tau_{m}} \rVert\)
for \(m \geq 1\), and \(\Delta_{0} = \tau_{1}\),
\(\Xi_{0} = \sup_{s \leq \tau_{1}} \lVert X_{s} \rVert\):

(i)
\(\overline{\mathbb{E}}\left\lbrack e^{\theta_{1}\Delta} \right\rbrack < \infty\)
and
\(\overline{\mathbb{E}}\left\lbrack e^{\theta_{1}\Xi} \right\rbrack < \infty\);
likewise
\(\mathbb{E}\left\lbrack e^{\theta_{1}\Delta_{0}} \right\rbrack < \infty\)
and
\(\mathbb{E}\left\lbrack e^{\theta_{1}\Xi_{0}} \right\rbrack < \infty\);
and likewise for the two-block variables \(\tau_{m + 2} - \tau_{m}\) and
\(\sup_{\tau_{m} \leq s \leq \tau_{m + 2}} \lVert X_{s} - X_{\tau_{m}} \rVert\).

(ii) Consequently the diameters of block traces, the areas
\[\mathcal{H}^{2}\left( X\left( \left\lbrack \tau_{m},\tau_{m + 1} \right\rbrack \right)^{(r)} \right) \leq \pi\left( \Xi_{m} + r \right)^{2}\]
for \(r \leq r_{1}\), and --- via Theorem~\ref{thm:4-11} applied
pathwise,
\(\int_{r_{0}}^{r_{1}}\beta_{1}\, dr \leq 4050\,(2\Xi)^{4}r_{0}^{- 3}\)
--- the integrated hole counts of block and two-block sausages have
moments of all orders, under the block law and for the delayed block.

(iii) For every measurable \(f \geq 0\) of pairs of adjacent blocks, the
sequence \(\left( f\left( \mathcal{B}_{m},\mathcal{B}_{m + 1} \right) \right)_{m \geq 1}\)
is stationary and \(1\)-dependent.
\end{corollary}

\begin{proof}
(i) By Lemma~\ref{lem:5-13} it suffices to treat the delayed block under
\(\mathbb{P}\), at the cost \(p_{R}^{- 1}\). Pathwise
\(\Xi_{0} \leq \nu\tau_{1} + \sigma_{\tau_{1}}\) with
\(\sigma_{t} = \sup_{s \leq t} \lVert B_{s} \rVert\). For fixed
\(t\), splitting into coordinates, using Cauchy--Schwarz and the
reflection principle
(\(\sup_{s \leq t}B_{s}^{1}\overset{d}{=}\left| B_{t}^{1} \right|\)),

\[\mathbb{E}\, e^{\lambda\sigma_{t}} \leq \mathbb{E}\, e^{4\lambda\sup_{s \leq t}B_{s}^{1}} = \mathbb{E}\, e^{4\lambda\left| B_{t}^{1} \right|} \leq 2\, e^{8\lambda^{2}t}.\]

Hence, splitting on \(\{ n - 1 < \tau_{1} \leq n\}\), by Cauchy--Schwarz
and Proposition~\ref{prop:5-14} (with any fixed
\(\theta^{\dagger} \in \Theta^{*}\)),

\[\mathbb{E}\, e^{\lambda\sigma_{\tau_{1}}} \leq \sum_{n \geq 1}^{}\left( \mathbb{E}e^{2\lambda\sigma_{n}} \right)^{1/2}\mathbb{P}\left( \tau_{1} > n - 1 \right)^{1/2} \leq \sqrt{2}\, e^{\theta^{\dagger}/2}\left( \mathbb{E}e^{\theta^{\dagger}\tau_{1}} \right)^{1/2}\sum_{n \geq 1}^{}e^{\left( 16\lambda^{2} - \theta^{\dagger}/2 \right)n} < \infty\]
for \(\lambda\) with \(16\lambda^{2} < \theta^{\dagger}/2\); a final
Cauchy--Schwarz on
\(e^{\theta_{1}\left( \nu\tau_{1} + \sigma_{\tau_{1}} \right)}\) gives
the claim for small \(\theta_{1}\). The two-block variables are dominated by
sums of the corresponding variables of two independent blocks
(Corollary~\ref{cor:5-8}), and Laplace transforms multiply. (ii) is immediate
from (i), \(\operatorname{diam} \leq 2\Xi\) and the pathwise bound quoted;
(iii) from the i.i.d.~structure of \(\left( \mathcal{B}_{m} \right)_{m \geq 1}\).
\end{proof}

\subsection{Delayed-renewal counting}\label{sec:5-4}

The counting process attached to these blocks is a delayed renewal process in
the sense of \cite[Chapter V]{ref1}; we record the estimates used later.

\begin{proposition}[renewal counting]\label{prop:5-16}
~ Set \(\tau_{0} = 0\) and
\(N_{T} = \max\{ m \geq 0:\ \tau_{m} \leq T\}\) for \(T \geq 0\). Then:

(i) \(N_{T} < \infty\) a.s., and \(N_{T}/T \rightarrow 1/\bar{\mu}\)
a.s.~as \(T \rightarrow \infty\).

(ii) With
\(\lambda_{0} = \bar{\mu}/\overline{\mathbb{E}}\left\lbrack \tau_{1}^{2} \right\rbrack \in (0,\infty)\):
for every \(n\) with \(n - 1 \geq 4T/\bar{\mu}\),

\[\mathbb{P}\left( N_{T} \geq n \right)\  \leq \ e^{- \lambda_{0}\bar{\mu}\,(n - 1)/4}.\]

(iii) For every \(q \geq 1\) there is \(C_{q} = C_{q}(\nu,R,L)\) with
\(\mathbb{E}\left\lbrack N_{T}^{\, q} \right\rbrack \leq C_{q}\left( 1 + T^{q} \right)\),
and \(N_{T}/T \rightarrow 1/\bar{\mu}\) in \(L^{q}\).
\end{proposition}

\begin{proof}
(i) By Corollary~\ref{cor:5-8} and the strong law of large numbers,
\(\tau_{n}/n = \tau_{1}/n + n^{- 1}\sum_{m = 1}^{n - 1}\Delta_{m} \rightarrow \bar{\mu}\)
a.s.~(\(\overline{\mathbb{E}}\Delta = \bar{\mu} < \infty\) by
Proposition~\ref{prop:5-14}); the standard inversion
\(\{ N_{T} \geq n\} = \{\tau_{n} \leq T\}\) gives \(N_{T} < \infty\)
a.s.~and \(N_{T}/T \rightarrow 1/\bar{\mu}\).

(ii) For every \(\lambda > 0\), by Corollary~\ref{cor:5-8} and Chebyshev,
\[\mathbb{P}\left( N_{T} \geq n \right) = \mathbb{P}\left( \tau_{n} \leq T \right) \leq \mathbb{P}\left( \sum_{m = 1}^{n - 1}\Delta_{m} \leq T \right) \leq e^{\lambda T}\rho_{\lambda}^{\, n - 1}\]
with
\(\rho_{\lambda} = \overline{\mathbb{E}}\left\lbrack e^{- \lambda\tau_{1}} \right\rbrack\).
Since \(e^{- x} \leq 1 - x + x^{2}/2\) for \(x \geq 0\),
\[\rho_{\lambda} \leq 1 - \lambda\bar{\mu} + \lambda^{2}\overline{\mathbb{E}}\left\lbrack \tau_{1}^{2} \right\rbrack/2 \leq 1 - \lambda\bar{\mu}/2 \leq e^{- \lambda\bar{\mu}/2}\]
for \(0 < \lambda \leq \lambda_{0}\) (note
\(\lambda_{0}\bar{\mu} = {\bar{\mu}}^{2}/\overline{\mathbb{E}}\left\lbrack \tau_{1}^{2} \right\rbrack \leq 1\)
by Cauchy--Schwarz). Take \(\lambda = \lambda_{0}\): for
\(n - 1 \geq 4T/\bar{\mu}\) one has
\(\lambda_{0}T \leq \lambda_{0}\bar{\mu}(n - 1)/4\), whence
\(e^{\lambda_{0}T}\rho_{\lambda_{0}}^{\, n - 1} \leq e^{- \lambda_{0}\bar{\mu}(n - 1)/4}\).

(iii)
\[\mathbb{E}\left\lbrack N_{T}^{q} \right\rbrack = \sum_{n \geq 1}^{}\left( n^{q} - (n - 1)^{q} \right)\mathbb{P}\left( N_{T} \geq n \right) \leq \left( 2 + 4T/\bar{\mu} \right)^{q} + \sum_{n \geq 1}^{}q\, n^{q - 1}e^{- \lambda_{0}\bar{\mu}(n - 1)/4} \leq C_{q}\left( 1 + T^{q} \right).\]
The \(L^{q}\) convergence follows from (i) and the uniform integrability
of \(\left( N_{T}/T \right)^{q}\), \(T \geq 1\), supplied by the moment
bound at exponent \(q + 1\). \end{proof}

\section{Law of large numbers for smoothed Betti-curve functionals}\label{sec:6}

We combine the persistence formalism of Section~\ref{sec:2} with the
regeneration structure of Section~\ref{sec:5}. The main difficulty is that
persistence is not additive under concatenation of path pieces. The remedy has
two parts. First, one rewrites the persistence functional in terms of the Betti
curve \(r \mapsto \beta_{1}( \cdot \,;r)\), so that one only has to control
topology at a fixed radius. Second, one uses the regeneration cuts to localize
all non-additivity to narrow interface slabs, whose contribution has finite
moments of every order.

The architecture of the proof is short enough to state in advance, and the
reader may find it useful to keep in view:
\begin{center}
independent blocks \(\Longrightarrow\) local interface
\(\Longrightarrow\) one-dependent rewards \(\Longrightarrow\) LLN and CLT.
\end{center}
The four stages are Corollary~\ref{cor:5-8}; Lemmas~\ref{lem:6-1} and
\ref{lem:6-6}; Lemma~\ref{lem:6-9}; and Sections~\ref{sec:6}--\ref{sec:7}
respectively. The first is probabilistic and is the content of
Section~\ref{sec:5}. The second is geometric: the drift separates the blocks
along its own direction, so two blocks meet only inside a slab of bounded
width, and non-adjacent blocks do not meet at all. The third converts the
second into the statement that the increment contributed by a cycle is a
functional of two consecutive blocks and of nothing else, which is exactly the
\(1\)-dependence the classical limit theorems require. The fourth is standard
machinery.

Throughout, \(X_{t}\) is the planar Brownian motion with non-zero drift
\(\mu\) of Section~\ref{sec:5}, \(\nu = \lVert \mu \rVert\),
\(e = \mu/\nu\), \(U_{t} = \langle X_{t},\, e\rangle\), and
\(\left( \tau_{k} \right)_{k \geq 0}\) are the regeneration times of
Definition~\ref{def:5-4}, with \(\tau_{0} = 0\). Fix a compact
radius window \(0 < r_{0} < r_{1} < \infty\) and a bounded Borel weight
\(\psi\) supported in \(\left\lbrack r_{0},r_{1} \right\rbrack\). Recall
from Section~\ref{sec:2-5} that for
\(\phi_{\psi}(b,d) = \int_{b}^{d}\psi(r)\, dr\),

\[\Phi_{\psi}(T) = \int_{r_{0}}^{r_{1}}\beta_{1}\left( \mathcal{S}_{T}(r) \right)\,\psi(r)\, dr,\quad\quad\mathcal{S}_{T}(r) = X\left( \lbrack 0,T\rbrack \right)^{(r)}.\]

For a compact \(A \subset \mathbb{R}^{2}\) we also write
\(\Phi_{\psi}(A) = \int_{r_{0}}^{r_{1}}\beta_{1}A^{(r)}\psi(r)\, dr\).

\subsection{Block sausages and interface slabs}\label{sec:6-1}

Fix the regeneration parameters \(L > 0\) and \(R > 0\) of Section~\ref{sec:5} and
a constant \(a\) with

\[a > R + 2r_{1},\quad\quad L > 2a.\]

For \(k \geq 1\) define the \(k\)-th regeneration range and its
\(r\)-sausage by

\[K_{k}:= X\left( \left\lbrack \tau_{k - 1},\tau_{k} \right\rbrack \right),\quad\quad S_{k}(r):= K_{k}^{(r)}\quad\left( r \in \left\lbrack r_{0},r_{1} \right\rbrack \right),\]

and for \(n \geq 1\) put
\(\mathcal{U}_{n}(r):= \bigcup_{k = 1}^{n}S_{k}(r) = X\left( \left\lbrack 0,\tau_{n} \right\rbrack \right)^{(r)}\).
Write \(\ell_{k}:= U_{\tau_{k}}\) for the \(k\)-th cut level,
\(D_{k}:= \operatorname{diam}\left( K_{k} \right)\), and

\[\Sigma_{k}:= \{ x \in \mathbb{R}^{2}:\ \left| \langle x,\, e\rangle - \ell_{k} \right| \leq a\}\]
for the \(k\)-th interface slab. We adopt throughout the convention
\[K_{0}:= \left\{ X_{0} \right\},\qquad D_{0}:= 0,\qquad \mathcal{U}_{0}:= \varnothing,\]
so that every expression below indexed by \(N_{T}\) is defined also on the event
\(\left\{ N_{T} = 0 \right\} = \left\{ T < \tau_{1} \right\}\), and no
statement needs a separate pre-first-regeneration case. Throughout we use the
good-cut property
Lemma~\ref{lem:5-9}: almost surely, for every \(k \geq 1\),

\[\langle x,\, e\rangle \leq \ell_{k}\ \ \left( x \in X\left( \left\lbrack 0,\tau_{k} \right\rbrack \right) \right),\quad\quad\langle x,\, e\rangle > \ell_{k} - R\ \ \left( x \in X\left( \lbrack\tau_{k},\infty \right) \right)),\quad\quad\ell_{k + 1} \geq \ell_{k} + L.\]

\begin{lemma}[localization of overlaps]\label{lem:6-1}
~ For every
\(r \in \left\lbrack r_{0},r_{1} \right\rbrack\) and every \(k \geq 1\),

\[S_{k}(r) \cap S_{k + 1}(r) \subseteq \Sigma_{k},\]

and \(S_{k}(r) \cap S_{j}(r) = \varnothing\) whenever
\(|j - k| \geq 2\).
\end{lemma}

\begin{proof}
Let \(x \in S_{k}(r)\), so
\(\lVert x - y \rVert \leq r\) for some \(y \in K_{k}\); since
\(\langle y,\, e\rangle \leq \ell_{k}\), we get
\(\langle x,\, e\rangle \leq \ell_{k} + r_{1} < \ell_{k} + a\). Let
\(x \in S_{k + 1}(r)\), so \(\lVert x - z \rVert \leq r\) for some
\(z \in K_{k + 1} \subseteq X\left( \lbrack\tau_{k},\infty \right))\);
then \(\langle x,\, e\rangle \geq \ell_{k} - R - r_{1} > \ell_{k} - a\).
This is the first claim. If \(j \geq k + 2\), every point of \(K_{j}\)
lies after \(\tau_{k + 1}\), hence at level
\(> \ell_{k + 1} - R \geq \ell_{k} + L - R\), so every point of
\(S_{j}(r)\) is at level \(> \ell_{k} + L - R - r_{1}\), while every
point of \(S_{k}(r)\) is at level \(\leq \ell_{k} + r_{1}\). Since
\(L > 2a > 2\left( R + 2r_{1} \right)\) we have \(L - R - 2r_{1} > 0\),
and the two sets are disjoint. \end{proof}

Only Lemma~\ref{lem:6-1} uses \(a\), and only through \(a > R + 2r_{1}\) and
\(L > 2a\); no constant below depends on the choice.

\subsection{The integrated interface bound}\label{sec:6-2}

The deterministic input is the Mayer--Vietoris bound of
Lemma~\ref{lem:6-2}, applied at the radii at which its finiteness hypothesis
\eqref{eq:6-2-fin} holds; Theorem~\ref{thm:4-11} supplies it at almost every
radius, which is all that the radius integral defining \(\Phi_{\psi}\) sees.

\begin{corollary}[integrated interface bound]\label{cor:6-3}
~ For compact
\(A,B \subset \mathbb{R}^{2}\) set

\[I_{|\psi|}(A,B):= \int_{r_{0}}^{r_{1}}\left( \beta_{1}\left( A^{(r)} \cap B^{(r)} \right) + \beta_{0}\left( A^{(r)} \cap B^{(r)} \right) \right)|\psi(r)|\, dr.\]

Then
\(\left| \Phi_{\psi}(A \cup B) - \Phi_{\psi}(A) - \Phi_{\psi}(B) \right| \leq I_{|\psi|}(A,B)\).
\end{corollary}

\begin{proof}
Since \(A\) and \(B\) are compact of finite diameter, Theorem~\ref{thm:4-11}
gives
\(\int_{r_{0}}^{r_{1}}\bigl( \beta_{1}A^{(r)} + \beta_{1}B^{(r)} \bigr)dr < \infty\),
so \(\beta_{1}A^{(r)}\) and \(\beta_{1}B^{(r)}\) are finite for Lebesgue-almost
every \(r \in \left\lbrack r_{0},r_{1} \right\rbrack\). At each such \(r\),
Lemma~\ref{lem:6-2} applies to \(A^{(r)}\) and \(B^{(r)}\), and its two-sided
bound is majorised in absolute value by
\[\beta_{1}\left( A^{(r)} \cap B^{(r)} \right) + \beta_{0}\left( A^{(r)} \cap B^{(r)} \right).\]
Multiplying by \(|\psi(r)| \geq 0\) and integrating --- the exceptional radii
forming a null set and contributing nothing --- gives the claim for a weight of
arbitrary sign. Note
\(\left( A \cup B \right)^{(r)} = A^{(r)} \cup B^{(r)}\). \end{proof}

\begin{proposition}[deterministic almost-additivity at regeneration times]\label{prop:6-4}
~ For every \(n \geq 1\),

\[\left| \Phi_{\psi}\left( \tau_{n} \right) - \sum_{k = 1}^{n}J_{k} \right|\  \leq \ \sum_{k = 1}^{n - 1}I_{k},\quad\quad J_{k}:= \Phi_{\psi}\left( K_{k} \right),\quad I_{k}:= I_{|\psi|}\left( K_{k},K_{k + 1} \right).\]
\end{proposition}

\begin{proof}
By Lemma~\ref{lem:6-1} only adjacent block sausages meet, so when
the blocks are added one at a time the only interface created at step
\(k + 1\) is between \(\bigcup_{j \leq k}S_{j}(r)\) and
\(S_{k + 1}(r)\), and by Lemma~\ref{lem:6-1} again that intersection equals
\(S_{k}(r) \cap S_{k + 1}(r)\). Iterating Corollary~\ref{cor:6-3} gives the claim.
\end{proof}

\subsection{Complexity inputs}\label{sec:6-3}

We record the consequences of Theorem~\ref{thm:4-11}, applied pathwise to the
regeneration blocks, in the form in which they are used.

\begin{proposition}[block complexity]\label{prop:6-5}
~ Let
\(\theta_{1} > 0\) be as in Corollary~\ref{cor:5-15}. Then
\(\mathbb{E}\left\lbrack e^{\theta_{1}D_{k}} \right\rbrack < \infty\)
for every \(k \geq 1\), with a bound uniform in \(k\), and consequently,
for every \(q \geq 1\),

\[\sup_{k \geq 1}\ \mathbb{E}\left\lbrack \left| J_{k} \right|^{\, q} \right\rbrack < \infty,\quad\quad\sup_{k \geq 1}\ \mathbb{E}\left\lbrack \left( \int_{r_{0}}^{r_{1}}\beta_{1}\left( \left( K_{k} \cup K_{k + 1} \right)^{(r)} \right)dr \right)^{q} \right\rbrack < \infty.\]

Moreover, pathwise, for every
\(t \in \left\lbrack \tau_{k - 1},\tau_{k} \right\rbrack\),

\[\left| \Phi_{\psi}\left( X\left( \left\lbrack \tau_{k - 1},t \right\rbrack \right) \right) \right|\  \leq \  \lVert \psi \rVert_{\infty}\, 4050\, r_{0}^{- 3}D_{k}^{4}.\]
\end{proposition}

\begin{proof}
The exponential moments of
\(D_{k} = \operatorname{diam}\left( K_{k} \right) \leq 2\Xi_{k - 1}\) are Corollary~\ref{cor:5-15}(i), which covers the delayed block as well. Since
\(K_{k}\) is a compact connected set (a continuous path image) of
diameter \(D_{k}\), Theorem~\ref{thm:4-11} gives
\[\left| J_{k} \right| \leq \lVert \psi \rVert_{\infty}\int_{r_{0}}^{r_{1}}\beta_{1}K_{k}^{(r)}dr \leq \lVert \psi \rVert_{\infty}4050\, r_{0}^{- 3}D_{k}^{4},\]
and \(q\)-th moments follow from the exponential moments of \(D_{k}\).
For the second display, \(K_{k} \cup K_{k + 1}\) is compact connected
(the two ranges share \(X_{\tau_{k}}\)) of diameter at most
\(D_{k} + D_{k + 1}\), and
\((u + v)^{4} \leq 8\left( u^{4} + v^{4} \right)\). The pathwise bound
is Theorem~\ref{thm:4-11} applied to
\(X\left( \left\lbrack \tau_{k - 1},t \right\rbrack \right) \subseteq K_{k}\),
a compact connected set of diameter at most \(D_{k}\). \end{proof}

\subsection{The Boundary Lemma}\label{sec:6-4}

For \(k \geq 1\) define the slab portions of two consecutive blocks,

\[P_{k}:= K_{k} \cap \{ x:\ \langle x,\, e\rangle > \ell_{k} - R - 2r_{1}\},\quad\quad Q_{k}:= K_{k + 1} \cap \{ x:\ \langle x,\, e\rangle \leq \ell_{k} + 2r_{1}\}.\]

\begin{lemma}[localization of the overlap by level]\label{lem:6-6}
For every \(k \geq 1\) and
\(r \in \left\lbrack r_{0},r_{1} \right\rbrack\),

\[S_{k}(r) \cap S_{k + 1}(r) = P_{k}^{(r)} \cap Q_{k}^{(r)},\]

and, for every \(t \in (\tau_{k},\tau_{k + 1}\rbrack\), writing
\(\mathcal{S}_{\tau_{k}}(r) = X\left( \left\lbrack 0,\tau_{k} \right\rbrack \right)^{(r)}\)
and \(K_{k,t} = X\left( \left\lbrack \tau_{k},t \right\rbrack \right)\),

\[\mathcal{S}_{\tau_{k}}(r) \cap K_{k,t}^{(r)} = S_{k}(r) \cap K_{k,t}^{(r)}.\]

In particular \(I_{k}\) is a measurable functional of the pair
\(\left( K_{k},K_{k + 1} \right)\) alone.
\end{lemma}

\begin{proof}
\(\supseteq\) is clear from \(P_{k} \subseteq K_{k}\),
\(Q_{k} \subseteq K_{k + 1}\). For \(\subseteq\), let
\(x \in S_{k}(r) \cap S_{k + 1}(r)\) and pick \(y \in K_{k}\),
\(z \in K_{k + 1}\) with
\(\lVert x - y \rVert \vee \lVert x - z \rVert \leq r\). By
the good-cut property \(\langle z,\, e\rangle > \ell_{k} - R\), so
\(\langle y,\, e\rangle \geq \langle z,\, e\rangle - 2r > \ell_{k} - R - 2r_{1}\)
and \(y \in P_{k}\); and \(\langle y,\, e\rangle \leq \ell_{k}\), so
\(\langle z,\, e\rangle \leq \langle y,\, e\rangle + 2r \leq \ell_{k} + 2r_{1}\)
and \(z \in Q_{k}\). Hence
\(x \in P_{k}^{(r)} \cap Q_{k}^{(r)}\).

For the second identity, let \(x\) belong to the left-hand side, with
\(y \in X\left( \left\lbrack 0,\tau_{k} \right\rbrack \right)\) and
\(z \in K_{k,t} \subseteq X\left( \lbrack\tau_{k},\infty \right))\) at
distance at most \(r\) from \(x\). The same computation gives
\(\langle y,\, e\rangle > \ell_{k} - R - 2r_{1}\). Every point of
\(X\left( \left\lbrack 0,\tau_{k - 1} \right\rbrack \right)\) has level
at most \(\ell_{k - 1} \leq \ell_{k} - L < \ell_{k} - R - 2r_{1}\),
because \(L > 2a > 2\left( R + 2r_{1} \right)\); hence \(y \in K_{k}\)
and \(x \in S_{k}(r)\). The reverse inclusion is trivial. The final
assertion follows since \(P_{k}\), \(Q_{k}\) and \(\ell_{k}\) are
determined by \(K_{k}\) and \(K_{k + 1}\). \end{proof}

\begin{theorem}[Boundary Lemma]\label{thm:6-8}
~ For every \(k \geq 1\) and every
\(r \in \left\lbrack r_{0},r_{1} \right\rbrack\),

\[\beta_{1}\left( S_{k}(r) \cap S_{k + 1}(r) \right) + \beta_{0}\left( S_{k}(r) \cap S_{k + 1}(r) \right) \leq \beta_{1}\left( S_{k}(r) \right) + \beta_{1}\left( S_{k + 1}(r) \right) + \beta_{1}\left( S_{k}(r) \cup S_{k + 1}(r) \right) + 1,\]

hence, pathwise,

\[I_{k}\  \leq \  \lVert \psi \rVert_{\infty}\left\lbrack \left( r_{1} - r_{0} \right) + 36450\, r_{0}^{- 3}\left( D_{k}^{4} + D_{k + 1}^{4} \right) \right\rbrack.\]

Consequently
\(\sup_{k \geq 1}\mathbb{E}\left\lbrack I_{k}^{\, q} \right\rbrack < \infty\)
for every \(q \geq 1\); the sequence \(\left( I_{k} \right)_{k \geq 2}\)
is stationary and \(1\)-dependent; and
\(\mathbb{E}\left\lbrack I_{1} \right\rbrack < \infty\).
\end{theorem}

\begin{proof}
\(S_{k}(r)\) and \(S_{k + 1}(r)\) are compact connected
(offsets of continuous path images) and meet at \(X_{\tau_{k}}\), so
Lemma~\ref{lem:6-7} applies to them directly and gives the first display.
Integrating against \(\psi\) and applying Theorem~\ref{thm:4-11} to \(K_{k}\),
\(K_{k + 1}\) and \(K_{k} \cup K_{k + 1}\) --- compact connected of
diameters \(D_{k}\), \(D_{k + 1}\) and at most \(D_{k} + D_{k + 1}\) ---
yields

\[I_{k} \leq \lVert \psi \rVert_{\infty}\left\lbrack \left( r_{1} - r_{0} \right) + 4050\, r_{0}^{- 3}\left( D_{k}^{4} + D_{k + 1}^{4} + \left( D_{k} + D_{k + 1} \right)^{4} \right) \right\rbrack,\]

and \((u + v)^{4} \leq 8\left( u^{4} + v^{4} \right)\) gives the stated
constant \(4050 \cdot 9 = 36450\). Moments of all orders follow from
Proposition~\ref{prop:6-5}. By Lemma~\ref{lem:6-6}, \(I_{k}\) is a
functional of \(\left( K_{k},K_{k + 1} \right)\); by Corollary~\ref{cor:5-8} the blocks \(\left( K_{k} \right)_{k \geq 2}\) are
i.i.d.~under \(\overline{\mathbb{P}}\), so
\(\left( I_{k} \right)_{k \geq 2}\) is stationary and \(I_{k}\) is
independent of \(I_{j}\) whenever \(|j - k| \geq 2\). The delayed block
\(K_{1}\) has exponential diameter moments as well, whence
\(\mathbb{E}\left\lbrack I_{1} \right\rbrack < \infty\). \end{proof}

\subsection{Regeneration-time increments and the LLN}\label{sec:6-5}

Define the increments of the smoothed persistence functional along
regeneration times by

\[Z_{1}:= \Phi_{\psi}\left( \tau_{1} \right),\quad\quad Z_{k}:= \Phi_{\psi}\left( \tau_{k} \right) - \Phi_{\psi}\left( \tau_{k - 1} \right),\quad k \geq 2,\]

and let \(\mathcal{A}_{k}:= \sigma\left( K_{k} \right)\) be the
\(\sigma\)-field of the \(k\)-th block.

\begin{lemma}[locality of the regeneration-time increments]\label{lem:6-9}
~ For
every \(k \geq 2\) there is a measurable \(F\), not depending on \(k\),
with

\[Z_{k} = F\left( K_{k - 1},K_{k} \right) = \int_{r_{0}}^{r_{1}}\left( \beta_{1}\left( S_{k - 1}(r) \cup S_{k}(r) \right) - \beta_{1}\left( S_{k - 1}(r) \right) \right)\psi(r)\, dr;\]

in particular \(Z_{k} \in \mathcal{A}_{k - 1} \vee \mathcal{A}_{k}\),
and \(\left( Z_{k} \right)_{k \geq 3}\) is stationary and
\(1\)-dependent. (Measurability of \(F\) is Lemma~\ref{lem:4-meas}.)
\end{lemma}

\begin{proof}
Fix \(k \geq 2\) and
\(r \in \left\lbrack r_{0},r_{1} \right\rbrack\), and set
\(A:= \mathcal{U}_{k - 2}(r)\), \(B:= S_{k - 1}(r)\),
\(C:= S_{k}(r)\) (with \(A = \varnothing\) if \(k = 2\)). By Lemma~\ref{lem:6-1},
\(A \cap C = \varnothing\).

The Euler-characteristic computation that follows subtracts integers, so it
requires the Betti numbers involved to be genuinely finite; in the extended
framework of Remark~\ref{rem:4-1} the operation \(\infty - \infty\) is not
available. We therefore carry it out only at those radii at which the finitely
many sets \(A,B,C,A \cup B,B \cup C,A \cup B \cup C\) --- all compact and
connected, and all offsets of fixed compact sets of finite diameter, namely of
\(K_{1} \cup \cdots \cup K_{j}\), \(K_{j}\) and \(K_{j - 1} \cup K_{j}\) for
\(j \leq k\) --- have finite Betti numbers; at such a radius the intersection
\(A \cap B\), which is not an offset, has finite Betti numbers as well, by
Lemma~\ref{lem:6-7} (\(\beta_{1}(A \cap B) \leq \beta_{1}(A) + \beta_{1}(B)\)
and \(\beta_{0}(A \cap B) \leq \beta_{1}(A \cup B) + 1\)), so every Euler
characteristic below is a finite integer. By Theorem~\ref{thm:4-11} each of the
radius integrals of the six offsets is finite pathwise, so this holds for
Lebesgue-almost every \(r \in \left\lbrack r_{0},r_{1} \right\rbrack\), which
is all that is needed:
the quantity being represented is an absolutely convergent radius integral, and
two integrands agreeing almost everywhere have the same integral. Fix such an
\(r\). Additivity of the Euler characteristic --- the alternating sum of
dimensions along the Mayer--Vietoris sequence of Remark~\ref{rem:4-1}, all of
whose terms are now finite-dimensional --- gives
\(\chi(A \cup B) = \chi(A) + \chi(B) - \chi(A \cap B)\) and, using
\(A \cap C = \varnothing\) so that \(A \cap (B \cup C) = A \cap B\),
\(\chi(A \cup B \cup C) = \chi(A) + \chi(B \cup C) - \chi(A \cap B)\).
Subtracting,

\[\chi\left( \mathcal{U}_{k}(r) \right) - \chi\left( \mathcal{U}_{k - 1}(r) \right) = \chi(B \cup C) - \chi(B).\]

All sets being connected and of finite Betti number at the radius fixed above,
\(\beta_{1}(Y) = 1 - \chi(Y)\), so
\begin{equation}\label{eq:6-9-pointwise}
\beta_{1}\left( \mathcal{U}_{k}(r) \right) - \beta_{1}\left( \mathcal{U}_{k - 1}(r) \right) = \beta_{1}\left( S_{k - 1}(r) \cup S_{k}(r) \right) - \beta_{1}\left( S_{k - 1}(r) \right)
\end{equation}
for almost every \(r\). Multiplying by \(\psi(r)\) and integrating gives the
representation, both integrals being absolutely convergent by
Theorem~\ref{thm:4-11}.

The identity is one at the level of the functional
\(F(A,B):= \Phi_{\psi}(A \cup B) - \Phi_{\psi}(A)\), which is defined for
every pair of compact sets with no finiteness proviso, \(\Phi_{\psi}\) being
an absolutely convergent radius integral by Theorem~\ref{thm:4-11}; the
pointwise identity \eqref{eq:6-9-pointwise} is a step in the proof, claimed
only almost everywhere in \(r\), and is used again in that form in
Lemma~\ref{lem:8-2}.

Since the blocks \(\left( K_{k} \right)_{k \geq 2}\) are
i.i.d.~(Corollary~\ref{cor:5-8}), the two-block factor
\(\left( Z_{k} \right)_{k \geq 3}\) is stationary, and \(Z_{k}\) is
independent of \(Z_{j}\) for \(|j - k| \geq 2\) because the pairs
\(\left( K_{k - 1},K_{k} \right)\) and
\(\left( K_{j - 1},K_{j} \right)\) are then functions of disjoint sets
of i.i.d.~blocks. \end{proof}

\begin{proposition}[integrability of the cycle increments and cycle oscillations]\label{prop:6-10}
~ For every \(q \geq 1\),
\(\sup_{k \geq 2}\mathbb{E}\left\lbrack \left| Z_{k} \right|^{q} \right\rbrack < \infty\).
Moreover, if for \(k \geq 2\)

\[M_{k}:= \sup_{\tau_{k - 1} \leq t \leq \tau_{k}}\left| \Phi_{\psi}(t) - \Phi_{\psi}\left( \tau_{k - 1} \right) \right|,\]

then
\(\sup_{k \geq 2}\mathbb{E}\left\lbrack M_{k}^{\, q} \right\rbrack < \infty\)
for every \(q \geq 1\), and \(\left( M_{k} \right)_{k \geq 3}\) is
stationary and \(1\)-dependent.
\end{proposition}

\begin{proof}
By Lemma~\ref{lem:6-9} and Theorem~\ref{thm:4-11}, pathwise

\[\left| Z_{k} \right| \leq \lVert \psi \rVert_{\infty}\int_{r_{0}}^{r_{1}}\left( \beta_{1}\left( \left( K_{k - 1} \cup K_{k} \right)^{(r)} \right) + \beta_{1}K_{k - 1}^{(r)} \right)dr \leq \lVert \psi \rVert_{\infty}\, 4050\, r_{0}^{- 3}\left( 9D_{k - 1}^{4} + 8D_{k}^{4} \right),\]

using \((u + v)^{4} \leq 8\left( u^{4} + v^{4} \right)\);
Proposition~\ref{prop:6-5} gives the moments. For \(M_{k}\), repeat the
Euler-characteristic computation of Lemma~\ref{lem:6-9} with
\(K_{k,t}:= X\left( \left\lbrack \tau_{k - 1},t \right\rbrack \right)\)
in place of \(K_{k}\): for
\(t \in \left\lbrack \tau_{k - 1},\tau_{k} \right\rbrack\),
\(\mathcal{U}_{k - 2}(r)\) and \(K_{k,t}^{(r)}\) are disjoint by
Lemma~\ref{lem:6-1} (as \(K_{k,t} \subseteq K_{k}\)), so

\[\Phi_{\psi}(t) - \Phi_{\psi}\left( \tau_{k - 1} \right) = \int_{r_{0}}^{r_{1}}\left( \beta_{1}\left( \left( K_{k - 1} \cup K_{k,t} \right)^{(r)} \right) - \beta_{1}K_{k - 1}^{(r)} \right)\psi(r)\, dr.\]

Since
\(\operatorname{diam}\left( K_{k - 1} \cup K_{k,t} \right) \leq D_{k - 1} + D_{k}\)
uniformly in \(t\), Theorem~\ref{thm:4-11} gives the same pathwise bound as
for \(\left| Z_{k} \right|\), now uniform in \(t\); hence
\(M_{k} \leq \lVert \psi \rVert_{\infty}4050\, r_{0}^{- 3}\left( 9D_{k - 1}^{4} + 8D_{k}^{4} \right)\)
and the moments follow. Stationarity and \(1\)-dependence of
\(\left( M_{k} \right)_{k \geq 3}\) follow as in Lemma~\ref{lem:6-9}, \(M_{k}\)
being a functional of \(\left( K_{k - 1},K_{k} \right)\). \end{proof}

\begin{lemma}[uniform renewal-window moments]\label{lem:6-11a}
~ Let
\(N_{T} = \max\{ n \geq 0:\tau_{n} \leq T\}\),
\(A_{T} = T - \tau_{N_{T}}\), and let
\(U( \cdot ) = \sum_{n \geq 0}^{}\mathbb{P}\left( \tau_{n} \in \cdot \right)\)
be the renewal measure. Then
\(c_{U}:= \sup_{u \geq 0}U\left( (u,u + 1\rbrack \right) < \infty\),
and, for every measurable \(G \geq 0\) on pairs of blocks satisfying both
\begin{equation}\label{eq:6-11a-hyp}
{\overline{\mathbb{E}}}^{\otimes 2}\left\lbrack G^{2} \right\rbrack < \infty
\qquad\text{and}\qquad
\mathbb{E}\left\lbrack G\left( K_{0},K_{1} \right) \right\rbrack + \mathbb{E}\left\lbrack G\left( K_{1},K_{2} \right) \right\rbrack < \infty,
\end{equation}
one has
\[\sup_{T \geq 0}\ \mathbb{E}\left\lbrack G\left( K_{N_{T}},K_{N_{T} + 1} \right) \right\rbrack < \infty.\]

The second condition in \eqref{eq:6-11a-hyp} concerns the two indices
\(n \in \{ 0,1\}\) at which the pair involves the delayed initial block, whose
law is not \({\overline{\mathbb{P}}}^{\otimes 2}\), and must be assumed
separately. Both conditions hold for the two functionals used below,
\(G\left( K,K' \right) = \operatorname{diam}(K)^{q} + \operatorname{diam}\left( K' \right)^{q}\)
and \(G\left( K,K' \right) = \mathbf{1}_{\left\{ \Delta\left( K' \right) > s \right\}}\),
by Corollary~\ref{cor:5-15} and by the exponential moments of
Proposition~\ref{prop:5-14}, which cover the first block. Consequently
\[\sup_{T}\mathbb{E}\left\lbrack D_{N_{T}}^{\, q} + D_{N_{T} + 1}^{\, q} \right\rbrack < \infty
\qquad\text{and}\qquad
\sup_{T}\mathbb{E}\left\lbrack A_{T}^{\, q} \right\rbrack < \infty
\qquad (q \geq 1),\]
which is the only form in which the lemma is applied.
\end{lemma}

\begin{proof}
Fix \(u \geq 0\) and let
\(\eta = \inf\{ n:\tau_{n} > u\}\), so
\(\{\eta = n\} \in \mathcal{G}_{n}\) --- the cumulative block filtration of
Corollary~\ref{cor:5-8}, not the single-block field \(\mathcal{A}_{n}\) above ---
and, on that event, the blocks
after index \(n\) are i.i.d.~\(\overline{\mathbb{P}}\) and independent
of the past (Corollary~\ref{cor:5-8}). Any renewal in
\((u,u + 1\rbrack\) other than \(\tau_{\eta}\) requires the
corresponding partial sum of subsequent block lengths to be at most
\(1\), so with
\(\rho_{\lambda} = \overline{\mathbb{E}}\left\lbrack e^{- \lambda\tau_{1}} \right\rbrack < 1\),

\[U\left( (u,u + 1\rbrack \right) \leq 1 + \sum_{j \geq 1}^{}{\overline{\mathbb{P}}}^{\otimes j}\left( \sum_{i \leq j}^{}\Delta_{i} \leq 1 \right) \leq 1 + e^{\lambda}\sum_{j \geq 1}^{}\rho_{\lambda}^{\, j} < \infty,\]

uniformly in \(u\), exactly as in Proposition~\ref{prop:5-16}(ii).
Next, for \(n \geq 2\) the pair \(\left( K_{n},K_{n + 1} \right)\) is
independent of
\(\sigma\left( K_{1},\ldots,K_{n - 1} \right) \ni \tau_{n - 1}\) with
law \({\overline{\mathbb{P}}}^{\otimes 2}\), so

\[\mathbb{E}\left\lbrack G\left( K_{n},K_{n + 1} \right);N_{T} = n \right\rbrack = \mathbb{E}\left\lbrack H\left( T - \tau_{n - 1} \right);\tau_{n - 1} \leq T \right\rbrack,\quad\quad H(u) = {\overline{\mathbb{E}}}^{\otimes 2}\left\lbrack G\,\mathbf{1}_{\{\Delta \leq u < \Delta + \Delta'\}} \right\rbrack,\]

where \(\Delta,\Delta'\) are the two block lengths. By
Cauchy--Schwarz and the exponential moments of block lengths (Proposition~\ref{prop:5-14}),
\[H(u) \leq \left( {\overline{\mathbb{E}}}^{\otimes 2}\left\lbrack G^{2} \right\rbrack \right)^{1/2}{\overline{\mathbb{P}}}^{\otimes 2}(\Delta + \Delta' > u)^{1/2} \leq C_{G}e^{- \theta_{1}u/4}.\]
Summing over \(n\) and splitting the renewal measure into unit
intervals,

\[\sum_{n \geq 2}^{}\mathbb{E}\left\lbrack G\left( K_{n},K_{n + 1} \right);N_{T} = n \right\rbrack \leq \int_{\lbrack 0,T\rbrack}^{}H(T - u)\, U(du) \leq c_{U}\sum_{m \geq 0}^{}C_{G}e^{- \theta_{1}(m - 1)/4} < \infty,\]

uniformly in \(T\). The two remaining terms are
\begin{align*}
\mathbb{E}\left\lbrack G\left( K_{0},K_{1} \right);N_{T} = 0 \right\rbrack &+ \mathbb{E}\left\lbrack G\left( K_{1},K_{2} \right);N_{T} = 1 \right\rbrack\\
&\leq \ \mathbb{E}\left\lbrack G\left( K_{0},K_{1} \right) \right\rbrack + \mathbb{E}\left\lbrack G\left( K_{1},K_{2} \right) \right\rbrack,
\end{align*}
which is finite by the second hypothesis in \eqref{eq:6-11a-hyp} and is
independent of \(T\). Taking \(G\left( K,K' \right) = \operatorname{diam}(K)^{q} + \operatorname{diam}\left( K' \right)^{q}\)
gives the diameter bound. Finally \(A_{T} \leq \Delta\left( K_{N_{T} + 1} \right)\)
pathwise, so taking \(G\left( K,K' \right) = \mathbf{1}_{\left\{ \Delta\left( K' \right) > s \right\}}\)
gives \(\mathbb{P}\left( A_{T} > s \right) \leq Ce^{- \theta_{1}s/4}\) uniformly
in \(T\), whence all moments. \end{proof}

\begin{corollary}[uniform \(L^q\) bound for the last incomplete cycle]\label{cor:6-11}
~ Let
\(R_{T}:= \Phi_{\psi}(T) - \Phi_{\psi}\left( \tau_{N_{T}} \right)\).
Then, pathwise,

\[\left| R_{T} \right|\  \leq \  \lVert \psi \rVert_{\infty}\left\lbrack \left( r_{1} - r_{0} \right) + 36450\, r_{0}^{- 3}\left( D_{N_{T}}^{4} + D_{N_{T} + 1}^{4} \right) \right\rbrack,\]

and consequently
\(\sup_{T \geq 0}\mathbb{E}\left\lbrack \left| R_{T} \right|^{q} \right\rbrack < \infty\)
for every \(q \geq 1\); in particular
\(\sup_{T \geq 1}\mathbb{E}\left| R_{T} \right| < \infty\).
\end{corollary}

\begin{proof}
Write \(n = N_{T}\), \(A = \mathcal{S}_{\tau_{n}}(r)\),
\(B = K_{n,T}^{(r)}\) with
\(K_{n,T} = X\left( \left\lbrack \tau_{n},T \right\rbrack \right)\), so
\(A \cup B = \mathcal{S}_{T}(r)\). By Corollary~\ref{cor:6-3},

\[\left| R_{T} - \Phi_{\psi}\left( K_{n,T} \right) \right| \leq \int_{r_{0}}^{r_{1}}\left( \beta_{1}(A \cap B) + \beta_{0}(A \cap B) \right)\left| \psi(r) \right|\, dr.\]

By the second identity of Lemma~\ref{lem:6-6},
\(A \cap B = S_{n}(r) \cap B\), and \(S_{n}(r)\), \(B\) are compact
connected sets meeting at \(X_{\tau_{n}}\); so Lemma~\ref{lem:6-7} and Theorem~\ref{thm:4-11} apply exactly as in the proof of Theorem~\ref{thm:6-8}, with
\(\operatorname{diam}\left( K_{n,T} \right) \leq D_{n + 1}\) since
\(K_{n,T} \subseteq K_{n + 1}\). Adding
\[\left| \Phi_{\psi}\left( K_{n,T} \right) \right| \leq \lVert \psi \rVert_{\infty}4050\, r_{0}^{- 3}D_{n + 1}^{4}\]
(Proposition~\ref{prop:6-5}) gives the pathwise bound; on
\(\left\{ N_{T} = 0 \right\}\) it holds with \(D_{0} = 0\) by the convention of
Section~\ref{sec:6-1}. Both diameters are
functionals of the pair \(\left( K_{N_{T}},K_{N_{T} + 1} \right)\), so
Lemma~\ref{lem:6-11a} gives uniformly bounded moments of every order. \end{proof}

\begin{theorem}[LLN at regeneration times]\label{thm:6-12}
~ Let
\(\gamma_{\psi}:= \mathbb{E}\left\lbrack Z_{3} \right\rbrack = {\overline{\mathbb{E}}}^{\otimes 2}\lbrack F\rbrack\).
Then

\[\frac{\Phi_{\psi}\left( \tau_{n} \right)}{n}\  \rightarrow \ \gamma_{\psi}\quad\quad\text{almost surely and in }L^{1}.\]
\end{theorem}

\begin{proof}
By Lemma~\ref{lem:6-9} and Proposition~\ref{prop:6-10},
\(\left( Z_{k} \right)_{k \geq 3}\) is stationary, \(1\)-dependent and
integrable. Being a two-block factor of the i.i.d.~sequence
\(\left( K_{k} \right)_{k \geq 2}\), it is ergodic; alternatively, split
into the two subsequences of odd and of even indices, each i.i.d., and
apply Kolmogorov's law of large numbers. Either way

\[\frac{1}{n}\sum_{k = 3}^{n}Z_{k}\  \rightarrow \ \mathbb{E}\left\lbrack Z_{3} \right\rbrack\quad\quad\text{a.s. and in }L^{1}.\]

Since
\(\Phi_{\psi}\left( \tau_{n} \right) = Z_{1} + Z_{2} + \sum_{k = 3}^{n}Z_{k}\)
with
\(\mathbb{E}\left| Z_{1} \right| + \mathbb{E}\left| Z_{2} \right| < \infty\)
(Propositions~\ref{prop:6-5} and~\ref{prop:6-10}), the two delayed terms contribute
\(O(1/n)\) in \(L^{1}\) and vanish almost surely after division by
\(n\). \end{proof}

\begin{lemma}[uniform \(L^{p}\) bound for the normalised functional]\label{lem:6-13}
For every \(p \geq 1\),
\[\sup_{T \geq 1}\ \mathbb{E}\left\lbrack \left| \frac{\Phi_{\psi}(T)}{T} \right|^{p} \right\rbrack\  < \ \infty .\]
In particular the family \(\left\{ \Phi_{\psi}(T)/T\ :\ T \geq 1 \right\}\) is
uniformly integrable.
\end{lemma}

\begin{proof}
Fix \(p \geq 1\) and set \(V_{k}:= \left| J_{k} \right| + I_{k} + M_{k}\) for
\(k \geq 1\), with \(M_{1}:= \sup_{0 \leq t \leq \tau_{1}}\left| \Phi_{\psi}(t) \right|\).
By Proposition~\ref{prop:6-4} applied at \(n = N_{T}\), followed by the
interpolation bound
\(\left| \Phi_{\psi}(T) - \Phi_{\psi}\left( \tau_{N_{T}} \right) \right| \leq M_{N_{T} + 1}\)
of Proposition~\ref{prop:6-10},
\begin{equation}\label{eq:6-13-block}
\left| \Phi_{\psi}(T) \right|\  \leq \ \sum_{k = 1}^{N_{T} + 1}V_{k}
\qquad \text{pathwise, for every } T \geq 0 .
\end{equation}
By Proposition~\ref{prop:6-5}, Theorem~\ref{thm:6-8} and
Proposition~\ref{prop:6-10}, \(v_{p}:= \sup_{k \geq 1}\left\| V_{k} \right\|_{p} < \infty\)
for every \(p\).

Let \(n_{T}:= \left\lceil 4T/\bar{\mu} \right\rceil + 1\), so that
Proposition~\ref{prop:5-16}(ii) gives
\[\mathbb{P}\left( N_{T} + 1 > n_{T} \right) \leq \mathbb{P}\left( N_{T} \geq n_{T} \right) \leq e^{- \lambda_{0}\bar{\mu}\, n_{T}/8} \leq C e^{- \lambda_{0}T/2}.\]
Split \eqref{eq:6-13-block} accordingly. On \(\{ N_{T} + 1 \leq n_{T}\}\),
Minkowski's inequality gives
\[\left\| \sum_{k = 1}^{N_{T} + 1}V_{k}\ ;\ N_{T} + 1 \leq n_{T} \right\|_{p}\  \leq \ \left\| \sum_{k = 1}^{n_{T}}V_{k} \right\|_{p}\  \leq \ n_{T}\, v_{p}\  \leq \ C_{p}\, (1 + T).\]
On the complementary event we use the crude pathwise majorant supplied by
Theorem~\ref{thm:4-11} applied to the whole range \(K_{T}\), namely
\(\left| \Phi_{\psi}(T) \right| \leq \lVert \psi \rVert_{\infty}4050\, r_{0}^{- 3}\Xi_{T}^{4}\)
with \(\Xi_{T}:= \operatorname{diam}\left( K_{T} \right) \leq 2\sup_{s \leq T}\lVert X_{s} \rVert\).
Since \(\mathbb{E}\left\lbrack \Xi_{T}^{\, m} \right\rbrack \leq C_{m}(1 + T)^{m}\)
for every \(m\) (Lemma~\ref{lem:4-diam}), Cauchy--Schwarz gives
\[\mathbb{E}\left\lbrack \left| \Phi_{\psi}(T) \right|^{p};N_{T} + 1 > n_{T} \right\rbrack\  \leq \ \mathbb{E}\left\lbrack \left| \Phi_{\psi}(T) \right|^{2p} \right\rbrack^{1/2}\mathbb{P}\left( N_{T} + 1 > n_{T} \right)^{1/2}\  \leq \ C_{p}(1 + T)^{4p}e^{- \lambda_{0}T/4},\]
which is bounded in \(T\). Combining the two estimates,
\(\mathbb{E}\left\lbrack \left| \Phi_{\psi}(T) \right|^{p} \right\rbrack \leq C_{p}(1 + T)^{p}\),
which is the assertion. Uniform integrability follows from boundedness in
\(L^{p}\) for any \(p > 1\).
\end{proof}

\begin{theorem}[law of large numbers for the smoothed Betti-curve functional]\label{thm:6-14}
~ Let
\(\bar{\mu} = \overline{\mathbb{E}}\left\lbrack \tau_{1} \right\rbrack \in (0,\infty)\)
be the mean block length under the conditioned law and set
\(\rho_{\psi}:= \gamma_{\psi}/\bar{\mu}\). Then

\[\frac{\Phi_{\psi}(T)}{T}\  \rightarrow \ \rho_{\psi}\quad\quad\text{almost surely and in }L^{p}\text{ for every }p \geq 1,\quad T \rightarrow \infty.\]
\end{theorem}

\begin{proof}
By Proposition~\ref{prop:5-16},
\(N_{T}/T \rightarrow 1/\bar{\mu}\) almost surely and in \(L^{q}\) for
every \(q \geq 1\). By Theorem~\ref{thm:6-12},
\(\Phi_{\psi}\left( \tau_{n} \right)/n \rightarrow \gamma_{\psi}\)
almost surely and in \(L^{1}\).

\emph{Almost sure convergence.} For
\(T \in \lbrack\tau_{N_{T}},\tau_{N_{T} + 1})\) we have
\(\left| \Phi_{\psi}(T) - \Phi_{\psi}\left( \tau_{N_{T}} \right) \right| \leq M_{N_{T} + 1}\).
By Proposition~\ref{prop:6-10}, \(\left( M_{k} \right)_{k \geq 3}\) is stationary
with \(\mathbb{E}\left\lbrack M_{3} \right\rbrack < \infty\), so
\[\sum_{n}^{}\mathbb{P}\left( M_{3} > \varepsilon n \right) \leq \varepsilon^{- 1}\mathbb{E}\left\lbrack M_{3} \right\rbrack < \infty\]
and Borel--Cantelli gives \(M_{n}/n \rightarrow 0\) a.s.; combined with
\(N_{T}/T \rightarrow 1/\bar{\mu}\) this gives
\(M_{N_{T} + 1}/T \rightarrow 0\) a.s. Hence
\[\Phi_{\psi}(T)/T = \left( \Phi_{\psi}\left( \tau_{N_{T}} \right)/N_{T} \right)\left( N_{T}/T \right) + o(1) \rightarrow \gamma_{\psi}/\bar{\mu}.\]

\(L^{p}\) convergence. By Lemma~\ref{lem:6-13} the family
\(\left\{ \Phi_{\psi}(T)/T:T \geq 1 \right\}\) is bounded in \(L^{p'}\) for every
\(p' \geq 1\), hence \(\left\{ \left| \Phi_{\psi}(T)/T \right|^{p}:T \geq 1 \right\}\)
is uniformly integrable for every \(p \geq 1\). Almost sure convergence together
with uniform integrability gives convergence in \(L^{p}\), and in particular in
\(L^{1}\). \end{proof}

\begin{remark}[on the route to \(L^{1}\)]\label{rem:6-14a}
Lemma~\ref{lem:6-13} is used only to supply uniform integrability after almost
sure convergence has been established by independent means; no independence
between \(N_{T}\) and the cycle rewards is assumed, and no limit theorem for
randomly indexed sums is invoked. None is available for free: a statement of
the form ``\(A_{n}/n \to \alpha\) in \(L^{1}\) and \(N_{T}/T \to c\) imply
\(A_{N_{T}}/T \to \alpha c\) in \(L^{1}\)'' is false without joint control of
index and sequence, since a random index may be selected precisely on the rare
events where the deviation is large. Proposition~\ref{prop:6-4} and the block
moment bounds give the required \(L^{p}\) control of \(\Phi_{\psi}(T)/T\)
directly.
\end{remark}

\subsection{Intensity measures}\label{sec:6-6}

\begin{corollary}[Betti-curve intensity measure]\label{cor:6-15}
~ There exists a
finite positive measure \(\lambda\) on
\(\left\lbrack r_{0},r_{1} \right\rbrack\) such that for every bounded
continuous \(\psi\) supported in
\(\left\lbrack r_{0},r_{1} \right\rbrack\),

\[\frac{1}{T}\int_{r_{0}}^{r_{1}}\beta_{1}\left( \mathcal{S}_{T}(r) \right)\psi(r)\, dr\  \rightarrow \ \int_{r_{0}}^{r_{1}}\psi(r)\,\lambda(dr)\quad\quad\text{almost surely and in }L^{1}.\]
\end{corollary}

\begin{proof}
By Theorem~\ref{thm:6-14} the map \(\psi \mapsto \rho_{\psi}\) is
well defined on
\(C_{c}\left( \left\lbrack r_{0},r_{1} \right\rbrack \right)\); it is
linear, because \(\psi \mapsto \Phi_{\psi}(T)\) is linear and the limit
is unique, positive, because \(\Phi_{\psi}(T) \geq 0\) for
\(\psi \geq 0\), and bounded, since Theorem~\ref{thm:6-8} and Proposition~\ref{prop:6-5} give
the two-sided estimate
\[\left| \rho_{\psi} \right| \leq \lVert \psi \rVert_{\infty}\, 4050\, r_{0}^{- 3}{\bar{\mu}}^{- 1}\,\overline{\mathbb{E}}\left\lbrack 9D_{2}^{4} + 8D_{3}^{4} \right\rbrack,\]
the absolute value being required because \(\psi\) may change sign.
The Riesz representation theorem supplies a finite positive Radon
measure \(\lambda\) on \(\left\lbrack r_{0},r_{1} \right\rbrack\) with
\(\rho_{\psi} = \int\psi\, d\lambda\). \end{proof}

Finally, the same limit can be stated in diagram language.

\begin{corollary}[diagram-measure formulation on the test class \(\varphi_\psi\)]\label{cor:6-16}
~ For every bounded continuous \(\psi\)
supported in \(\left\lbrack r_{0},r_{1} \right\rbrack\),

\[\frac{1}{T}\int_{\Delta}^{}\varphi_{\psi}(b,d)\,\mu_{T}^{(1)}(db\, dd)\  \rightarrow \ \rho_{\psi}\quad\quad\text{almost surely and in }L^{1}.\]
\end{corollary}

\begin{proof}
This is Theorem~\ref{thm:6-14} rewritten through the Betti-curve
representation \eqref{eq:2-betarep}, whose hypotheses are met because
\(\int_{r_{0}}^{r_{1}}\beta_{1}(T,r)dr < \infty\) almost surely by
Theorem~\ref{thm:4-11}. \end{proof}

\begin{remark}[the test class is not measure-determining]\label{rem:6-17}
~ The
class
\[\mathcal{T} = \{\varphi_{\psi}:\varphi_{\psi}(b,d) = \int_{b}^{d}\psi(r)dr,\ \psi \in C_{c}\left( r_{0},r_{1} \right)\}\]
does not determine measures on the compact birth--death rectangle
\(\Delta\left( r_{0},r_{1} \right) = \{(b,d):r_{0} \leq b < d \leq r_{1}\}\).
Indeed, pick \(r_{0} < b_{1} < b_{2} < d_{1} < d_{2} < r_{1}\) and set
\(\mu_{1} = \delta_{\left( b_{1},d_{1} \right)} + \delta_{\left( b_{2},d_{2} \right)}\),
\(\mu_{2} = \delta_{\left( b_{1},d_{2} \right)} + \delta_{\left( b_{2},d_{1} \right)}\).
Then \(\mu_{1} \neq \mu_{2}\), but
\[\mathbf{1}_{\lbrack b_{1},d_{1})} + \mathbf{1}_{\lbrack b_{2},d_{2})} = \mathbf{1}_{\lbrack b_{1},d_{2})} + \mathbf{1}_{\lbrack b_{2},d_{1})}\]
pointwise, so
\(\int\varphi_{\psi}\, d\mu_{1} = \int\varphi_{\psi}\, d\mu_{2}\) for
every admissible \(\psi\). Equivalently, the observables
\(\varphi_{\psi}\) see only the alive-count profile
\(r \mapsto \mu\{(b,d):b \leq r < d\}\), the Betti-curve transform of
the diagram measure. Corollary~\ref{cor:6-16} is therefore a statement about
the Betti-curve transform, at the level of smoothed alive-count observables,
and not a statement about the diagram measure itself.
\end{remark}

\section{Gaussian fluctuations}\label{sec:7}

This section identifies the fluctuations around the growth rate
\(\rho_{\psi}\) of Section~\ref{sec:6}. The regeneration structure reduces the
problem to a sum of centred cycle rewards which is stationary and
\(1\)-dependent, and the moment estimates already established
(Theorem~\ref{thm:4-11}, Corollary~\ref{cor:5-15},
Proposition~\ref{prop:6-10}, Corollary~\ref{cor:6-11}) supply far more
integrability than a central limit theorem needs. What remains is assembly: a
central limit theorem for \(1\)-dependent sequences, a renewal time change,
and the negligibility of the final partial cycle.

\subsection{Centred cycle rewards}\label{sec:7-1}

Recall from Section~\ref{sec:6} the regeneration-time increments
\(Z_{1} = \Phi_{\psi}\left( \tau_{1} \right)\),
\(Z_{k} = \Phi_{\psi}\left( \tau_{k} \right) - \Phi_{\psi}\left( \tau_{k - 1} \right)\)
for \(k \geq 2\), and write

\[\eta_{k}:= \tau_{k} - \tau_{k - 1}\quad\quad(k \geq 1)\]
for the cycle lengths. By Corollary~\ref{cor:5-8} the blocks
\(K_{k}\) are independent, and identically distributed for \(k \geq 2\)
with the conditioned law \(\overline{\mathbb{P}}\), the first block
being delayed; consequently \(\left( \eta_{k} \right)_{k \geq 2}\) is
i.i.d.~with mean
\(\bar{\mu} = \overline{\mathbb{E}}\left\lbrack \tau_{1} \right\rbrack\)
and exponential moments (Proposition~\ref{prop:5-14}), and
\(\left( Z_{k} \right)_{k \geq 3}\) is stationary and \(1\)-dependent
(Lemma~\ref{lem:6-9}) with moments of every order
(Proposition~\ref{prop:6-10}). Recall also
\(\rho_{\psi} = \gamma_{\psi}/\bar{\mu}\) with
\(\gamma_{\psi} = \mathbb{E}\left\lbrack Z_{3} \right\rbrack\)
Theorems~\ref{thm:6-12} and~\ref{thm:6-14}.

Define the centred cycle rewards

\[\widetilde{Z}_{k}:= Z_{k} - \rho_{\psi}\,\eta_{k},\quad\quad k \geq 3.\]

\begin{proposition}[centred, stationary, \(1\)-dependent, \(L^q\) for all \(q\)]\label{prop:7-1}
~ The sequence
\(\left( \widetilde{Z}_{k} \right)_{k \geq 3}\) is centred, stationary and
\(1\)-dependent, and
\(\mathbb{E}\left\lbrack \left| \widetilde{Z}_{k} \right|^{q} \right\rbrack < \infty\)
for every \(q \geq 1\), with a bound uniform in \(k\).
\end{proposition}

\begin{proof}
\(Z_{k}\) is a measurable functional of the pair
\(\left( K_{k - 1},K_{k} \right)\) (Lemma~\ref{lem:6-9}), and
\(\eta_{k}\) of the single block \(K_{k}\); since the blocks
\(\left( K_{k} \right)_{k \geq 2}\) are i.i.d.\ (Corollary~\ref{cor:5-8}), the
pair sequence is stationary and any two-block factor is
stationary and \(1\)-dependent. Centring is the definition of
\(\rho_{\psi}\): for \(k \geq 3\),
\(\mathbb{E}\left\lbrack Z_{k} \right\rbrack = \gamma_{\psi}\) and
\(\mathbb{E}\left\lbrack \eta_{k} \right\rbrack = \bar{\mu}\), so
\(\mathbb{E}\left\lbrack \widetilde{Z}_{k} \right\rbrack = \gamma_{\psi} - \rho_{\psi}\bar{\mu} = 0\).
The moment bound combines Proposition~\ref{prop:6-10} with the exponential moments
of the cycle lengths. \end{proof}

\subsection{The central limit theorem at regeneration times}\label{sec:7-2}

Put \(\Sigma_{n}:= \sum_{k = 3}^{n}\widetilde{Z}_{k}\) for \(n \geq 3\), and
\(\Sigma_{n}:= 0\) for \(n < 3\).

\begin{theorem}[fluctuations along the regeneration times]\label{thm:7-2}
~ As
\(n \rightarrow \infty\),

\[\frac{\Sigma_{n}}{\sqrt{n}}\  \rightarrow \ \mathcal{N}\left( 0,\sigma_{cyc}^{2} \right),\quad\quad\sigma_{cyc}^{2}:= \operatorname{Var}\left( \widetilde{Z}_{3} \right) + 2\operatorname{Cov}\left( \widetilde{Z}_{3},\widetilde{Z}_{4} \right),\]

and the same limit holds for
\(\left( \Phi_{\psi}\left( \tau_{n} \right) - \rho_{\psi}\tau_{n} \right)/\sqrt{n}\).
\end{theorem}

\begin{proof}
Proposition~\ref{prop:7-1} gives exactly the hypotheses of the
central limit theorem of Hoeffding and Robbins \cite[Theorem 1]{HR48} for
stationary \(m\)-dependent square-integrable sequences, applied with
\(m = 1\); the
limiting variance of such a sequence is
\(\operatorname{Var}\left( \widetilde{Z}_{3} \right) + 2\operatorname{Cov}\left( \widetilde{Z}_{3},\widetilde{Z}_{4} \right)\), the
higher-lag covariances vanishing by \(1\)-dependence. For the second
statement, telescoping gives

\[\Phi_{\psi}\left( \tau_{n} \right) - \rho_{\psi}\tau_{n} = \Sigma_{n} + \left( Z_{1} + Z_{2} - \rho_{\psi}\left( \eta_{1} + \eta_{2} \right) \right),\]

and the bracket is a fixed integrable random variable --- it involves
only the delayed block and its successor --- hence \(O_{\mathbb{P}}(1)\)
and negligible after division by \(\sqrt{n}\). Slutsky's lemma completes
the proof. \end{proof}

\begin{proposition}[functional form]\label{prop:7-3}
~ Let
\(W_{n}(s):= n^{- 1/2}\Sigma_{\lfloor ns\rfloor}\), \(s \geq 0\). Then
\(W_{n} \Rightarrow \sigma_{cyc}W\) in \(D\lbrack 0,\infty)\) with the
Skorokhod topology, where \(W\) is standard Brownian motion.
\end{proposition}

\begin{proof}
This is the functional central limit theorem for stationary
\(m\)-dependent square-integrable sequences, \(m = 1\); see
\cite[Theorem 20.1 and Section 19]{Bil99}, whose hypotheses are supplied by
Proposition~\ref{prop:7-1}. Finite-dimensional convergence follows from
Theorem~\ref{thm:7-2} applied to increments together with the Cram\'er--Wold
device, and tightness follows from the moment criterion
\cite[Theorem 13.5]{Bil99}, whose fourth-moment hypothesis is available here
because \(\left( \widetilde{Z}_{k} \right)_{k \geq 3}\) has moments of every order by
Proposition~\ref{prop:7-1}. \end{proof}

\subsection{Renewal time change}\label{sec:7-3}

Let \(N_{T} = \max\{ n \geq 0:\tau_{n} \leq T\}\),
\(A_{T} = T - \tau_{N_{T}}\) the age, and
\(R_{T} = \Phi_{\psi}(T) - \Phi_{\psi}\left( \tau_{N_{T}} \right)\) the
contribution of the last incomplete cycle, as in Section~\ref{sec:6}.

\begin{lemma}[the incomplete cycle is negligible at scale \(\sqrt{T}\)]\label{lem:7-4}
~
\(\left( R_{T} - \rho_{\psi}A_{T} \right)/\sqrt{T} \rightarrow 0\) in
\(L^{1}\), hence in probability.
\end{lemma}

\begin{proof}
By Corollary~\ref{cor:6-11},
\(\sup_{T}\mathbb{E}\left| R_{T} \right| < \infty\), and by Lemma~\ref{lem:6-11a},
\(\sup_{T}\mathbb{E}\left\lbrack A_{T} \right\rbrack < \infty\). Hence
\(\mathbb{E}\left| R_{T} - \rho_{\psi}A_{T} \right|/\sqrt{T} \leq C/\sqrt{T} \rightarrow 0\).
\end{proof}

\begin{proposition}[fluctuations at the last completed cycle]\label{prop:7-5}
~ As
\(T \rightarrow \infty\),

\[\frac{\Sigma_{N_{T}}}{\sqrt{T}}\  \rightarrow \ \mathcal{N}\left( 0,\ \frac{\sigma_{cyc}^{2}}{\bar{\mu}} \right).\]
\end{proposition}

\begin{proof}
We spell out the pair of processes to which \cite[Theorem 14.4]{Bil99} is
applied. For \(m \in \mathbb{N}\) let
\[W^{(m)}(s):= \frac{\Sigma_{\lfloor ms \rfloor}}{\sqrt{m}},\qquad s \geq 0,\]
be the rescaled partial-sum process of the centred cycle rewards, an element of
\(D\lbrack 0,\infty)\); by Proposition~\ref{prop:7-3},
\(W^{(m)} \Rightarrow \sigma_{cyc}B\) in the Skorokhod topology, where \(B\) is
a standard Brownian motion. Let
\[\varphi_{m}(s):= \frac{N_{ms\bar{\mu}}}{m},\qquad s \geq 0,\]
be the rescaled counting process, a nondecreasing element of
\(D\lbrack 0,\infty)\). By Proposition~\ref{prop:5-16}(i),
\(\varphi_{m}(s) \rightarrow s\) almost surely for each fixed \(s\). This
upgrades to local uniform convergence by the elementary principle (not Dini's
theorem, whose monotonicity hypothesis is in the index and is unavailable
here) that if nondecreasing functions \(f_{m}\) converge pointwise on
\(\lbrack 0,S\rbrack\) to a continuous nondecreasing \(f\), the convergence is
uniform: given \(\varepsilon > 0\), choose a finite grid
\(0 = s_{0} < \cdots < s_{p} = S\) with
\(f\left( s_{i} \right) - f\left( s_{i - 1} \right) < \varepsilon\) for all
\(i\), possible by uniform continuity; for \(s \in \lbrack s_{i - 1},s_{i}\rbrack\),
monotonicity gives
\(f_{m}\left( s_{i - 1} \right) - f\left( s_{i} \right) \leq f_{m}(s) - f(s) \leq f_{m}\left( s_{i} \right) - f\left( s_{i - 1} \right)\),
and both bounds are within \(2\varepsilon\) of zero for \(m\) large, uniformly
in \(i\). Applying this on each compact interval gives
\(\varphi_{m} \rightarrow \operatorname{id}\) almost surely, hence in
probability, uniformly on compacts.

The limiting time change \(\operatorname{id}\) is deterministic and the
limit process \(\sigma_{cyc}B\) has continuous paths. Hence the pair
\(\left( W^{(m)},\varphi_{m} \right)\) converges jointly, marginal convergence
sufficing precisely because one coordinate has a constant limit, and
\cite[Theorem 14.4]{Bil99} applies to give
\begin{equation}\label{eq:7-tc}
W^{(m)} \circ \varphi_{m}\  \Rightarrow \ \sigma_{cyc}B \circ \operatorname{id} = \sigma_{cyc}B
\end{equation}
in \(D\lbrack 0,\infty)\).

It remains to pass from the index \(m\) to physical time \(T\), and here one
must be slightly careful: \(W^{(m)} \circ \varphi_{m}(s) = \Sigma_{N_{m s \bar{\mu}}}/\sqrt{m}\),
so evaluating at \(s = 1\) produces \(N_{m\bar{\mu}}\), and \(m\bar{\mu}\) is
not equal to \(T\) for any integer \(m\) in general. Instead put
\[m_{T}:= \left\lfloor T/\bar{\mu} \right\rfloor,\qquad s_{T}:= \frac{T}{m_{T}\bar{\mu}}\  \longrightarrow \ 1\quad (T \rightarrow \infty),\]
so that \(m_{T}s_{T}\bar{\mu} = T\) exactly and therefore
\(W^{\left( m_{T} \right)} \circ \varphi_{m_{T}}\left( s_{T} \right) = \Sigma_{N_{T}}/\sqrt{m_{T}}\).
Since the limit process in \eqref{eq:7-tc} has continuous paths and
\(s_{T} \rightarrow 1\), the evaluation map \((x,s) \mapsto x(s)\) is
continuous at every point of \(D\lbrack 0,\infty) \times \lbrack 0,\infty)\)
whose first coordinate is continuous at the second, so the continuous mapping
theorem gives
\(\Sigma_{N_{T}}/\sqrt{m_{T}} \Rightarrow \mathcal{N}\left( 0,\sigma_{cyc}^{2} \right)\).
Finally \(m_{T}/(T/\bar{\mu}) \rightarrow 1\), so
\(\sqrt{m_{T}}/\sqrt{T/\bar{\mu}} \rightarrow 1\) and Slutsky's lemma converts
this into
\(\Sigma_{N_{T}}/\sqrt{T/\bar{\mu}} \Rightarrow \mathcal{N}\left( 0,\sigma_{cyc}^{2} \right)\),
which is the assertion after rescaling.

No independence between \(N_{T}\) and the reward process is used, and none
holds; the theorem requires only that the time change have a deterministic
limit, which Proposition~\ref{prop:5-16}(i) supplies. \end{proof}

\subsection{The central limit theorem at deterministic time}\label{sec:7-4}

\begin{theorem}[Gaussian fluctuations of the smoothed persistence functional]\label{thm:7-6}
~ Let \(\psi\) be bounded and Borel with support in
\(\left\lbrack r_{0},r_{1} \right\rbrack \subset (0,\infty)\). Then

\[\frac{\Phi_{\psi}(T) - \rho_{\psi}T}{\sqrt{T}}\  \rightarrow_{T \rightarrow \infty}\ \mathcal{N}\left( 0,\sigma_{\psi}^{2} \right),\quad\quad\sigma_{\psi}^{2} = \frac{\operatorname{Var}\left( \widetilde{Z}_{3} \right) + 2\operatorname{Cov}\left( \widetilde{Z}_{3},\widetilde{Z}_{4} \right)}{\bar{\mu}}.\]
\end{theorem}

\begin{proof}
Decompose

\[\frac{\Phi_{\psi}(T) - \rho_{\psi}T}{\sqrt{T}} = \frac{\Phi_{\psi}\left( \tau_{N_{T}} \right) - \rho_{\psi}\tau_{N_{T}}}{\sqrt{T}} + \frac{R_{T} - \rho_{\psi}A_{T}}{\sqrt{T}}.\]

The second term vanishes in probability by Lemma~\ref{lem:7-4}. For the first,
telescoping as in Theorem~\ref{thm:7-2} replaces
\(\Phi_{\psi}\left( \tau_{N_{T}} \right) - \rho_{\psi}\tau_{N_{T}}\) by
\(\Sigma_{N_{T}}\) up to the fixed integrable variable
\(Z_{1} + Z_{2} - \rho_{\psi}\left( \eta_{1} + \eta_{2} \right)\), which
is \(O_{\mathbb{P}}(1)\) and disappears at scale \(\sqrt{T}\).
Proposition~\ref{prop:7-5} and Slutsky's lemma give the conclusion. \end{proof}

\begin{remark}[the variance as a Green--Kubo formula]\label{rem:7-7}
~ The
numerator of \(\sigma_{\psi}^{2}\) is the asymptotic variance per
regeneration cycle of the centred reward sequence and the denominator
converts cycle index into physical time. Only the lag-\(0\) and
lag-\(1\) covariances appear, because the \(1\)-dependence produced by
the interface structure of Section~\ref{sec:6} is exact: an increment \(Z_{k}\) is
a functional of two consecutive blocks and of nothing else (Lemma~\ref{lem:6-9}),
so \(\operatorname{Cov}\left( \widetilde{Z}_{3},\widetilde{Z}_{3 + j} \right) = 0\) for \(j \geq 2\). The
delayed first block contributes nothing to \(\sigma_{\psi}^{2}\); it is
absorbed in the \(O_{\mathbb{P}}(1)\) term.

The variance is finite for every bounded Borel \(\psi\) of arbitrary sign: the
second moments it involves are those of \(\widetilde{Z}_{3}\) and \(\widetilde{Z}_{4}\), and these are
supplied by Proposition~\ref{prop:7-1}, whose moment bound rests in turn on the
interface majorant \(I_{|\psi|}\) of Corollary~\ref{cor:6-3} --- defined with
\(|\psi|\) precisely so that it dominates the non-additivity of \(\Phi_{\psi}\)
without a sign restriction --- together with Proposition~\ref{prop:6-10}.
\end{remark}

\subsection{Finite-dimensional Gaussian fluctuations}\label{sec:7-5}

\begin{theorem}[joint fluctuations]\label{thm:7-8}
~ Let
\(\psi_{1},\ldots,\psi_{m}\) be bounded Borel functions supported in a
common compact window
\(\left\lbrack r_{0},r_{1} \right\rbrack \subset (0,\infty)\). Then, as
\(T \rightarrow \infty\),

\[\left( \frac{\Phi_{\psi_{1}}(T) - \rho_{\psi_{1}}T}{\sqrt{T}},\ \ldots,\ \frac{\Phi_{\psi_{m}}(T) - \rho_{\psi_{m}}T}{\sqrt{T}} \right)\  \rightarrow \ \mathcal{N}(0,\Sigma),\]

where, writing \(\widetilde{Z}_{k}^{i}\) for the centred cycle rewards attached to
\(\psi_{i}\),

\[\Sigma_{ij} = \frac{\operatorname{Cov}\left( \widetilde{Z}_{3}^{i},\widetilde{Z}_{3}^{j} \right) + \operatorname{Cov}\left( \widetilde{Z}_{3}^{i},\widetilde{Z}_{4}^{j} \right) + \operatorname{Cov}\left( \widetilde{Z}_{4}^{i},\widetilde{Z}_{3}^{j} \right)}{\bar{\mu}}.\]
\end{theorem}

\begin{proof}
The map \(\psi \mapsto \Phi_{\psi}(T)\) is linear, and so
are \(\psi \mapsto \rho_{\psi}\) and \(\psi \mapsto \widetilde{Z}_{k}\). Hence for
\(a \in \mathbb{R}^{m}\) the linear combination
\(\sum_{i}^{}a_{i}\left( \Phi_{\psi_{i}}(T) - \rho_{\psi_{i}}T \right)\)
is \(\Phi_{\psi_{a}}(T) - \rho_{\psi_{a}}T\) with
\(\psi_{a} = \sum_{i}^{}a_{i}\psi_{i}\), again bounded Borel with
support in \(\left\lbrack r_{0},r_{1} \right\rbrack\). Theorem~\ref{thm:7-6}
applies to \(\psi_{a}\) and gives a centred Gaussian limit with variance
\(\sigma_{\psi_{a}}^{2} = a^{\top}\Sigma a\), the displayed form of
\(\Sigma\) following by bilinearity of the covariance. The Cram\'er--Wold
device concludes. \end{proof}

\begin{remark}[what is not claimed]\label{rem:7-9}
~ We do not address strict
positivity of \(\sigma_{\psi}^{2}\). The fluctuation theorem is complete
without it --- the variance is identified through its Green--Kubo
expression --- but degeneracy is not excluded by the present argument,
and \(\sigma_{\psi}^{2} = 0\) would force the centred rewards to cancel
at lags \(0\) and \(1\) in the exact combination above. Nor do we obtain
a fluctuation theorem at a fixed radius: the argument of Section~\ref{sec:8}
supplies first moments of \(\beta_{1}(T,r)\) at almost every radius,
whereas the present one requires second moments of the cycle rewards,
and the integrated bound of Theorem~\ref{thm:4-11} does not survive squaring. A
fixed-radius central limit theorem therefore remains out of reach by
this route, at every radius, and not only at the exceptional ones.
\end{remark}

\section{Fixed-radius asymptotics}\label{sec:8}

Section~\ref{sec:6} produced a limit for the smoothed functional
\(\Phi_{\psi}(T)\) and a finite positive Betti-curve intensity measure
\(\lambda\) on the radius axis. This section asks what survives when the
smoothing is removed: does the hole count
\(\beta_{1}(T,r):= \beta_{1}\left( \mathcal{S}_{T}(r) \right)\) at a single
radius satisfy a law of large numbers, and is \(\lambda\) absolutely continuous
with that limit as its density? We prove an unconditional answer at the
regeneration times, for almost every radius, and identify \(\lambda\)
(Theorem~\ref{thm:8-3}). Passing to deterministic times is a different matter
at a single radius, and we present it conditionally, on an explicit
block-oscillation hypothesis (Proposition~\ref{prop:8-4}) which
Remark~\ref{rem:8-5} shows is not supplied by the deterministic theory of
Section~\ref{sec:4} and must not be presumed. Section~\ref{sec:8-2} records
what would be needed to remove the exceptional null set of radii. Throughout,
\(0 < r_{0} \leq r \leq r_{1} < \infty\) are fixed.

\subsection{The law of large numbers at almost every radius}\label{sec:8-1}

The first statement costs nothing beyond Section~\ref{sec:6} and separability,
and it is the strongest simultaneous statement available: almost surely,
the whole Betti curve, rescaled, converges as a measure.

\begin{theorem}[convergence of the Betti-curve measures]\label{thm:8-1}
Let \(\nu_{T}(dr):= T^{- 1}\beta_{1}(T,r)\, dr\) on
\(\lbrack r_{0},r_{1}\rbrack\). Then, almost surely, \(\nu_{T}\) converges
weakly to \(\lambda\): for every \(\psi \in C\lbrack r_{0},r_{1}\rbrack\),
simultaneously,
\[\frac{1}{T}\int_{r_{0}}^{r_{1}}\beta_{1}(T,r)\,\psi(r)\, dr\  \rightarrow \ \int_{r_{0}}^{r_{1}}\psi\, d\lambda .\]
\end{theorem}

\begin{proof}
Let \(\mathcal{D}\) be a countable subset of
\(C\left\lbrack r_{0},r_{1} \right\rbrack\), dense in the uniform norm and
containing the constant function \(1\); such a set exists because
\(C\left\lbrack r_{0},r_{1} \right\rbrack\) is separable. By
Theorem~\ref{thm:6-14} and Corollary~\ref{cor:6-15}, for each
\(\psi \in \mathcal{D}\) we have
\(T^{- 1}\Phi_{\psi}(T) \rightarrow \int\psi\, d\lambda\) almost surely; let
\(\Omega_{0}\) be the intersection of these countably many full-measure events.

On \(\Omega_{0}\) the total masses are bounded: taking \(\psi \equiv 1\),
admissible since Theorem~\ref{thm:6-14} allows bounded Borel weights on the
closed interval,
\[\nu_{T}\left( \left\lbrack r_{0},r_{1} \right\rbrack \right) = \frac{1}{T}\int_{r_{0}}^{r_{1}}\beta_{1}(T,r)\, dr\  \rightarrow \ \lambda\left( \left\lbrack r_{0},r_{1} \right\rbrack \right) < \infty .\]
This bounds the masses only for large \(T\); for the initial range, fixing \(T_{0}\) beyond which the displayed limit is within
\(1\) of its value, Theorem~\ref{thm:4-11} gives, for \(1 \leq T \leq T_{0}\),
\[\nu_{T}\left( \left\lbrack r_{0},r_{1} \right\rbrack \right)\  \leq \ \int_{r_{0}}^{r_{1}}\beta_{1}(T,r)\, dr\  \leq \ 4050\, r_{0}^{- 3}\operatorname{diam}\left( K_{T_{0}} \right)^{4},\]
using \(T \geq 1\) and \(K_{T} \subseteq K_{T_{0}}\); the right-hand side is
almost surely finite and does not depend on \(T\). Hence
\(\sup_{T \geq 1}\nu_{T}\left( \left\lbrack r_{0},r_{1} \right\rbrack \right) < \infty\)
on \(\Omega_{0}\). Given an arbitrary
\(\psi \in C\left\lbrack r_{0},r_{1} \right\rbrack\) and \(\varepsilon > 0\),
choose \(\psi_{m} \in \mathcal{D}\) with
\(\lVert \psi - \psi_{m} \rVert_{\infty} < \varepsilon\); then
\[\left| \int\psi\, d\nu_{T} - \int\psi_{m}\, d\nu_{T} \right| \leq \varepsilon\,\nu_{T}\left( \left\lbrack r_{0},r_{1} \right\rbrack \right)\]
and likewise for \(\lambda\), so on \(\Omega_{0}\)
\(\limsup_{T}\left| \int\psi\, d\nu_{T} - \int\psi\, d\lambda \right| \leq C\varepsilon\).
Letting \(\varepsilon \downarrow 0\) gives the claim.
\end{proof}

We now pass to a single radius. Two points of definition must be settled
first, and both are genuine: at a radius where \(\beta_{1}\) is infinite a
difference of hole counts is not an extended real number, and the increments
are signed, since a new path segment may fill holes as well as create them.
For \(k \geq 2\) and \(r \in \left\lbrack r_{0},r_{1} \right\rbrack\) put
\begin{equation}\label{eq:8-Zr}
Z_{k}(r):= \beta_{1}\left( S_{k - 1}(r) \cup S_{k}(r) \right) - \beta_{1}\left( S_{k - 1}(r) \right)
\end{equation}
whenever both terms are finite, and \(Z_{k}(r):= 0\) otherwise --- a
convention that makes \((\omega,r) \mapsto Z_{k}(r)\) a globally defined
product-measurable map, to which Fubini's theorem is applied in part (iii) of
Theorem~\ref{thm:8-3}; by Lemma~\ref{lem:8-2}(i) the set on which the
convention is invoked is null for the product measure, so nothing depends on
it. Put also
\[W_{k}(r):= \beta_{1}\left( S_{k}(r) \right) + \beta_{1}\left( S_{k - 1}(r) \cap S_{k}(r) \right) + \beta_{0}\left( S_{k - 1}(r) \cap S_{k}(r) \right) \in \lbrack 0,\infty\rbrack,\]
which is always defined. The sausages and unions of sausages in \eqref{eq:8-Zr}
and in \(W_{k}(r)\) are \(r\)-offsets of fixed compact sets, and the
intersection terms are controlled through Lemma~\ref{lem:6-7} by such offsets;
no supremum over intermediate times appears, which is what makes the next
lemma available. All of these are measurable, jointly in \((\omega,r)\), by
Lemma~\ref{lem:4-meas}.

\begin{lemma}[a full-measure set of radii]\label{lem:8-2}
There is a Borel set \(N \subset \left\lbrack r_{0},r_{1} \right\rbrack\) with
\(\left| N \right| = 0\) such that for every \(r \notin N\):

(i) almost surely \(\beta_{1}\left( \mathcal{U}_{k}(r) \right) < \infty\)
and \(\beta_{1}\left( S_{k}(r) \right) < \infty\) for every \(k \geq 1\); in
particular \(Z_{k}(r)\) is almost surely a well-defined real number for every
\(k \geq 2\), and \(\left| Z_{k}(r) \right| \leq W_{k}(r) < \infty\);

(ii) almost surely, for every \(n \geq 1\),
\[\beta_{1}\left( \tau_{n},r \right)\  = \ \beta_{1}\left( \mathcal{U}_{1}(r) \right) + \sum_{k = 2}^{n}Z_{k}(r);\]

(iii) \(\overline{\mathbb{E}}\left\lbrack \left| Z_{3}(r) \right| \right\rbrack < \infty\)
and
\[\mathbb{E}\left\lbrack \beta_{1}\left( \mathcal{U}_{1}(r) \right) + \left| Z_{2}(r) \right| \right\rbrack < \infty.\]
\end{lemma}

\begin{proof}
Consider the countable family of random compact connected sets
\[\mathcal{C}:= \left\{ \mathcal{U}_{k} \right\}_{k \geq 1}\ \cup\ \left\{ K_{k} \right\}_{k \geq 1}\ \cup\ \left\{ K_{k - 1} \cup K_{k} \right\}_{k \geq 2},\]
which contains every trace whose offset appears in \eqref{eq:8-Zr}, in
\(W_{k}(r)\), or in the Euler identity of Lemma~\ref{lem:6-9}. Each member is
compact, connected and of finite diameter, so Theorem~\ref{thm:4-11} gives,
pathwise,
\[\int_{r_{0}}^{r_{1}}\beta_{1}\left( C^{(r)} \right)dr < \infty\qquad\text{for every }C \in \mathcal{C}.\] Hence for every \(\omega\) there is a
Lebesgue-null set of radii off which all these Betti numbers are finite. Let
\[E:= \left\{ (\omega,r):\ \beta_{1}\left( C^{(r)} \right) = \infty\ \text{for some }C \in \mathcal{C} \right\};\]
\(E\) is product-measurable by Lemma~\ref{lem:4-meas} and each
\(\omega\)-section is Lebesgue-null, so by Fubini's theorem almost every
\(r\)-section is \(\mathbb{P}\)-null; discarding the exceptional radii gives
the finiteness of the offset Betti numbers in (i), and Lemma~\ref{lem:6-7}
then gives the finiteness of the two intersection Betti numbers in
\(W_{k}(r)\), so that \(W_{k}(r) < \infty\). Assertion (ii) is the telescoping
of the identity of Lemma~\ref{lem:6-9}, valid pointwise at every radius at
which the finitely many Betti numbers involved are finite, hence almost surely
for \(r \notin N\).

For (iii), note first that for \(r \notin N\) the finiteness hypothesis
\eqref{eq:6-2-fin} holds almost surely for the pair
\(S_{k - 1}(r),S_{k}(r)\), by part (i); Lemma~\ref{lem:6-2} therefore applies
and gives \(\left| Z_{k}(r) \right| \leq W_{k}(r)\). Lemma~\ref{lem:6-7}
together with Theorem~\ref{thm:4-11} gives, pathwise,
\begin{equation}\label{eq:8-integrated}
\int_{r_{0}}^{r_{1}}W_{k}(r)\, dr\  \leq \ 4050\, r_{0}^{- 3}\left( 9D_{k - 1}^{4} + 8D_{k}^{4} \right) + \left( r_{1} - r_{0} \right),
\end{equation}
whose expectation is finite by Corollary~\ref{cor:5-15}. Tonelli's theorem
applied to \eqref{eq:8-integrated} for \(k = 2\) and \(k = 3\) gives two further
Lebesgue-null sets of radii off which
\(\overline{\mathbb{E}}\left\lbrack W_{3}(r) \right\rbrack < \infty\) and
\(\mathbb{E}\left\lbrack W_{2}(r) \right\rbrack < \infty\); adjoin them to
\(N\), and adjoin also the null set on which
\(\mathbb{E}\left\lbrack \beta_{1}\left( \mathcal{U}_{1}(r) \right) \right\rbrack = \infty\),
which is null by the same argument applied to
\(\int\beta_{1}\left( \mathcal{U}_{1}(r) \right)dr \leq 4050\, r_{0}^{- 3}D_{1}^{4}\).

Stationarity does the rest: for every \(k \geq 3\) the law of
\(\left( K_{k - 1},K_{k} \right)\) under \(\overline{\mathbb{P}}\) is that of
\(\left( K_{2},K_{3} \right)\) by Corollary~\ref{cor:5-8}, so
\[\overline{\mathbb{E}}\left\lbrack \left| Z_{k}(r) \right| \right\rbrack = \overline{\mathbb{E}}\left\lbrack \left| Z_{3}(r) \right| \right\rbrack.\]
No supremum over \(k\) is extracted from \eqref{eq:8-integrated}, and none
could be: a bound on \(\sup_{k}\int F_{k}(r)dr\) does not control
\(\sup_{k}F_{k}(r)\) at a prescribed \(r\), since the mass may move with \(k\).
\end{proof}

\begin{theorem}[fixed-radius law of large numbers at the regeneration times]\label{thm:8-3}
Let \(N\) be as in Lemma~\ref{lem:8-2} and set
\(\gamma(r):= \overline{\mathbb{E}}\left\lbrack Z_{3}(r) \right\rbrack\) and
\(\rho(r):= \gamma(r)/\bar{\mu}\) for \(r \notin N\). Then:

(i) for every \(r \notin N\) there is \(c(r) < \infty\) with
\(\mathbb{E}\left\lbrack \beta_{1}\left( \tau_{n},r \right) \right\rbrack \leq c(r)\, n\)
for every \(n \geq 1\);

(ii) for every \(r \notin N\),
\[\frac{\beta_{1}\left( \tau_{n},r \right)}{n}\  \longrightarrow \ \gamma(r)\quad\quad\text{almost surely and in }L^{1};\]

(iii) \(\gamma\) is Borel and integrable on
\(\left\lbrack r_{0},r_{1} \right\rbrack\), and
\(\lambda(dr) = \rho(r)\, dr\) as measures on
\(\left\lbrack r_{0},r_{1} \right\rbrack\); in particular \(\lambda\) has no
atom at \(r_{0}\) or at \(r_{1}\).
\end{theorem}

\begin{proof}
Fix \(r \notin N\) and abbreviate \(Z_{k} = Z_{k}(r)\).

\emph{(i).} By Lemma~\ref{lem:8-2}(ii) and the triangle inequality,
\[\mathbb{E}\left\lbrack \beta_{1}\left( \tau_{n},r \right) \right\rbrack \leq \mathbb{E}\left\lbrack \beta_{1}\left( \mathcal{U}_{1}(r) \right) + \left| Z_{2} \right| \right\rbrack + (n - 2)^{+}\,\overline{\mathbb{E}}\left\lbrack \left| Z_{3} \right| \right\rbrack,\]
which is at most \(c(r)\, n\) by Lemma~\ref{lem:8-2}(iii).

\emph{(ii).} The sequence \(\left( Z_{k} \right)_{k \geq 3}\) is stationary
under \(\overline{\mathbb{P}}\) and, being a two-block factor of an i.i.d.\
sequence (Corollary~\ref{cor:5-8}), ergodic; it is integrable by
Lemma~\ref{lem:8-2}(iii). Birkhoff's ergodic theorem \cite{ref3} gives
\(n^{- 1}\sum_{k \leq n}Z_{k} \rightarrow \gamma(r)\) almost surely; the
convergence also holds in \(L^{1}\), because a family of identically
distributed integrable random variables is uniformly integrable, hence so is
the family of its Ces\`aro averages, and almost sure convergence together with
uniform integrability gives \(L^{1}\) convergence (see \cite[Chapter~1]{Kre85},
where the ergodic theorem is stated in both forms). By
Lemma~\ref{lem:8-2}(ii) the two delayed terms
\(\beta_{1}\left( \mathcal{U}_{1}(r) \right) + Z_{2}\) are almost surely finite
and do not depend on \(n\), so they divide out in both modes of convergence.

\emph{(iii).} We identify the density at the level of a single cycle, which
avoids passing to the limit under the radius integral. Measurability of
\(r \mapsto \gamma(r)\) follows from Fubini's theorem applied to the
product-measurable map \((\omega,r) \mapsto Z_{3}(r)\). By
\eqref{eq:8-integrated},
\(\overline{\mathbb{E}}\bigl\lbrack \int_{r_{0}}^{r_{1}}\left| Z_{3}(r) \right|dr \bigr\rbrack < \infty\),
so \(\gamma\) is integrable and Fubini's theorem applies to
\(Z_{3} = \int_{r_{0}}^{r_{1}}Z_{3}(r)\psi(r)dr\) --- the representation of
Lemma~\ref{lem:6-9} --- for every bounded Borel \(\psi\) on
\(\left\lbrack r_{0},r_{1} \right\rbrack\), yielding
\[\gamma_{\psi}\  = \ \overline{\mathbb{E}}\left\lbrack Z_{3} \right\rbrack\  = \ \int_{r_{0}}^{r_{1}}\psi(r)\,\overline{\mathbb{E}}\left\lbrack Z_{3}(r) \right\rbrack dr\  = \ \int_{r_{0}}^{r_{1}}\psi(r)\gamma(r)\, dr .\]
Dividing by \(\bar{\mu}\) and comparing with Corollary~\ref{cor:6-15}, which
gives \(\rho_{\psi} = \int\psi\, d\lambda\), we obtain
\(\int\psi\, d\lambda = \int\psi(r)\rho(r)dr\). Since this holds for every
\(\psi \in C\left\lbrack r_{0},r_{1} \right\rbrack\) --- test functions on the
closed interval, not merely those vanishing at the endpoints --- the two
finite measures agree, and in particular \(\lambda\) charges neither endpoint.
\end{proof}

\begin{remark}[what part (iii) avoids]\label{rem:8-3a}
One might try to obtain (iii) by dominated convergence inside
\(T^{- 1}\Phi_{\psi}(T) = \int T^{- 1}\beta_{1}(T,r)\psi(r)dr\). That does not
close: the estimates of Sections~\ref{sec:4}--\ref{sec:6} control the radius
integral for each random compact set, but produce no function of \(r\) alone
dominating \(T^{- 1}\beta_{1}(T,r)\) uniformly in \(T\). The cycle-level
identification needs no such majorant, precisely because \eqref{eq:8-Zr}
involves only offsets of fixed compact sets.
\end{remark}

Passing from the regeneration times to deterministic times requires, at a
single radius, an input that the deterministic theory of Section~\ref{sec:4}
does not supply. For \(k \geq 1\) and
\(t \in \left\lbrack \tau_{k - 1},\tau_{k} \right\rbrack\) write
\(K_{k,t}:= X\left( \left\lbrack \tau_{k - 1},t \right\rbrack \right)\) and,
with \(S_{0}(r):= \varnothing\) and \(\mathcal{U}_{0}:= \varnothing\) as in
Section~\ref{sec:6-1}, put
\begin{equation}\label{eq:8-osc}
\widehat{M}_{k}(r):= \sup_{\tau_{k - 1} \leq t \leq \tau_{k}}\left\lbrack \beta_{1}\left( K_{k,t}^{(r)} \right) + \beta_{1}\left( S_{k - 1}(r) \cap K_{k,t}^{(r)} \right) + \beta_{0}\left( S_{k - 1}(r) \cap K_{k,t}^{(r)} \right) \right\rbrack \in \lbrack 0,\infty\rbrack .
\end{equation}
For \(k = 1\) the intersection terms vanish and
\(\widehat{M}_{1}(r) = \sup_{0 \leq t \leq \tau_{1}}\beta_{1}(t,r)\). By
construction \(\widehat{M}_{k}(r)\) is a functional of the pair
\(\left( K_{k - 1},K_{k} \right)\) alone, defined with no appeal to the Euler
identity of Lemma~\ref{lem:6-9} --- an almost-everywhere-in-radius statement,
which could not be applied simultaneously at all intermediate times \(t\) at
a prescribed radius --- and its finiteness forces
\(\beta_{1}\left( K_{k,t}^{(r)} \right) < \infty\) for every \(t\), which is
the hypothesis \eqref{eq:6-2-fin} that Lemma~\ref{lem:6-2} needs at each
intermediate time.

\begin{proposition}[deterministic-time interpolation, conditional]\label{prop:8-4}
Fix \(r \notin N\) and assume
\begin{equation}\label{eq:8-Or}
\mathbb{E}\left\lbrack \widehat{M}_{1}(r) \right\rbrack < \infty,\qquad
\mathbb{E}\left\lbrack \widehat{M}_{2}(r) \right\rbrack < \infty,\qquad
\overline{\mathbb{E}}\left\lbrack \widehat{M}_{3}(r) \right\rbrack < \infty .
\end{equation}
Then \(\mathbb{E}\left\lbrack \beta_{1}(T,r) \right\rbrack \leq c(r)(1 + T)\)
for every \(T \geq 0\), and
\[\frac{\beta_{1}(T,r)}{T}\  \longrightarrow \ \rho(r)\quad\quad\text{almost surely and in }L^{1}\text{ as }T \rightarrow \infty .\]
\end{proposition}

The three conditions in \eqref{eq:8-Or} control, respectively, the delayed
initial cycle, the pair \(\left( K_{1},K_{2} \right)\), whose law is not the
stationary product law, and the stationary regime; we call \eqref{eq:8-Or}
the oscillation hypothesis at radius \(r\). The supremum in \eqref{eq:8-osc}
runs over an uncountable family of times, which Lemma~\ref{lem:4-meas} does not
cover, and \eqref{eq:8-Or} is understood to include the measurability of
\(\widehat{M}_{k}(r)\).

\begin{proof}
For \(k \geq 3\) the pair \(\left( K_{k - 1},K_{k} \right)\) has law
\({\overline{\mathbb{P}}}^{\otimes 2}\) by Corollary~\ref{cor:5-8}, so the third
condition in \eqref{eq:8-Or} gives
\(\overline{\mathbb{E}}\bigl\lbrack \widehat{M}_{k}(r) \bigr\rbrack = \overline{\mathbb{E}}\bigl\lbrack \widehat{M}_{3}(r) \bigr\rbrack < \infty\)
for every \(k \geq 3\), and \(\bigl( \widehat{M}_{k}(r) \bigr)_{k \geq 3}\) is
stationary and \(1\)-dependent. Put
\[V_{1}:= \widehat{M}_{1}(r),\qquad V_{k}:= \left| Z_{k}(r) \right| + \widehat{M}_{k}(r)\quad (k \geq 2),\]
so that \(v:= \overline{\mathbb{E}}\left\lbrack V_{3} \right\rbrack < \infty\)
by Lemma~\ref{lem:8-2}(iii) and \eqref{eq:8-Or}, and \(V_{1},V_{2}\) are
integrable by the first two conditions of \eqref{eq:8-Or} together with
Lemma~\ref{lem:8-2}(iii).

\emph{A pathwise majorant.} We claim
\begin{equation}\label{eq:8-pathwise}
\beta_{1}(T,r)\  \leq \ \sum_{k = 1}^{N_{T} + 1}V_{k}\qquad\text{for every }T \geq 0 .
\end{equation}
On \(\left\{ N_{T} = 0 \right\} = \left\{ T < \tau_{1} \right\}\) the right-hand
side is \(V_{1} = \widehat{M}_{1}(r) \geq \beta_{1}(T,r)\) by definition.
(Bounding \(\beta_{1}(T,r)\) by \(\beta_{1}\left( \mathcal{U}_{1}(r) \right)\)
instead would be false: \(\beta_{1}\) is not monotone under inclusion, since
enlarging a path can fill holes as well as create them.) On
\(\left\{ N_{T} = n \right\}\) with \(n \geq 1\), Lemma~\ref{lem:6-1} gives
\(\mathcal{U}_{n}(r) \cap K_{n + 1,T}^{(r)} = S_{n}(r) \cap K_{n + 1,T}^{(r)}\),
and \eqref{eq:8-Or} makes \(\beta_{1}\bigl( K_{n + 1,T}^{(r)} \bigr)\) almost
surely finite, so Lemma~\ref{lem:6-2} applies to the pair
\(\mathcal{U}_{n}(r),K_{n + 1,T}^{(r)}\) and yields
\(\left| \beta_{1}(T,r) - \beta_{1}\left( \tau_{n},r \right) \right| \leq \widehat{M}_{n + 1}(r)\);
combining with
\(\beta_{1}\left( \tau_{n},r \right) \leq \beta_{1}\left( \mathcal{U}_{1}(r) \right) + \sum_{k = 2}^{n}\left| Z_{k}(r) \right|\)
from Lemma~\ref{lem:8-2}(ii), and with
\(\beta_{1}\left( \mathcal{U}_{1}(r) \right) \leq \widehat{M}_{1}(r) = V_{1}\),
gives \eqref{eq:8-pathwise}.

\emph{The first-moment bound.} Fix \(n \geq 3\). Since \(V_{n}\) is a
functional of \(\left( K_{n - 1},K_{n} \right)\), it is not independent of
\(\left\{ \tau_{n - 1} \leq T \right\}\), which involves the duration of
\(K_{n - 1}\); we therefore enlarge the event by one block before invoking
independence: since \(\left\{ N_{T} + 1 \geq n \right\} = \left\{ \tau_{n - 1} \leq T \right\}\),
\[\mathbb{E}\left\lbrack V_{n};N_{T} + 1 \geq n \right\rbrack\  \leq \ \mathbb{E}\left\lbrack V_{n};\tau_{n - 2} \leq T \right\rbrack\  = \ v\,\mathbb{P}\left( \tau_{n - 2} \leq T \right)\  = \ v\,\mathbb{P}\left( N_{T} \geq n - 2 \right);\]
the inequality holds because \(\tau_{n - 2} \leq \tau_{n - 1}\), and the
equality after it because
\(\left\{ \tau_{n - 2} \leq T \right\} \in \sigma\left( K_{1},\ldots,K_{n - 2} \right)\)
is independent of \(\left( K_{n - 1},K_{n} \right)\) by
Corollary~\ref{cor:5-8}. Summing over \(n \geq 3\) and adding the two delayed
terms, which are integrable by \eqref{eq:8-Or},
\begin{equation}\label{eq:8-wald}
\mathbb{E}\left\lbrack \beta_{1}(T,r) \right\rbrack\  \leq \ \mathbb{E}\left\lbrack V_{1} + V_{2} \right\rbrack + v\left( \mathbb{E}\left\lbrack N_{T} \right\rbrack + 1 \right)\  \leq \ c(r)(1 + T),
\end{equation}
using \(\mathbb{E}\left\lbrack N_{T} \right\rbrack \leq C(1 + T)\) from
Proposition~\ref{prop:5-16}(iii). The constant \(c(r)\) does not depend on
\(T\), so the bound holds simultaneously for every \(T\); applying Tonelli
separately at each fixed \(T\) would instead produce a null set of radii
depending on \(T\), and a union over uncountably many \(T\) need not be null.

\emph{Almost sure convergence.} Put \(W_{n}:= \sum_{k \leq n}V_{k}\). By the
ergodic theorem \(W_{n}/n \rightarrow v\) almost surely, and
\(N_{T}/T \rightarrow 1/\bar{\mu}\) with \(N_{T} \rightarrow \infty\) by
Proposition~\ref{prop:5-16}(i); hence, almost surely,
\[\frac{V_{N_{T} + 1}}{T} = \frac{W_{N_{T} + 1} - W_{N_{T}}}{T} \rightarrow 0
\qquad\text{and}\qquad
\frac{\beta_{1}\left( \tau_{N_{T}},r \right)}{T} = \frac{\beta_{1}\left( \tau_{N_{T}},r \right)}{N_{T}}\cdot\frac{N_{T}}{T} \rightarrow \frac{\gamma(r)}{\bar{\mu}},\]
the latter by Theorem~\ref{thm:8-3}(ii). Since
\(\left| \beta_{1}(T,r) - \beta_{1}\left( \tau_{N_{T}},r \right) \right| \leq \widehat{M}_{N_{T} + 1}(r) \leq V_{N_{T} + 1}\)
on \(\left\{ N_{T} \geq 1 \right\}\), an event of probability tending to one,
\(\beta_{1}(T,r)/T \rightarrow \rho(r)\) almost surely.

\emph{\(L^{1}\) convergence.} It suffices to show that the dominating family
\(\left\{ T^{- 1}W_{N_{T} + 1} \right\}_{T \geq 1}\) of \eqref{eq:8-pathwise}
is uniformly integrable: being also almost surely convergent, it then
converges in \(L^{1}\), and the generalised dominated convergence theorem
upgrades the almost sure convergence of \(\beta_{1}(T,r)/T\) to \(L^{1}\).
Repeating \eqref{eq:8-wald} with \(V_{n}\mathbf{1}_{\left\{ V_{n} > b \right\}}\)
in place of \(V_{n}\) gives
\[\limsup_{T \rightarrow \infty}\frac{1}{T}\,\mathbb{E}\left\lbrack \sum_{k \leq N_{T} + 1}V_{k}\mathbf{1}_{\left\{ V_{k} > b \right\}} \right\rbrack\  \leq \ \frac{C}{\bar{\mu}}\,\overline{\mathbb{E}}\left\lbrack V_{3}\mathbf{1}_{\left\{ V_{3} > b \right\}} \right\rbrack\  \xrightarrow[\ b \rightarrow \infty\ ]{}\ 0,\]
which controls large individual rewards. A large value of the normalised sum
can also arise from an unusually large number of cycles; splitting each reward
at level \(b\),
\[\frac{W_{N_{T} + 1}}{T}\  \leq \ \frac{1}{T}\sum_{k \leq N_{T} + 1}V_{k}\mathbf{1}_{\left\{ V_{k} > b \right\}}\ +\ b\,\frac{N_{T} + 1}{T},\]
the first term is uniformly small in \(L^{1}\) for large \(b\) by the display
above, while for each fixed \(b\) the family
\(\left\{ \left( N_{T} + 1 \right)/T \right\}_{T \geq 1}\) is bounded in every
\(L^{q}\) (Proposition~\ref{prop:5-16}(iii)), hence uniformly integrable.
Explicitly, given \(\varepsilon > 0\) choose \(b\) so that the first term has
expectation at most \(\varepsilon\) uniformly in \(T \geq 1\); then for
\(\lambda > 0\), by Cauchy--Schwarz,
\[\mathbb{E}\left\lbrack \frac{W_{N_{T} + 1}}{T};\ \frac{W_{N_{T} + 1}}{T} > \lambda \right\rbrack\  \leq \ \varepsilon\ +\ b\left\| \frac{N_{T} + 1}{T} \right\|_{2}\,\mathbb{P}\left( \frac{W_{N_{T} + 1}}{T} > \lambda \right)^{1/2},\]
and the probability is at most
\(\lambda^{- 1}\sup_{T \geq 1}\mathbb{E}\left\lbrack W_{N_{T} + 1}/T \right\rbrack \leq C/\lambda\)
by \eqref{eq:8-wald}. Letting \(\lambda \rightarrow \infty\) and then
\(\varepsilon \downarrow 0\) gives uniform integrability. Only first moments at
radius \(r\) are used, and no random-index limit theorem is invoked.
\end{proof}

\begin{remark}[why \eqref{eq:8-Or} is not automatic, and is not assumed elsewhere]\label{rem:8-5}
The oscillation hypothesis \eqref{eq:8-Or} does not follow from
Theorem~\ref{thm:4-11}, and we do not claim it. Theorem~\ref{thm:4-11} bounds
\(\int_{r_{0}}^{r_{1}}\beta_{1}\left( C^{(r)} \right)dr\) for each fixed
compact connected \(C\); applied to the family
\(\left\{ K_{k - 1} \cup K_{k,t} \right\}_{t}\), whose diameters are bounded by
\(D_{k - 1} + D_{k}\) uniformly in \(t\), it yields
\[\sup_{t}\int_{r_{0}}^{r_{1}}f_{t}(r)\, dr\  < \ \infty,\qquad f_{t}(r):= \left| \beta_{1}\left( \mathcal{U}_{k - 1}(r) \cup K_{k,t}^{(r)} \right) - \beta_{1}\left( \mathcal{U}_{k - 1}(r) \right) \right|,\]
whereas \eqref{eq:8-Or} requires \(\sup_{t}f_{t}(r)\) itself to be integrable
in \(\omega\) at a prescribed \(r\), and
\(\sup_{t}\int \leq \int\sup_{t}\) goes the wrong way. This is precisely the
distinction that makes the smoothed theory of Sections~\ref{sec:6}--\ref{sec:7}
unconditional --- there the observable is a radius integral, so
Proposition~\ref{prop:6-10} needs only \(\sup_{t}\int\) --- while the
fixed-radius theory at deterministic times is not.

Nothing in Sections~\ref{sec:6} and \ref{sec:7}, nor Theorem~\ref{thm:8-1} or
Theorem~\ref{thm:8-3}, depends on \eqref{eq:8-Or}: it is used only in
Proposition~\ref{prop:8-4}. Obtaining it appears to require control of
near-tangent configurations along a block, the mechanism discussed in
Section~\ref{sec:8-2}; we leave it open.
\end{remark}

\subsection{The fixed-radius problem, and what is missing}\label{sec:8-2}

Theorem~\ref{thm:8-3} leaves open whether its conclusions hold at every
\(r > 0\) rather than at almost every \(r\). This subsection isolates the
estimate that would settle this, proves the elementary parts of a natural
route to it, and records where that route stops; the near-tangency mechanism
discussed at the end is a heuristic, not a partial proof. By
\cite[Theorem 4.1]{RSS09} and the Cameron--Martin theorem, for every fixed
\(r > 0\) the radius \(r\) is almost surely not a critical value of the
distance function of the drifted planar Brownian path, so no regularity
assumption on \(r\) is involved.

\emph{What removes the exceptional set.} The radius enters the proof of
Theorem~\ref{thm:8-3} only through Lemma~\ref{lem:8-2}, and enters
Lemma~\ref{lem:8-2} only through the finiteness of the expected hole counts of
one and two consecutive block sausages.

\begin{hypothesis}[finite expected hole counts at radius \(r\)]\label{hyp:L}
For \(k = 1,2\) (the delayed blocks, with \(S_{0}(r):=\varnothing\)), and for
\(k = 3\) under the stationary law \(\overline{\mathbb{P}}\) --- hence for
every \(k \geq 3\) by Corollary~\ref{cor:5-8} ---
\[\mathbb{E}\left\lbrack \beta_{1}\left( S_{k}(r) \right) \right\rbrack < \infty
\qquad\text{and}\qquad
\mathbb{E}\left\lbrack \beta_{1}\left( S_{k - 1}(r) \cup S_{k}(r) \right) \right\rbrack < \infty .\]
\end{hypothesis}

\begin{proposition}[the law of large numbers at a prescribed radius]\label{prop:8-H}
If Hypothesis~\ref{hyp:L} holds at some \(r \in \left\lbrack r_{0},r_{1} \right\rbrack\),
then the conclusions (i)--(ii) of Theorem~\ref{thm:8-3} hold at that \(r\):
\(\left( Z_{k}(r) \right)_{k \geq 3}\) is stationary and \(1\)-dependent,
\(\gamma(r) = \overline{\mathbb{E}}\left\lbrack Z_{3}(r) \right\rbrack\) is
finite, and \(\beta_{1}\left( \tau_{n},r \right)/n \rightarrow \gamma(r)\)
almost surely and in \(L^{1}\).
\end{proposition}

\begin{proof}
The block sausages are connected, so Lemma~\ref{lem:6-7} gives
\[W_{k}(r)\  \leq \ 2\beta_{1}\left( S_{k}(r) \right) + \beta_{1}\left( S_{k - 1}(r) \right) + \beta_{1}\left( S_{k - 1}(r) \cup S_{k}(r) \right) + 1,\]
and Hypothesis~\ref{hyp:L} gives \(\mathbb{E}\left\lbrack W_{k}(r) \right\rbrack < \infty\)
for \(k = 1,2\) and \(\overline{\mathbb{E}}\left\lbrack W_{3}(r) \right\rbrack < \infty\).
In particular all Betti numbers in \eqref{eq:8-Zr} and in \(W_{k}(r)\) are
almost surely finite at \(r\), and so are those of the unions
\(\mathcal{U}_{k}(r)\), by induction from the one-sided bound
\[\beta_{1}\left( \mathcal{U}_{k}(r) \right)\  \leq \ \beta_{1}\left( \mathcal{U}_{k - 1}(r) \right) + \beta_{1}\left( S_{k}(r) \right) + \beta_{0}\left( S_{k - 1}(r) \cap S_{k}(r) \right)\]
of Lemma~\ref{lem:6-2} with Lemma~\ref{lem:6-1}. These are exactly the
assertions of Lemma~\ref{lem:8-2}(i)--(iii) for \(r \notin N\), and the proof
of Theorem~\ref{thm:8-3}(i)--(ii) used nothing else. \end{proof}

Lemma~\ref{lem:8-2} shows that Hypothesis~\ref{hyp:L} holds for almost every
\(r\); Theorem~\ref{thm:4-11} cannot supply it at a prescribed \(r\)
(Remark~\ref{rem:4-7}). Hypothesis~\ref{hyp:L} is the analogue, for the block
sausages of Section~\ref{sec:5} under the drift, of the finiteness of the
expected hole count at a fixed radius and a fixed time established for the
driftless sausage in \cite[Proposition 4.1]{RSM09} by the critical-value
calculus underlying Theorem~\ref{thm:4-11}; carrying that calculus over to
random-time blocks under the drift is not attempted here. Part (iii) of
Theorem~\ref{thm:8-3}, the identification of the density of \(\lambda\), is an
almost-everywhere statement by nature and is not affected.

\emph{A route through chopping.} Chop the path at the successive exit times
\(\sigma_{i}\) of balls of radius \(\rho = \frac{\sqrt{3}}{2}r\) centred at the
current position, \(\sigma_{0} = 0\). The \(\sigma_{i}\) are stopping times, so
by the strong Markov property the durations \(\sigma_{i} - \sigma_{i - 1}\) are
i.i.d., with exponential moments, and the number
\(n_{T}:=\#\left\{ i \geq 1:\sigma_{i - 1} \leq T \right\}\) of pieces started
by time \(T\) satisfies \(\mathbb{E}\left\lbrack n_{T} \right\rbrack \leq C(1 + T)\)
by the elementary renewal theorem \cite{ref1}. For \(i \leq n_{T}\) let
\[K_{i}^{T}:=X\left( \left\lbrack \sigma_{i - 1},\ \sigma_{i} \land T \right\rbrack \right),\qquad
A_{i}^{T}:=\left( K_{i}^{T} \right)^{(r)},\]
the pieces truncated at \(T\) (the last one generally is), so that
\(K_{T}^{(r)} = \bigcup_{i \leq n_{T}}A_{i}^{T}\); each \(K_{i}^{T}\) has
diameter at most \(2\rho = \sqrt{3}\, r\), so each \(A_{i}^{T}\) is star-shaped
by Lemma~\ref{lem:8-4}, and
\(A_{i}^{T} \subseteq B\left( X_{\sigma_{i - 1}},\rho + r \right)\).

\begin{proposition}[chopping bound]\label{prop:8-5}
Pathwise, \(\beta_{1}(T,r) \leq P_{T} + F_{T}\), where
\[P_{T}:=\#\left\{ (i,j):\ i < j \leq n_{T},\ A_{i}^{T} \cap A_{j}^{T} \neq \varnothing \right\},\qquad
F_{T}:=\sum_{i < j \leq n_{T}}\left( \beta_{0}\left( A_{i}^{T} \cap A_{j}^{T} \right) - 1 \right)^{+}.\]
\end{proposition}

\begin{proof}
Mayer--Vietoris in degree one along the chain
\(U_{i} = A_{1}^{T} \cup \ldots \cup A_{i}^{T}\), \(i \leq n_{T}\), whose last
member is \(K_{T}^{(r)}\), gives
\[\beta_{1}\left( U_{i} \right) \leq \beta_{1}\left( U_{i - 1} \right) + \beta_{1}\left( A_{i}^{T} \right) + \beta_{0}\left( A_{i}^{T} \cap U_{i - 1} \right);\]
the middle term vanishes by Lemma~\ref{lem:8-4}, and
\(\beta_{0}\left( A_{i}^{T} \cap U_{i - 1} \right) \leq \sum_{j < i}^{}\beta_{0}\left( A_{i}^{T} \cap A_{j}^{T} \right)\),
each nonempty intersection \(A_{i}^{T} \cap A_{j}^{T}\) contributing at most
\(\beta_{0}\left( A_{i}^{T} \cap A_{j}^{T} \right) = 1 + \left( \beta_{0}\left( A_{i}^{T} \cap A_{j}^{T} \right) - 1 \right)^{+}\).
\end{proof}

\begin{lemma}[intersecting pairs]\label{lem:8-pairs}
\(\mathbb{E}\left\lbrack P_{T} \right\rbrack \leq C(r,\nu)(1 + T)\) for every
\(T \geq 0\).
\end{lemma}

\begin{proof}
By the strong Markov property at the stopping times \(\sigma_{i}\), the
displacements \(\xi_{i}:=X_{\sigma_{i}} - X_{\sigma_{i - 1}}\) are i.i.d.\ with
\(\lVert \xi_{i} \rVert = \rho\), and
\(\mathbb{E}\left\lbrack \xi_{1} \right\rbrack = \mu\,\mathbb{E}\left\lbrack \sigma_{1} \right\rbrack\)
by optional stopping applied to \(B\) at \(\sigma_{1}\), legitimate since
\(\sup_{t \leq \sigma_{1}}\lVert B_{t} \rVert \leq \rho + \nu\sigma_{1}\) is
integrable. Put \(S_{m}:=\xi_{1} + \cdots + \xi_{m}\). Since
\(A_{i}^{T} \subseteq B\left( X_{\sigma_{i - 1}},\rho + r \right)\), a pair
\(i < j \leq n_{T}\) contributes to \(P_{T}\) only if \(\sigma_{i - 1} \leq T\)
and \(\lVert X_{\sigma_{j - 1}} - X_{\sigma_{i - 1}} \rVert \leq 2(\rho + r)\),
and the latter displacement is a copy of \(S_{j - i}\) independent of
\(\mathcal{F}_{\sigma_{i - 1}} \ni \left\{ \sigma_{i - 1} \leq T \right\}\).
Hence
\[\mathbb{E}\left\lbrack P_{T} \right\rbrack\  \leq \ \sum_{i \geq 1}\mathbb{P}\left( \sigma_{i - 1} \leq T \right)\ \sum_{m \geq 1}\mathbb{P}\left( \lVert S_{m} \rVert \leq 2(\rho + r) \right)\  = \ \mathbb{E}\left\lbrack n_{T} \right\rbrack\sum_{m \geq 1}\mathbb{P}\left( \lVert S_{m} \rVert \leq 2(\rho + r) \right).\]
The first factor is at most \(C(1 + T)\). For the second, \(\langle S_{m},e\rangle\)
is a sum of \(m\) i.i.d.\ variables with values in \(\lbrack - \rho,\rho\rbrack\)
and mean \(a:=\nu\,\mathbb{E}\left\lbrack \sigma_{1} \right\rbrack > 0\), so
Hoeffding's inequality gives, for \(ma > c:=2(\rho + r)\),
\(\mathbb{P}\left( \lVert S_{m} \rVert \leq c \right) \leq \mathbb{P}\left( \langle S_{m},e\rangle \leq c \right) \leq \exp\left( - (ma - c)^{2}/(2m\rho^{2}) \right)\),
which is summable in \(m\); the sum depends only on \(r\) and \(\nu\).
\end{proof}

The pair count thus uses the drift in an essential way, through the
transience of the walk of piece anchors; the chopping inequality and the
next lemma are drift-free.

\begin{lemma}[near-tangency is necessary for fragmentation]\label{lem:8-6}
Let \(c_{i}\), \(c_{j}\) be the centres of the smallest discs containing
\(K_{i}^{T}\), \(K_{j}^{T}\). If the cores
\[B\left( c_{i},\ r - \operatorname{diam}K_{i}^{T}/\sqrt{3} \right)\qquad\text{and}\qquad B\left( c_{j},\ r - \operatorname{diam}K_{j}^{T}/\sqrt{3} \right)\]
meet, then \(A_{i}^{T} \cap A_{j}^{T}\) is star-shaped, hence connected. Moreover, at any
\(z \in \partial A_{i}^{T} \cap \partial A_{j}^{T}\) with contact points
\(y \in K_{i}^{T}\), \(w \in K_{j}^{T}\), put
\(\delta_{z}:=2r - \lVert y - w \rVert\), so that \(0 \leq \delta_{z} \leq 2r\),
and let \(\alpha_{z}:=\pi - \angle(y,z,w)\) be the deviation from tangency of
the two boundary circles at \(z\). Then, with \(x:=\delta_{z}/r\),
\[\cos\angle(y,z,w) = - 1 + 2x - \frac{x^{2}}{2},\qquad
\alpha_{z} = 2\arccos\left( 1 - \frac{x}{2} \right),\qquad
2\sqrt{\delta_{z}/r}\  \leq \ \alpha_{z}\  \leq \ \pi\sqrt{\delta_{z}/r}.\]
\end{lemma}

\begin{proof}
The first claim: by Lemma~\ref{lem:8-4} both sausages are star-shaped about
any common core point \(p\), and \(p\) lies in both, so their intersection is
star-shaped about \(p\). For the second, the law of cosines in the triangle
\(y,z,w\) with \(\lVert z - y \rVert = \lVert z - w \rVert = r\) and
\(\lVert y - w \rVert = (2 - x)r\) gives
\(\cos\angle(y,z,w) = 1 - (2 - x)^{2}/2 = - 1 + 2x - x^{2}/2\); hence
\(\cos\alpha_{z} = 1 - 2x + x^{2}/2 = 2\left( 1 - x/2 \right)^{2} - 1\), which
is the half-angle formula for \(\alpha_{z} = 2\arccos(1 - x/2)\). The two-sided
bound follows from \(2t^{2}/\pi^{2} \leq 1 - \cos t \leq t^{2}/2\) on
\(\lbrack 0,\pi\rbrack\), applied with \(1 - \cos t = x/2\). \end{proof}

By Proposition~\ref{prop:8-5} and Lemma~\ref{lem:8-pairs}, a bound
\(\mathbb{E}\left\lbrack F_{T} \right\rbrack \leq C(r,\nu)(1 + T)\) on the
fragmentation term would give
\(\mathbb{E}\left\lbrack \beta_{1}(T,r) \right\rbrack \leq C(r,\nu)(1 + T)\) at
every deterministic \(T\) and every \(r > 0\). By Lemma~\ref{lem:8-6} only
pairs near tangency can contribute to \(F_{T}\), and fragmentation inside such
a pair is produced by the corners of \(\partial A_{i}^{T} \cup \partial A_{j}^{T}\):
a corner of turning \(\gamma\) can reverse a crossing whose deviation from
tangency is \(\alpha_{z}\) only if \(\gamma \geq \alpha_{z} \geq 2\sqrt{\delta_{z}/r}\),
that is, only if the pair is within \(\delta_{z} \leq \gamma^{2}r/4\) of
tangency; so such a bound would follow
from a polynomial bound on the number of corners in near-tangent overlaps
together with a rarity estimate of the form
\[\mathbb{E}\,\#\left\{ i < j \leq n_{T}:\ \operatorname{dist}\left( K_{i}^{T},K_{j}^{T} \right) \in \lbrack 2r - \delta,2r\rbrack \right\}\  \leq \ C(r,\nu)\,\delta\,(1 + T),\qquad \delta \in (0,r\rbrack .\]
We have neither. The natural route to the rarity estimate, through the coarea
formula, bounds the area of the set of displacements putting a pair within
\(\delta\) of tangency, but converting that occupation bound into a count
requires fattening the window by the piece radius \(\rho \asymp r\), and the
factor \(\delta\), which is the entire content of the estimate, does not
survive; recovering it appears to need a bounded-density estimate for the
renewal measure of the walk of piece anchors, whose steps are supported on a
circle. And a deterministic-time bound would still have to be converted into
the block-level statement of Hypothesis~\ref{hyp:L}, a further step. We leave
all of this open.

One quantity is available for the sausage without any hypothesis, the
perimeter of the offset, and we formulate a conjecture in terms of it.

\begin{proposition}[expected perimeter of the sausage]\label{prop:8-10}
For every \(r > 0\) and every \(T \geq 0\),
\(\mathbb{E}\left\lbrack \mathcal{H}^{1}\left( \partial K_{T}^{(r)} \right) \right\rbrack < \infty\).
\end{proposition}

\begin{proof}
The hypotheses of Proposition~\ref{prop:4-perim} --- that \(r\) is not a
critical value of \(d_{K_{T}}\) and that \(s \mapsto \left| K_{T}^{(s)} \right|\)
is differentiable at \(r\) --- hold almost surely at every fixed \(r\) by
\cite[Theorem 4.1]{RSS09} and the Cameron--Martin theorem, so
\(\mathcal{H}^{1}\left( \partial K_{T}^{(r)} \right) \leq 2\left| K_{T}^{(r)} \right|/r\)
almost surely, and \(\mathbb{E}\left| K_{T}^{(r)} \right| < \infty\) trivially.
\end{proof}

\begin{conjecture}\label{conj:8-13}
\[\mathbb{E}\left\lbrack \beta_{1}(T,r) \right\rbrack\  \asymp \ \frac{\mathbb{E}\left\lbrack \mathcal{H}^{1}\left( \partial K_{T}^{(r)} \right) \right\rbrack}{2\pi r}\]
as \(T \rightarrow \infty\), for each fixed \(r > 0\). The right-hand side is
finite for every \(r > 0\) by Proposition~\ref{prop:8-10}; the assertion of the
conjecture is precisely that the left-hand side is finite as well and is
comparable to it. For the drifted sausage this predicts
\(\mathbb{E}\left\lbrack \beta_{1}(T,r) \right\rbrack \asymp c\,\nu T/r\), the
expected perimeter of the sausage growing linearly in \(T\).
\end{conjecture}

\section{Concluding remarks}\label{sec:9}

We treat planar
Brownian motion with nonzero drift and establish a law of large numbers and a
central limit theorem for smoothed Betti-curve functionals on compact radius
windows bounded away from \(0\). These assumptions are intrinsic to the method.
The proof depends on the regeneration structure created by the drift, which
reduces the problem to a local boundary analysis near successive cut times; and
the restriction to a radius window away from \(0\) is what makes the integrated
complexity bound of Theorem~\ref{thm:4-11} available, Remark~\ref{rem:4-7}
recording that no pointwise-in-radius substitute exists.

Two questions are left open by the method and are worth stating precisely.

The first concerns the fixed-radius theory, and it has two logically
independent halves. Theorem~\ref{thm:8-3} gives the law of large numbers at
the regeneration times for almost every radius, and Proposition~\ref{prop:8-H}
gives it at any prescribed radius at which the block sausages have finite
expected hole counts (Hypothesis~\ref{hyp:L}), which Lemma~\ref{lem:8-2}
supplies almost everywhere and which we cannot supply at a prescribed radius;
Section~\ref{sec:8-2} reduces a deterministic-time version of it to the
fragmentation of near-tangent pairs of sausage pieces, and neither that
fragmentation bound nor its conversion to the block level is proved. Passing
from regeneration times to deterministic times is a separate matter, governed
by the oscillation hypothesis \eqref{eq:8-Or} of Proposition~\ref{prop:8-4},
which Hypothesis~\ref{hyp:L} does not supply and which, by
Remark~\ref{rem:8-5}, the deterministic estimates of Section~\ref{sec:4} do
not supply either. Either half could be settled without the other. We
conjecture (Conjecture~\ref{conj:8-13}) that
\[\mathbb{E}\left\lbrack \beta_{1}(T,r) \right\rbrack\  \asymp \ \frac{\mathbb{E}\left\lbrack \mathcal{H}^{1}\left( \partial K_{T}^{(r)} \right) \right\rbrack}{2\pi r},\]
which for the drifted sausage predicts linear growth in \(T\) with a constant
of order \(\nu/r\).

The second is the driftless case of the limit theorems themselves, where the
regeneration mechanism fails for the reason given above: the one-dimensional
projection is recurrent, and the good-cut construction of Section~\ref{sec:5}
collapses. Any corresponding large-time theorem for zero-drift Brownian motion
therefore requires a genuinely different method, based on scaling rather than
renewal; this is the subject of \cite{G26c}, where Brownian
scaling replaces the renewal structure and the growth of the expected hole count
is logarithmic rather than linear. The contrast is instructive: the linear rate
\(\rho_{\psi}T\) obtained here is a renewal rate, produced by the drift carrying
the path into fresh space at constant speed, and it has no driftless analogue.

Beyond these, three directions seem natural. First, the results here concern the
Betti-curve transform of the persistence diagram rather than the diagram measure
itself (Remark~\ref{rem:6-17}); a limit theorem for the diagram measure on a
compact birth--death rectangle would require test functions outside the class
\(\left\{ \varphi_{\psi} \right\}\), for which the Betti-curve representation,
and hence the present regeneration argument, is unavailable. Second, we do not
address strict positivity of the limiting variance \(\sigma_{\psi}^{2}\)
(Remark~\ref{rem:7-9}); the fluctuation theorem is complete without it, but a
non-degeneracy criterion would be of interest. Third, the higher-dimensional
drifted case, where degree-\(q\) persistence for \(1 \leq q \leq d - 1\) is
available, appears to admit the same regeneration scheme, but the deterministic
geometric input of Section~\ref{sec:4} and \cite{G26g} --- planar throughout, resting on
Alexander duality in \(S^{2}\) and on the half-dimensionality of the
critical-value set in the plane --- would have to be replaced.

The present work thus provides, to our knowledge, the first law of large
numbers and central limit theorem for smoothed Betti-curve functionals of a
Wiener sausage, and points toward a broader program beyond the regenerative
setting.
\section*{Declaration of competing interest}

The author declares that there are no known competing financial interests or
personal relationships that could have appeared to influence the work reported
in this paper.

\section*{Declaration of generative AI and AI-assisted technologies in the writing process}

During the preparation of this work the author used Anthropic's Claude in order
to convert and typeset the manuscript source, to check internal cross-references
and bibliographic details, and to assist with editorial revision and with the
drafting of expository and proof text over successive revisions. The results
reported here and the mathematical program they belong to are the author's own.
The author has reviewed, checked and edited all content, and takes full
responsibility for the content of the publication.

\end{document}